\documentclass[11.5pt,twoside, a4paper, english, reqno]{amsart} 
\usepackage{amsaddr}
\usepackage{a4wide}
\usepackage{amscd}
\usepackage{amssymb}
\usepackage{amsthm}
\usepackage{graphicx}
\usepackage{amsmath,calrsfs}
\usepackage{latexsym}
\usepackage{enumitem,dsfont}

\usepackage{upref}
\usepackage[colorlinks]{hyperref}
\usepackage{color,graphics}
\usepackage{upref,hyperref,color}

\usepackage[usenames,dvipsnames]{xcolor}
\usepackage{pdfsync}
\usepackage{ulem,cancel,pgf}

\usepackage{frcursive}

\theoremstyle{plain}
\newtheorem{theorem}{Theorem}[section]
\theoremstyle{plain}
\newtheorem{lemma}[theorem]{Lemma}
\newtheorem{proposition}[theorem]{Proposition}
\newtheorem{corollary}[theorem]{Corollary}

\theoremstyle{definition}
\newtheorem{definition}{Definition}[section]

\newtheorem{remark}{Remark}[section]
\newtheorem{claim}{Claim}[section]

\newtheorem*{maintheorem*}{Main Theorem}
\newtheorem*{maincorollary*}{Main Corollary}

\DeclareFontFamily{U}{BOONDOX-calo}{\skewchar\font=50 }
\DeclareFontShape{U}{BOONDOX-calo}{m}{n}{
	<-> s*[1.05] BOONDOX-r-calo}{}
\DeclareFontShape{U}{BOONDOX-calo}{b}{n}{
	<-> s*[1.05] BOONDOX-b-calo}{}
\DeclareMathAlphabet{\mathcalb}{U}{BOONDOX-calo}{m}{n}
\SetMathAlphabet{\mathcalb}{bold}{U}{BOONDOX-calo}{b}{n}
\DeclareMathAlphabet{\mathbcalb}{U}{BOONDOX-calo}{b}{n}

\numberwithin{equation}{section} \allowdisplaybreaks

\title[Evolution Equations with Mixed Local and Nonlocal Operators]{Local Existence, Uniqueness, Regularity, and Global Behavior of Evolution Equations Involving Mixed Local and Nonlocal Operators}
\date{\today}

\author[Abdelhamid Gouasmia]{Abdelhamid Gouasmia}
\address[Abdelhamid Gouasmia]{
Department of Mathematics, Faculty of Sciences And Technology,\\
Mohamed Cherif Messaadia University,\\
P.O.Box 1553, Souk Ahras 41000, Algeria.\\[4pt]
Laboratoire d'equations aux d\'{e}riv\'{e}es partielles non lin\'{e}aires et histoire des math\'{e}matiques,\\
Ecole Normale Sup\'{e}rieure,\\
B.P. 92, Vieux Kouba, 16050 Algiers, Algeria.}
\email[Abdelhamid Gouasmia]{gouasmia.abdelhamid@gmail.com, abdelhamid.gouasmia@univ-soukahras.dz}

\author[Hichem Hajaiej]{Hichem Hajaiej}
\address[Hichem Hajaiej]{Department of Mathematics, California State University, Los Angeles,
CA 90032, USA}
\email[Hichem Hajaiej]{hichem.hajaiej@gmail.com}

\author[Kaushik Bal]{Kaushik Bal}
\address[Kaushik Bal]{
	Indian Institute of Technology Kanpur, Kanpur, India.}
\email[Kaushik Bal]{kaushik@iitk.ac.in}
\usepackage{comment}
\begin{document}
	
\begin{abstract}
	In this work, we address a parabolic problem featuring a potentially doubly nonlinear term, governed by a combination of local and nonlocal operators (see Problem \eqref{P1} below). We first establish the local existence of weak energy solutions via a semi-discretization in time applied to an auxiliary evolution problem. The uniqueness of these solutions is subsequently obtained through a novel generalization of the classical inequality of D\'{i}az and Saa, suitably adapted to the mixed local–nonlocal setting. This generalization provides a new comparison principle and establishes the $T$-accretivity of a corresponding operator in $L^2$. By employing this comparison principle, we construct suitable barrier functions that allow the global-in-time extension of solutions. Furthermore, we demonstrate the convergence of weak solutions to a nontrivial stationary state. Our approach relies on methods from the theory of contraction semigroups. It is noteworthy that these results are underpinned by a detailed analysis of the stationary problems associated with Problem \eqref{P1}, which also reveals several qualitative properties of the solutions.
\end{abstract}

\maketitle
\\[2mm]
\hspace*{0.45cm}\textbf{MSC:} Primary 35K55, 35J62; Secondary 35B65, 35B40.  \\

\setcounter{tocdepth}{1}
\tableofcontents
\newpage
\section{Introduction}
Let $m \geq 0$, $p>1$, and $N \in \mathbb{N}\setminus\{0\}$ be such that $m<p-1$ and $N>p$. In the case $m>0$, assume that the parameters $\delta$ and $\gamma$ satisfy $ \delta \in (0,m) $ and $ 2\delta - \textbf{d}\,\gamma > 2m - 1 $ for some $\textbf{d}>N$.  We consider the space--time cylinder $Q_T := (0,T)\times\Omega$, where $\Omega \subset \mathbb{R}^N$ is a bounded domain with smooth boundary, and we denote its lateral boundary by $ \Gamma_T := (0,T)\times(\mathbb{R}^N \setminus \Omega). $ In this paper, we investigate the existence, uniqueness, and qualitative behavior of weak solutions to the following parabolic problem, which is associated with a doubly nonlinear evolution equation in the case $m>0$:
\begin{equation}\label{P1}\tag{$Q_{T}$}
\left\{
\begin{aligned}
\frac{m + 1}{2m + 1} \,\partial_t \big(u^{2m + 1}\big) - \Delta_{p} u + (-\Delta)^{s}_{p} u &= g(t, x)\, u^{m} + m \, d(x)^{- \gamma}  \, u^{\delta} && \text{in } Q_T, \\
u &> 0 && \text{in } Q_T, \\
u &= 0 && \text{in } \Gamma_T, \\
u(0,\cdot) &= u_0 && \text{in } \Omega.
\end{aligned}
\right.
\end{equation}
Here, \(d(\cdot)\) denotes the distance to the boundary \(\partial \Omega\), defined by
\[
d(x) := \operatorname{dist}(x, \partial \Omega) = \inf_{y \in \partial \Omega} |x - y|.
\] 
The classical \(p\)-Laplacian is given by $ \Delta_{p} u := \operatorname{div}\big( |\nabla u|^{p-2} \nabla u \big), $ whereas, for a fixed \(s \in (0,1)\), the fractional \(p\)-Laplacian is defined as
\[
(-\Delta)^{s}_{p} u(x) := 2 \, \mathrm{P.V.} \int_{\mathbb{R}^N} 
\frac{|u(x) - u(y)|^{p-2} \big(u(x) - u(y)\big)}{|x-y|^{N+sp}} \, dy,
\]
with \(\mathrm{P.V.}\) representing the Cauchy principal value.  We assume that \( g \in L^{\infty}(Q_T) \) and that there exists a nontrivial function \( \underline{g} \in \left( L^{\infty}(\Omega) \right)^{+} \) such that
\begin{equation}\label{equ1}\tag{H1}
g(t,x) \geq \underline{g}(x) \quad \text{for almost every } (t,x) \in Q_T.
\end{equation}
The study of mixed local and nonlocal operators of the form $ u \mapsto - \Delta_{p} u + (-\Delta)^{s}_{p} u, $ has received significant attention in recent years, as such operators naturally arise in a wide range of applications, including the biological sciences, plasma physics, and various areas of mathematical analysis. Their investigation has been pursued through different methodologies, motivated by connections to several important problems. These include Bernstein-type regularity results, smooth approximations via suitable solutions, probability theory and Harnack inequalities, Cahn--Hilliard and Allen--Cahn equations, the Aubry--Mather theory for pseudo-differential equations, and viscosity solution techniques; see \cite{dipierro2021description, dipierro2022non, alibaud2020liouville, barles2012lipschitz, biswas2010viscosity, coville2014nonlocal, cabre2022bernstein, blazevski2013local} and the references therein.\\[4pt]
A representative example is provided in \cite{dipierro2022non}, where a model was proposed to describe the diffusion of a biological population within an ecological niche subject to both local and nonlocal dispersal mechanisms. This model is motivated by a biologically relevant scenario in which long-range foraging movements alternate with localized search strategies on smaller spatial scales. Within this framework, the exterior of the niche remains accessible to the population, but constitutes a hostile environment that enforces an immediate return upon any potential egression. Building on this setting, \cite{dipierro2021description} established a rigorous mathematical formulation of the same ecological diffusion process, obtained through the superposition of Brownian and L\'{e}vy processes, thereby offering a perspective particularly well suited for ecological applications. Precisely, the authors considered the following diffusive equation, which involves a combination of local and nonlocal operators (i.e. $ p = 2 $):
\begin{eqnarray}\label{equ2}
\partial_{t}u - \alpha \Delta u + \beta (-\Delta)^{s} u = 0 
\quad &\text{ in } (0, \infty) \times \Omega,
\end{eqnarray}
subject to the classical and nonlocal Neumann boundary conditions
\begin{eqnarray*}
	\dfrac{\partial u}{\partial \nu}(x) &=& 0, 
	\quad  x \in \partial \Omega,\\[10pt]
	\displaystyle \int_{\Omega} \dfrac{u(x) - u(y)}{|x - y|^{N+2s}}\,dy &=& 0, 
	\quad x \in \mathbb{R}^{N}\setminus \Omega,
\end{eqnarray*}
where $u = u(t, x)$ denotes the density of a biological population inhabiting the niche represented by the bounded smooth domain $\Omega \subset \mathbb{R}^N$ with outward unit normal vector $\nu$, and $s \in (0,1)$. The system is complemented with an initial condition prescribed at $t=0$. Here, $\alpha$ and $\beta$ are positive constants. For further details and related developments, we refer the reader to \cite{kao2012evolution, cancelier1986problems, chen2008evolution} and the references therein.\\[4pt]
On the one hand, various fundamental aspects of stationary elliptic problems driven by mixed operators--combining both local and nonlocal components---have been thoroughly investigated in the literature. Significant contributions in this direction include results on existence, multiplicity, nonexistence, uniqueness, Sobolev regularity of both power- and exponential-type, boundary behavior, H\"{o}lder continuity, and estimates for Green functions (see, e.g., \cite{bal2024regularity, dhanya2024interior, arora2025combined, gouasmia2024uniqueness, chen2010sharp, antonini2025global}), without attempting to provide an exhaustive overview. It is noteworthy that several of the results obtained in these works on such local--nonlocal structures play an essential role in our analysis. On the other hand, comparatively fewer studies have focused on the corresponding parabolic problems, in particular on doubly nonlinear evolution equations governed by this class of operators. In \cite{das2024gradient}, the author studied the heat equation \eqref{equ2}, establishing the local H\"{o}lder regularity of weak solutions and deriving estimates for the $L^{1}$-norm of their gradients. Earlier, \cite{garain2023weak} obtained a weak Harnack inequality for the same problem. Considering the more general operator $-\Delta_{p} (\cdot) + (-\Delta)^{s}_{p} (\cdot)  $ with $1 < p < \infty$ in \eqref{equ2}, \cite{fang2022regularity} established local boundedness and H\"{o}lder continuity of weak solutions, while \cite{garain2023weak, garain2024regularity} further derived a Harnack-type estimate. More recently, \cite{nakamura2022local, nakamura2023harnack} established the local boundedness of sign-changing weak solutions for \(p \geq 2\) and derived a Harnack inequality for globally bounded positive weak solutions of the doubly nonlinear parabolic equation:
\begin{equation*}
\partial_{t}\bigl( |u|^{p - 2}u\bigr) - \Delta_{p} u + (-\Delta)_{p}^{s} u = 0.
\end{equation*}
Regarding the issues of existence, uniqueness, and Sobolev regularity, \cite{bal2024mixed} addressed these questions by employing approximation schemes for the parabolic equation with a singular nonlinearity:
\begin{equation*}
\left\{
\begin{aligned}
\partial_{t}u - \Delta u + (-\Delta)^{s} u &= f(t, x)\, u^{-\gamma(t, x)} && \text{in } (0, T) \times \Omega, \\
u &= 0 && \text{in } (0, T) \times (\mathbb{R}^{N}\setminus\Omega), \\
u(0, x) &= u_0(x) && \text{in } \Omega,
\end{aligned}
\right.
\end{equation*}
where $0<T<\infty$, $\gamma$ is a positive continuous function on $(0, T) \times \Omega$, $f \geq 0$, $u_0 \geq 0$, and both $f$ and $u_0$ belong to appropriate Lebesgue spaces. Very recently, \cite{constantin2025doubly} addressed the more general case with $q>1$ and $p>2$, obtaining results on existence, uniqueness, and qualitative properties such as extinction, blow-up, and stabilization for solutions of  the following
doubly nonlinear parabolic equation:
\begin{equation*}
\left\{
\begin{aligned}
\partial_t \beta(u) - \Delta_{p} u + (-\Delta)^{s}_{q} u &= \text{div}(\vec{f}(u)) + g(t, x, u) && \text{in } Q_T, \\
u &= 0 && \text{in } \Gamma_T, \\
u(0,\cdot) &= u_0 && \text{in } \Omega,
\end{aligned}
\right.
\end{equation*}
where the nonlinearities $\beta$, $\vec{f}$, and $g$ satisfy appropriate growth and structural conditions. For additional related contributions, we refer to \cite{del2025fujita, carhuas2025global, biagi2025global, shang2022holder, shang2023harnack,  garain2025some, radulescu2024harnack} and the references therein.   Despite the substantial progress made in establishing local and global H\"{o}lder regularity, as well as various Harnack-type inequalities for mixed local–nonlocal problems, the literature still contains relatively few results concerning the fundamental issues of existence, uniqueness, and other qualitative properties of solutions.\\[4pt]
To the best of our knowledge, only a limited number of results are available concerning parabolic equations involving mixed local and nonlocal operators, particularly with respect to the existence, uniqueness, and other qualitative properties of weak solutions. The main objective of this paper is to address these questions in the context of problem \eqref{P1}, which corresponds to a doubly nonlinear evolution equation with \(m > 0\).   Depending on the value of \(m\), this problem encompasses a broad class of models.  Our methodology begins with a semi-discretization in time applied to an auxiliary evolution problem, through which we first establish the local existence of weak energy solutions. Uniqueness is also established.  In the case $m > 0$, it is obtained via a novel generalization of the classical inequality of D\'{i}az and Saa, suitably adapted to the mixed local–nonlocal framework. This generalization yields a new comparison principle and demonstrates the \(T\)-accretivity of an associated operator in \(L^2\).  By employing this comparison principle, we construct appropriate barrier functions that enable the establishment of global existence for weak solutions. It is worth emphasizing that, in the case $m>0$, the choice of appropriate test functions is essential for establishing the preceding results. Test functions that are suitable for $m=0$ typically fail to yield valid conclusions when $m>0$. This challenge becomes particularly pronounced in problems involving a source term, where overcoming the associated difficulties necessitates test functions specifically tailored to the case $m>0$. Moreover, we show that these solutions converge, as \(t \to \infty\), to a nontrivial stationary solution.The analysis relies on advanced techniques from contraction semigroup theory, providing a rigorous framework for these results.

\section{Preliminaries and main results}
We begin by recalling the definitions and fundamental properties of the functional spaces that constitute the framework of our analysis. We then present the main results related to the existence, uniqueness, regularity, and asymptotic behavior of weak solutions to problem~\eqref{P1}. In addition, we establish several auxiliary results for the associated stationary problems (see \eqref{PP2} and \eqref{P7}), which play an essential role in proving the aforementioned properties of the solutions to~\eqref{P1}.\\[4pt]
We first introduce the standard notation $ t^{\pm} := \max\{\pm t,\, 0\}, $  $ t \in \mathbb{R}. $ In addition, for a fixed parameter \(k > 0\), we define the truncation operator 
\(\mathbf{T}_{k}\colon \mathbb{R} \to \mathbb{R}\) by
\[
\mathbf{T}_{k}(s) :=
\begin{cases}
\min\{s,\, k\}, & \text{if } s \geq 0, \\[4pt]
-\mathbf{T}_{k}(-s), & \text{if } s < 0.
\end{cases}
\]
Throughout this work, we assume that \( \Omega \subset \mathbb{R}^{N} \) is a bounded domain with smooth boundary, where \( N > p \) for some \( p > 1 \). For a measurable function \( u : \mathbb{R}^{N} \to \mathbb{R} \), the norm in the Lebesgue space \( L^{p}(\Omega) \) is defined as
\[
\|u\|_{L^{p}(\Omega)} := \left( \int_{\Omega} |u|^{p} \, dx \right)^{1/p}.
\]
The Sobolev space \( W^{1,p}(\Omega) \) is defined by
\[
W^{1,p}(\Omega) := \left\{ u \in L^{p}(\Omega) \; : \; \nabla u \in L^{p}(\Omega) \right\},
\]
and is endowed with the norm
\[
\|u\|_{W^{1,p}(\Omega)}^{p} := \|u\|_{L^{p}(\Omega)}^{p} + \|\nabla u\|_{L^{p}(\Omega)}^{p}.
\]
Furthermore, we define the space
\[
W_{0}^{1,p}(\Omega) := \left\{ u \in W^{1,p}(\Omega) \; : \; u\mid_{\partial\Omega }= 0 \right\},
\]
which is equipped with the norm
\[
\|u\|_{W_{0}^{1,p}(\Omega)} := \|\nabla u\|_{L^{p}(\Omega)}.
\]
In this framework, we recall the classical Sobolev embedding theorem:
\begin{theorem}[{\cite[Theorem 3]{evans2022partial}}] \label{thm0}
	Let \( p \geq 1 \) with \( N > p \). Then there exists a positive constant \( C = C(N, p, \Omega) \) such that, for every function \( u \in W^{1, p}_{0}(\Omega) \), the following inequality holds:
	\[
	\| u \| _{L^{p^{*}}(\Omega)} \leq C \| u \| _{W^{1, p}_{0}(\Omega)},
	\]
	where \( p^{*} \) denotes the Sobolev critical exponent, given by \( p^{*} = \frac{Np}{N - p} \).
\end{theorem}
\begin{remark}\label{remark0}
	As a direct consequence of Theorem \ref{thm0}, the norms \( \|\cdot\| _{W^{1, p}(\Omega)} \) and \( \|\cdot\| _{W_0^{1, p}(\Omega)} \) are equivalent on \( W_0^{1, p}(\Omega) \). Moreover, according to \cite{evans2022partial}, the space \( W^{1,p}_0(\Omega) \) is continuously embedded into \( L^r(\Omega) \) for all \( 1 \leq r \leq p^{*} \), and compactly embedded for \( 1 \leq r < p^{*} \).
\end{remark}

\noindent 
For \(0 < s < 1\), the fractional Sobolev space \( W^{s, p}(\mathbb{R}^{N}) \) is defined as
\[
W^{s, p}(\mathbb{R}^{N}) := \left\lbrace u \in L^{p}(\mathbb{R}^{N}) \;:\; [u]_{s, p}^{p} := \iint_{\mathbb{R}^{2N}} \frac{\lvert u(x)-u(y)\rvert^{p}}{\lvert x-y\rvert^{N+sp}} \, dx \, dy < \infty \right\rbrace,
\]
endowed with the norm
\[
\|u\|_{W^{s, p}(\mathbb{R}^{N})}^{p} := \|u\|_{L^{p}(\mathbb{R}^{N})}^{p} + [u]_{s, p}^{p}.
\]
The space \( W_{0}^{s, p}(\Omega) \) is defined by
\[
W_{0}^{s, p}(\Omega) := \left\lbrace u \in W^{s, p}(\mathbb{R}^{N}) \;:\; u = 0 \ \text{a.e. in } \mathbb{R}^{N}\setminus\Omega \right\rbrace,
\]
and is equipped with the Gagliardo seminorm $ \|u\|_{W_{0}^{s, p}(\Omega)} := [u]_{s, p}. $ Moreover, if \( \partial \Omega \) is sufficiently smooth, then \( W_{0}^{s, p}(\Omega) \) coincides with the closure of \( C_{c}^{\infty}(\Omega) \) in the norm \( [\cdot]_{s, p} \), where
\[
C^{\infty}_{c}(\Omega) := \left\lbrace \varphi : \mathbb{R}^{N} \to \mathbb{R} \;:\; \varphi \in C^{\infty}(\mathbb{R}^{N}) \text{ and } \mathrm{supp}(\varphi) \Subset \Omega \right\rbrace.
\]
We set
\[
C_{0}(\overline{\Omega})
:=
\left\{
u \in C(\overline{\Omega}) \,:\, u|_{\partial\Omega}=0
\right\}.
\]

\noindent
We next recall the fractional Poincaré inequality:
\begin{theorem}[{\cite[Theorem~6.5]{di2012hitchhikers}}] \label{thm3}
	Let \( s \in (0,1) \), \( p \geq 1 \), and assume \( N > sp \). Then there exists a constant \( C = C(N,p,s,\Omega) > 0 \) such that, for every measurable compactly supported function \( u : \mathbb{R}^{N} \to \mathbb{R} \), one has
	\[
	\|u\|_{L^{p_{s}^{*}}(\mathbb{R}^{N})}^{p} \leq C \iint_{\mathbb{R}^{2N}} \frac{\lvert u(x)-u(y)\rvert^{p}}{\lvert x-y\rvert^{N+sp}} \, dx \, dy,
	\]
	where \( p_{s}^{*} \) denotes the Sobolev critical exponent, given by \( p_{s}^{*} = \frac{Np}{N-sp} \).
\end{theorem}

\noindent 
We define the space
\[
\mathbb{W}_{0}^{1,p}(\Omega) 
= \left\{ u \in W^{1,p}(\mathbb{R}^{N}) \; : \; u = 0 \;\; \text{a.e. in } \mathbb{R}^{N} \setminus \Omega \right\}.
\]
According to \cite[Proposition~9.18]{brezis2011functional}, if \( \Omega \) has a sufficiently smooth boundary, then
\(\mathbb{W}_{0}^{1,p}(\Omega)\) coincides with \( W^{1,p}_{0}(\Omega) \). In particular, if \( \Omega \) is of class \( C^{1} \), we have
\[
u \in W^{1,p}_{0}(\Omega) 
\;\;\Longleftrightarrow\;\;
\tilde{u} := u \chi_{\Omega} \in W^{1,p}(\mathbb{R}^{N}) \;\;\Longleftrightarrow\;\;
\tilde{u} \in \mathbb{W}_{0}^{1,p}(\Omega),
\]
which shows that each function \( u \in W^{1,p}_{0}(\Omega) \) can be identified with its zero extension in \( W^{1,p}(\mathbb{R}^{N}) \). Conversely, the restriction to \( \Omega \) of any \( u \in W^{1,p}(\mathbb{R}^{N}) \) vanishing a.e. in \( \mathbb{R}^{N} \setminus \Omega \) belongs to \( W^{1,p}_{0}(\Omega) \). On the other hand, it follows from \cite[Lemma~2.1]{buccheri2022system} (see also \cite[Lemma~1.8]{gouasmia2024uniqueness}) that
\[
W^{1,p}_{0}(\Omega) \subset W^{s,p}_{0}(\Omega).
\]
In this paper, since the analysis of parabolic problems requires function spaces that explicitly involve the time variable, we introduce them here for completeness. Let $T>0$, and let $u$ be a function defined by $ u : (0,T) \longrightarrow W_{0}^{1,p}(\Omega),  \, u(t)(x):=u(t,x). $ We denote by $C\big([0,T], W_{0}^{1,p}(\Omega)\big)$ the space of continuous functions on $[0,T]$ with values in $W_{0}^{1,p}(\Omega)$, equipped with the norm
\[
\|u\|_{C([0,T], W_{0}^{1,p}(\Omega))} := \sup_{t\in[0,T]} \|u(t)\|_{W_{0}^{1,p}(\Omega)}.
\]
In addition, we introduce the associated Bochner spaces:
\[
L^{1}\!\left(0,T; W^{1,p}_{0}(\Omega)\right) 
:= \Big\{\, u : (0,T)\to W^{1,p}_{0}(\Omega)\, :\, \|u\|_{L^{1}(0,T; W^{1,p}_{0}(\Omega))} < +\infty \,\Big\},
\]
and
\[
L^{\infty}\!\left(0,T; W^{1,p}_{0}(\Omega)\right) 
:= \Big\{\, u : (0,T)\to W^{1,p}_{0}(\Omega)\, :\, \|u\|_{L^{\infty}(0,T; W^{1,p}_{0}(\Omega))} < +\infty \,\Big\},
\]
where the norms are defined, respectively, by
\[
\|u\|_{L^{1}(0,T; W^{1,p}_{0}(\Omega))} 
:= \int_{0}^{T} \|u(t)\|_{W^{1,p}_{0}(\Omega)}\, dt,
\]
and
\[
\|u\|_{L^{\infty}(0,T; W^{1,p}_{0}(\Omega))} 
:= \operatorname*{ess\,sup}_{t\in(0,T)} \|u(t)\|_{W^{1,p}_{0}(\Omega)}.
\]
For \(r > 0\) and parameters \(\alpha, \beta > 0\), we set
\[
\mathcal{M}_{\alpha, \beta}^{r}(\Omega)
:= \Bigl\{\, u : \Omega \to \mathbb{R}^{+} \;\Big|\; u \in L^\infty(\Omega), \ \exists\, c > 0 \ \text{such that} \ c^{-1} d^{\alpha}(x) \leq u^{r}(x) \leq c\, d^{\beta}(x) \Bigr\},
\]
and define
\begin{equation}\label{eq1}
\dot{\mathbf{V}}_{+}^{\, r + 1}
:= \Bigl\{\, u : \mathbb{R}^{N} \to (0, \infty) \;\Big|\; u^{\frac{1}{r+1}} \in W^{1,p}(\mathbb{R}^{N}) \Bigr\}.
\end{equation}
Furthermore, we introduce the weighted space
\[
L_{d^{\, r+1}}^{\infty}(\Omega)
:= \Bigl\{\, u : \Omega \to \mathbb{R} \ \Big|\ 
u \in L^\infty(\Omega) \ \text{and} \ d^{-(r+1)}(\cdot)\, u(\cdot) \in L^\infty(\Omega) \Bigr\}.
\]

\begin{remark}\label{remark1}
	Let \( \phi^{s}_{1,p} \) denote the positive eigenfunction corresponding to the first eigenvalue \( \lambda^{s}_{1,p} \) of the operator \( -\Delta_{p}(\cdot) + (-\Delta)^{s}_{p}(\cdot) \) in \( W^{1,p}_{0}(\Omega) \). It is well known that \( \phi^{s}_{1,p} \in L^{\infty}(\Omega) \) (see \cite[Theorem~2.2]{palatucci2026mixed}). We normalize this eigenfunction by imposing \( \|\phi^{s}_{1,p}\|_{L^{\infty}(\Omega)} = 1 \). Furthermore, the regularity result \( \phi^{s}_{1,p} \in C^{1,\zeta}(\overline{\Omega}) \), for some \( \zeta \in (0,1) \), follows from \cite[Theorem~1.1]{antonini2025global}. Furthermore, Hopf's lemma \cite[Theorem~1.2]{antonini2025global} ensures that  $ \phi^{s}_{1,p}(x) \geq c_{1}\, d(x) $ in $ \Omega, $ whereas Theorem~\ref{theorem1} (see below) provides the upper bound  $ \phi^{s}_{1,p}(x) \leq c_{2}\, d^{\alpha}(x) $ in $ \Omega, $ for every \( \alpha \in [s,1) \) with \( \alpha \neq \frac{ps}{p-1} \), where \( c_{1}, c_{2} > 0 \) are constants.  Hence, we deduce that  $ 	\phi^{s}_{1,p} \in \mathcal{M}_{1,\alpha}^{1}(\Omega).	 $
\end{remark}
\begin{remark}
		Indeed, the case $ \alpha = \frac{sp}{p - 1} $ has not been treated. This is due to the difficulty in constructing a supersolution for problem~\ref{P8} (see below). Moreover, handling the fractional part requires estimates that hold under this specific condition. We also note that $p(1 - s) < 1$, which implies $ 1 < \frac{ps}{p - 1}  $.
\end{remark}
\begin{definition} \label{defn1} \rm
	Let \(X\) be a real vector space and \(C \subset X\) a nonempty convex cone.  
	A functional \(\mathcal{W} : C \to \mathbb{R}\) is \textbf{ray-strictly convex} (resp. \textbf{strictly convex}) if
	\[
	\mathcal{W}\big((1-t)v_1 + t v_2\big) \leq (1-t)\mathcal{W}(v_1) + t\,\mathcal{W}(v_2),
	\quad v_1, v_2 \in C,\; t \in (0,1),
	\]
	with strict inequality unless \(v_1/v_2 \equiv c>0\) (unless \(v_1 \equiv v_2\), resp.).
\end{definition}

\noindent 
We aim to extend the well-known D\'{i}az--S\'{a}a inequality to the setting of mixed local and nonlocal operators.  This inequality is closely related to the strict convexity of an associated homogeneous energy-type functional.  Precisely, for $1 < r + 1 \leq p$, we introduce the functional $ \mathcal{W} : \dot{V}_{+}^{r + 1} \longrightarrow \mathbb{R}_{+} $ defined in a generalized setting by
\[
\mathcal{W}(w) = 
\int_{\omega} 
\mathcal{H}\!\left(x, \nabla \!\left(w^{\tfrac{1}{r + 1}}\right)\right) \, dx
+ 
\iint_{\omega \times \omega} 
K(x, y)\, \big| w(x)^{\tfrac{1}{r + 1}} - w(y)^{\tfrac{1}{r + 1}} \big|^{p}\, dx\,dy,
\]
for any $ \omega \subseteq \mathbb{R}^{N} $. Here, the measurable kernel $K : \omega \times \omega \to \mathbb{R}^{+}$ is assumed to satisfy
\[
K(x, y) = K(y, x) \quad \text{for a.e. } x, y \in \omega,
\]
and
\[
\dfrac{\kappa}{\Lambda |x - y|^{N + sp}} 
\;\leq\; K(x, y) \;\leq\; 
\dfrac{\kappa \Lambda}{|x - y|^{N + sp}} 
\quad \text{for a.e. } x, y \in \omega,
\]
for some fixed constants $\Lambda > 0$ and $\kappa \in (0,1]$.  Moreover, the Carath\'{e}odory function $\mathcal{H} : \Omega \times \mathbb{R}^{N} \to \mathbb{R}^{+}$ is required to satisfy
\[
z \mapsto \mathcal{H}(x, z) \;\; \text{is of class } C^{1} \text{ and convex},
\]
and
\[
\mathcal{H}(x, tz) = |t|^{p} \, \mathcal{H}(x, z) 
\quad \text{for all } t \in \mathbb{R}, \; (x, z) \in \Omega \times \mathbb{R}^{N}.
\]

\noindent 
Then we obtain the following result:

\begin{proposition} \label{pro3}
	The functional \( \mathcal{W} \) is ray-strictly convex on the convex cone \( \dot{V}_{+}^{r + 1} \). Moreover, if \( r \neq p - 1 \), then \( \mathcal{W} \) is strictly convex on \( \dot{V}_{+}^{r  +1 } \).
\end{proposition}
\begin{proof}
	By \cite[Propositions~2.6 and~4.1]{brasco2014convexity}, the set $\dot{V}_{+}^{r + 1}$ defined in \eqref{eq1} forms a convex cone. In particular, if $u, v \in \dot{V}_{+}^{r + 1}$ and $\lambda \in [0,1]$, then $(1-\lambda)u + \lambda v \in \dot{V}_{+}^{r + 1}$. According to Definition~\ref{defn1}, let $w_1, w_2 \in \dot{V}_{+}^{r + 1}$ and $t \in [0,1]$, and define $ w = (1-t)w_1 + t w_2. $ Consequently, one has
	\begin{equation*}
	\mathcal{W}(w) \leq (1-t)\,\mathcal{W}(w_1) + t\,\mathcal{W}(w_2).
	\end{equation*}
	Assume now that equality holds. In this case, one obtains
	\begin{align*}
	0 &\leq \int_{\omega} \Bigg[ (1-t)\,\mathcal{H}\!\left(x, \nabla \!\big(w_{1}^{\frac{1}{r + 1}}\big)\right) 
	+  t\,\mathcal{H}\!\left(x, \nabla \!\big(w_{2}^{\frac{1}{r + 1}}\big)\right) 
	- \mathcal{H}\!\left(x, \nabla \!\big(w^{\frac{1}{r + 1}}\big)\right) \Bigg] dx \\[4pt]
	&= \iint_{\omega \times \omega} 
	K(x, y)\, \big| w(x)^{\frac{1}{r + 1}} - w(y)^{\frac{1}{r + 1}} \big|^{p}\, dx\,dy \\[4pt]
	&\quad - (1-t)\iint_{\omega \times \omega} 
	K(x, y)\, \big| w_{1}(x)^{\frac{1}{r + 1}} - w_{1}(y)^{\frac{1}{r +1}} \big|^{p}\, dx\,dy \\[4pt]
	&\quad - t \iint_{\omega \times \omega} 
	K(x, y)\, \big| w_{2}(x)^{\frac{1}{r + 1}} - w_{2}(y)^{\frac{1}{r +1}} \big|^{p}\, dx\,dy \;\leq 0.
	\end{align*}
	Therefore,
	\[
	\big| w(x)^{\frac{1}{r+1}} - w(y)^{\frac{1}{r+1}} \big|^{p} 
	= (1-t)\, \big| w_{1}(x)^{\frac{1}{r+1}} - w_{1}(y)^{\frac{1}{r+1}} \big|^{p} 
	+ t\, \big| w_{2}(x)^{\frac{1}{r+1}} - w_{2}(y)^{\frac{1}{r+1}} \big|^{p}, 
	\]
	a.e. $ x, y \in \mathbb{R}^{N}.  $ If $r + 1= p$, it follows that
	\[
	\big| \| a \|_{\boldsymbol{\ell}^{r + 1}} - \| b \|_{\boldsymbol{\ell}^{r + 1}} \big|^{r+ 1}
	= \| a - b \|_{\boldsymbol{\ell}^{r+1}}^{r+ 1} 
	\quad \text{for a.e. } x, y \in \mathbb{R}^{N},
	\]
	where $\| \cdot \|_{\boldsymbol{\ell}^{r+1}}$ denotes the $\boldsymbol{\ell}^{r+1}$-norm in $\mathbb{R}^2$, and
	\[
	a = \big( ((1-t)w_1(x))^{\frac{1}{r+1}}, (t w_2(x))^{\frac{1}{r+1}} \big), 
	\quad
	b = \big( ((1-t)w_1(y))^{\frac{1}{r+1}}, (t w_2(y))^{\frac{1}{r+1}} \big).
	\]
	Since $r > 0$, there exists a constant $c>0$ such that $w_1 = c w_2$ a.e. in $\mathbb{R}^N$. Hence, $\mathcal{W}$ is ray-strictly convex on $\dot{V}_{+}^{r+1}$.  On the other hand, if $r + 1 \neq p$, then by the strict convexity of the mapping $\tau \mapsto \tau^{\frac{p}{r+1}}$ on $\mathbb{R}^{+}$, we deduce that $w_1 = w_2$ a.e. in $\mathbb{R}^N$. Consequently, $\mathcal{W}$ is strictly convex on $\dot{V}_{+}^{r+1}$.
\end{proof}

\begin{lemma} \label{Lem2}
	Let $1 < p < \infty$ and $0 < r \leq p -1$. Then the following inequality holds in the sense of distributions:
\begin{equation}\label{equ0}
		\begin{aligned}
		& \int_{\omega} \nabla_{z}\mathcal{H}(x, \nabla u) \cdot
		\nabla \!\left( \frac{u^{r+1} - v^{r+1}}{u^{r}} \right)  dx 
		+ \int_{\omega}  \nabla_{z}\mathcal{H}(x, \nabla v)\cdot
		\nabla \!\left( \frac{v^{r+1} - u^{r+1}}{v^{r}} \right)  dx \\[2mm]
		& + \iint_{\omega \times \omega} K(x, y) 
		\left[ u(x) - u(y) \right] ^{p-1} 
		\left( \frac{u(x)^{r+1} - v(x)^{r+1}}{u(x)^{r}} 
		- \frac{u(y)^{r+1} - v(y)^{r+1}}{u(y)^{r}}\right)  dxdy \\[1mm]
		& + \iint_{\omega \times \omega} K(x, y) 
		\left[ v(x) - v(y) \right] ^{p-1}
		\left( \frac{v(x)^{r+1} - u(x)^{r+1}}{v(x)^{r}} 
		- \frac{v(y)^{r+1} - u(y)^{r+1}}{v(y)^{r}} \right)  dxdy 	\geq 0.
		\end{aligned}
		\end{equation}
	
	\noindent
	This holds for all $u, v \in W^{1,p}_{0}(\Omega)$ such that $u > 0$ and $v > 0$ a.e. in $\Omega$, with $u/v,\, v/u \in L^{\infty}(\Omega)$.   Moreover, if equality holds, the following conclusions are valid:
	\begin{itemize}
		\item[(i)] $u/v \equiv \text{const} > 0$ a.e. in $\Omega$.
		\item[(ii)] If, in addition, $ r + 1\neq p $, then $u \equiv v$ a.e. in $\Omega$.
	\end{itemize}
\end{lemma}

\begin{proof}
	Let \( u, v \in W^{1, p}_0(\Omega) \) be such that \( u > 0 \) and \( v > 0 \) a.e. in \( \Omega \), and let \( \theta \in (0,1) \). Define $ w := (1-\theta) u^{r+1} + \theta v^{r+1}. $	By \cite[Propositions~2.6 and~4.1]{brasco2014convexity}, it follows that \( w \in \dot{V}_{+}^{r+1} \). Hence, by Proposition~\ref{pro3}, the mapping $ 	\theta \mapsto \Phi(\theta) := \mathcal{W}(w) = \mathcal{W}\big((1 - \theta) u^{r+1} + \theta v^{r+1} \big) $
	is convex and differentiable on \( [0, 1] \). For \( \theta \in (0,1) \), a straightforward computation gives
\begin{equation*}
		\begin{aligned}
		&\Phi'(\theta) = \int_{\omega}  \nabla_{z}\mathcal{H}(x, \nabla w)\cdot
		\nabla \!\left( \frac{v^{r+1} - u^{r+1}}{w^{r}} \right) dx + \\[4pt]
		&\iint_{\omega \times \omega} K(x, y) \Big[ w(x)^{\frac{1}{r+1}} - w(y)^{\frac{1}{r+1}}\Big] ^{p-1} \left( \frac{v(x)^{r+1} - u(x)^{r+1}}{w(x)^{1 - \frac{1}{r+1}}} - \frac{v(y)^{r+1} - u(y)^{r+1}}{w(y)^{1 - \frac{1}{r+1}}} \right) dxdy. 
		\end{aligned}
		\end{equation*}
	
	\noindent 
	Using the convexity of \(\Phi\) and the identities \( w = u^{r+1} \) for \(\theta = 0\) and \( w = v^{r+1} \) for \(\theta = 1\), we obtain
	\[
	\Phi'(0) = \lim_{\theta \to 0^+} \Phi'(\theta) \leq \lim_{\theta \to 1^-} \Phi'(\theta) = \Phi'(1),
	\]
	which is equivalent to inequality~\eqref{equ0}. 
	Finally, assume that equality holds in~\eqref{equ0}. Since \(\Phi' : (0,1) \to \mathbb{R}\) is monotone, it follows that \(\Phi'(\theta)\) is constant on \((0,1)\), and consequently, \(\Phi : [0,1] \to \mathbb{R}\) is linear:
	\[
	\Phi(\theta) = \mathcal{W}(w) = (1-\theta) \Phi(0) + \theta \Phi(1) = (1-\theta) \mathcal{W}(u^{r+1}) + \theta \mathcal{W}(v^{r+1}), 
	\quad \forall \theta \in [0,1].
	\]
	From this, we deduce that \( u \equiv c v \) for some \( c > 0 \). Furthermore, if \( r \neq p - 1 \), then \( u \equiv v \) by Proposition~\ref{pro3}.
\end{proof}

\noindent
Finally, we collect some auxiliary results associated with the following boundary value problem:
\begin{equation}\label{P8}
\left\{
\begin{aligned}
- \Delta_{p} u + (-\Delta)^{s}_{p} u &= f(x)
\quad &&\text{in } \Omega,\\
u &= 0
\quad &&\text{in } \mathbb{R}^{N} \setminus \Omega.
\end{aligned}
\right.
\end{equation}
Here, the source term \( f \in L^{d}(\Omega)\setminus\{0\} \), with \( d > N \), satisfies the pointwise estimate
\[
0 < f(x) \leq C\, d(x)^{-\beta}
\quad \text{for a.e. } x \in \Omega,
\]
for some constant \( C > 0 \).

\medskip

\begin{theorem}\label{theorem1}
	Let \( 0 < \beta < 1 \), and let \( u \in W^{1,p}_{0}(\Omega) \) be a weak solution of problem~\eqref{P8}. Then, for every \( \alpha \in [s,1) \) with \( \alpha \neq \frac{ps}{p-1} \), there exist positive constants \( c_{1} \) and \( c_{2} \) such that
	\[
	c_{1}\, d(x) \leq u(x) \leq c_{2}\, d(x)^{\alpha}
	\quad \text{for all } x \in \Omega.
	\]
\end{theorem}

\begin{proof}
	Since \( f \in L^{d}(\Omega)\setminus\{0\} \) with \( d > N \), it follows from \cite[Theorem~1.1]{antonini2025global} that
	\( u \in C^{1,\xi}(\overline{\Omega}) \) for some \( \xi \in (0,1) \).
	Moreover, by \cite[Proposition~6.1]{antonini2025global}, we have \( u > 0 \) in \( \Omega \).
	Consequently, an application of Hopf's lemma (see e.g. \cite[Theorem~1.2]{antonini2025global}) yields the existence of a constant \( c_{1} > 0 \) such that $ u(x) \geq c_{1}\, d(x) $ in $ \Omega. $ Next, to derive the upper bound, we adopt the strategy introduced in \cite{arora2021regularity}. The distance function \( d(\cdot) \) can be extended smoothly to \( \Omega^{c} := \mathbb{R}^{N} \setminus \Omega \) for some \( \rho > 0 \) by defining
	\[
	d_{e}(x) :=
	\begin{cases}
	d(x), & \text{if } x \in \Omega,\\
	-\,d(x), & \text{if } x \in (\Omega^{c})_{\rho},\\
	-\,\rho, & \text{otherwise},
	\end{cases}
	\]
	where
	$ (\Omega^{c})_{\rho} := \bigl\{ x \in \Omega^{c} : \mathrm{dist}(x,\partial\Omega) < \rho \bigr\}. $ Let $\eta>0$ be a positive constant, whose value will be fixed later. 
	
	\noindent 
	We introduce the function 	$ \overline{u}_{\rho}(x):=\eta\,\overline{w}_{\rho}(x), $ where $\alpha \in[s,1)$ with $\alpha \neq \frac{ps}{p-1}$, and $\overline{w}_{\rho}$ is defined by
	\[
	\overline{w}_{\rho}(x):=
	\begin{cases}
	\bigl(d_{e}(x)\bigr)_{+}^{\alpha}, & \text{if } x\in \Omega\cup(\Omega^{c})_{\rho},\\[1mm]
	0, & \text{otherwise}.
	\end{cases}
	\]
Assume that the boundary $\partial\Omega$ is of class $C^{2}$. It follows from \cite[Lemma~14.16]{gilbarg1977elliptic} that there exists a constant $\kappa>0$ such that the distance function $d$ belongs to $C^{2}(\Omega_{\kappa})$, where $ \Omega_{\kappa}:=\{x\in\overline{\Omega}:\ d(x)<\kappa\}. $ Furthermore, by virtue of \cite[Lemma~3.10]{giacomoni2021interior}, we infer that
\begin{equation}\label{equ1212}
\iint_{\mathbb{R}^{2N}}
\frac{\bigl[\overline{u}_{\rho}(x)-\overline{u}_{\rho}(y)\bigr]^{p-1}
	\bigl(\varphi(x)-\varphi(y)\bigr)}{|x-y|^{N+sp}}\,dx\,dy
= \eta^{p-1}\int_{\Omega_{\varrho}} h(x)\varphi\,dx,
\end{equation}
for some function $h\in L^{\infty}(\Omega_{\varrho})$, provided that $\varrho>0$ is \textbf{chosen sufficiently small}. Without loss of generality, we may assume that $ \varrho \leq \min\left\{\frac{\kappa}{2},\,1\right\}, $ which ensures that $\Delta d\in L^{\infty}(\Omega_{\varrho})$. Consequently, there exists a constant $M>0$ such that $ |\Delta d|\leq M $ in $ \Omega_{\varrho}. $  Moreover, a direct computation shows that
\[
\nabla \overline{u}_{\rho}
= \eta\,\alpha\, d^{\alpha - 1}\nabla d
\quad \text{in } \Omega_{\varrho}.
\]
Let \( \varphi \in C^{\infty}_{c}(\Omega_{\varrho}) \) be such that \( \varphi \geq 0 \).
Since \( |\nabla d| = 1 \), we obtain
\begin{align*}
-\int_{\Omega_{\varrho}} \Delta_{p}(\overline{u}_{\rho})\,\varphi\,dx
&\geq
\mathbf{C}\,\eta^{p-1}(1-\alpha)(p-1)
\int_{\Omega_{\varrho}}
d(x)^{(\alpha-1)(p-1)-1}\,\varphi\,dx .
\end{align*}
Using \eqref{equ1212} and arguing as in the proof of \cite[Theorem~4.4]{giacomoni2021interior}, we deduce
\begin{equation}\label{equ111}
\int_{\Omega_{\varrho}}
\bigl(-\Delta_{p}\overline{u}_{\rho}+(-\Delta)^{s}_{p}\overline{u}_{\rho}\bigr)
\varphi\,dx
\geq
\frac{\mathbf{C}\,\eta^{p-1}}{2}
\int_{\Omega_{\varrho}}
d(x)^{(\alpha - 1)(p-1)-1}\,\varphi\,dx .
\end{equation}
On the other hand, for the solution \( u \), since \( 0 < \beta < 1 \), we have
\begin{equation}\label{equ112}
\begin{aligned}
\int_{\Omega_{\varrho}}
\bigl(-\Delta_{p}u+(-\Delta)^{s}_{p}u\bigr)\varphi\,dx
&\leq
\mathbf{C}\int_{\Omega_{\varrho}}
d(x)^{(1-\alpha)(p-1)+1-\beta}
d(x)^{(\alpha-1)(p-1)-1}\,\varphi\,dx \\
&\leq
\mathbf{C}\int_{\Omega_{\varrho}}
d(x)^{(\alpha-1)(p-1)-1}\,\varphi\,dx .
\end{aligned}
\end{equation}
Choosing \( \eta > 0 \) sufficiently large, it follows from
\eqref{equ111}–\eqref{equ112} that
\[
-\Delta_{p}\overline{u}_{\rho}+(-\Delta)^{s}_{p}\overline{u}_{\rho}
\geq
-\Delta_{p}u+(-\Delta)^{s}_{p}u 
\quad \text{in } \Omega \text{ in the weak sense}.
\]
By the weak comparison principle (see, for instance, \cite[Proposition~4.1]{antonini2025global}), we deduce that $ u \leq \overline{u}_{\rho} $ in $ \Omega_{\varrho}. $ Moreover, by exploiting the regularity of \( u \) and taking \( \eta > 0 \) sufficiently large, we obtain
\[
u \leq \|u\|_{L^{\infty}(\Omega\setminus\Omega_{\varrho})}
\leq C_{\varrho}
\leq \overline{u}_{\rho}
\quad \text{in } \Omega\setminus\Omega_{\varrho}.
\]
Therefore, $ u \leq \overline{u}_{\rho} $ in $ \Omega, $ which completes the proof.
\end{proof}

\noindent 
Before presenting our main results, we rewrite problem~\eqref{P1} in an equivalent weak formulation. More precisely, for any $m > 0$, we consider the following problem:
\begin{equation}\label{PP1}\tag{$E_{T}$}
\left\{
\begin{aligned}
u^{m} \, \partial_t \!\big(u^{m+1}\big) - \Delta_{p} u + (-\Delta)^{s}_{p} u \;&= g(t, x)\, u^{m}  +  m \, d(x)^{- \gamma}  \, u^{\delta}&& \text{in } Q_T, \\
u \;&> 0 && \text{in } Q_T, \\
u \;&= 0 && \text{on } \Gamma_T, \\
u(0,\cdot) \;&= u_0 && \text{in } \Omega,
\end{aligned}
\right.
\end{equation}
We now introduce the notion of a weak solution:
\begin{definition}\label{definition1}
	Let \(T > 0\). A weak solution to \eqref{PP1} is a nonnegative function $ u \in L^{\infty}(0, T; W_{0}^{1,p}(\Omega)) \cap L^{\infty}(Q_{T}), $
	such that \(u > 0\) in \(\Omega\), \(\partial_{t}(u^{m+1}) \in L^{2}(Q_{T})\), and satisfying, for every \(t \in (0,T]\), the integral identity
	\begin{equation}\label{weakform}
	\begin{aligned}
	& \int_{0}^{t}\int_{\Omega} u^{m} \, \partial_{t}\!\big(u^{m+1}\big)\, \varphi \, dx \, ds 
	+ \int_{0}^{t} \int_{\Omega} |\nabla u|^{p-2} \nabla u \cdot \nabla \varphi \, dx \, ds \\
	&\quad + \int_{0}^{t}\iint_{\mathbb{R}^{2N}} 
	\frac{\left[ u(s,x)-u(s,y)\right] ^{p-1} \big(\varphi(s,x)-\varphi(s,y)\big)}{|x-y|^{N+sp}} \, dx \, dy \, ds \\
	&= \int_{0}^{t}\int_{\Omega}\Big( g(s,x)\, u^{m} + m\,  d(x)^{-\gamma} \, u^{\delta}\Big)\, \varphi \, dx \, ds,
	\end{aligned}
	\end{equation}
	for all test functions \( \varphi \in L^{2}(Q_{T}) \cap L^{1}(0,T; W_{0}^{1,p}(\Omega))\), with  \( u(0,\cdot) = u_{0} \) a.e. in \(\Omega\).
\end{definition}

\noindent 
Recall from Definition~\ref{definition1} that every weak solution of problem~\eqref{PP1} belongs to \(L^{\infty}(Q_{T})\). 
Furthermore, by invoking \cite[Proposition~9.5]{brezis2011functional}, we deduce that for every \(m>0\),
\[
\partial_{t}\!\left(u^{2m+1}\right)
= \frac{2m+1}{m+1}\,u^{m}\,\partial_{t}\!\left(u^{m+1}\right)
\quad \text{in the weak sense}.
\]
Hence, any weak solution of \eqref{PP1} also satisfies \eqref{P1}, which allows us to restrict our analysis to problem~\eqref{PP1} without loss of generality. 
This formulation is particularly convenient for implementing a semi-discretization in time scheme and facilitates the investigation of the convergence of weak solutions to a steady state via semigroup theory.   More precisely, adopting an approach inspired by the theory of nonlinear accretive operators as developed in \cite{arora2020doubly, badra2012singular, giacomoni2021existence}, the nonlinearity present in the time-derivative term for \(m>0\) is transferred to the diffusion term of the associated parabolic problem~\eqref{PP1}. Consequently, we arrive at the equivalent formulation
\begin{equation}\label{0equ4}
\begin{cases}
\partial_t v + \mathcal{T}_m v = g(t,x), & \text{in } Q_T,\\[1mm]
v > 0, & \text{in } Q_T,\\[1mm]
v = 0, & \text{on } \Gamma_T,\\[1mm]
v(0, \cdot) = v_0, & \text{in } \Omega,
\end{cases}
\end{equation}
where $v = u^{m+1}$, and the nonlinear operator
\[
\mathcal{T}_m : D(\mathcal{T}_m) \subset L^2(\Omega) \to L^2(\Omega)
\]
is defined by
\begin{equation}\label{equ37}
\begin{aligned}
\mathcal{T}_m v
&= v^{- \frac{m}{m+1}} \Bigg(
2\,\mathrm{P.V.}\!\int_{\mathbb{R}^N}
\frac{\big| v^{\frac{1}{m+1}}(x) - v^{\frac{1}{m+1}}(y) \big|^{p-2}
	\big( v^{\frac{1}{m+1}}(x) - v^{\frac{1}{m+1}}(y) \big)}
{|x - y|^{N + sp}}\,dy \\
&\quad - \Delta_{p}\!\left( v^{\frac{1}{m+1}} \right)
- m \, d(x)^{-\gamma}\, v^{\frac{\delta}{m+1}} \Bigg),
\end{aligned}
\end{equation}
with domain
\[
D(\mathcal{T}_m)
= \Big\{\, w : \Omega \to \mathbb{R}^{+} \text{ measurable } \,\big|\,
w^{\frac{1}{m+1}} \in W_0^{1,p}(\Omega) \cap L^{2(m+1)}(\Omega),\;
\mathcal{T}_m w \in L^2(\Omega) \,\Big\}.
\]

\subsection{Main results for the stationary problems related to problem~\eqref{PP1}} 
As a preliminary step, we examine two stationary nonlinear problems associated with \eqref{PP1}. We begin by studying the following one:
\begin{equation}\label{PP2}\tag{$S_{\lambda}$}
\left\{
\begin{aligned}
u^{2m+1} - \lambda \Delta_{p} u + \lambda (-\Delta)^{s}_{p} u \;&= g_{0}(x)\, u^{m} + \lambda m\, d(x)^{-\gamma}\, u^{\delta} && \text{in } \Omega, \\
u \;&> 0 && \text{in } \Omega, \\
u \;&= 0 && \text{on } \mathbb{R}^N \setminus \Omega
\end{aligned}
\right.
\end{equation}
where \( \lambda > 0 \) is a parameter, and \( g_{0} \) is a nonnegative function such that \( g_{0} \in L^{r}(\Omega) \), with \( r \geq 2 \), satisfying
\begin{equation}\label{equ3}\tag{H2}
g_{0}(x) \,\geq\, \lambda \,\underline{g}(x) \quad \text{for a.e. } x \in \Omega,
\end{equation}
with \( \underline{g} \) defined in \eqref{equ1}.  
In what follows, we introduce the notions of weak sub-solutions, super-solutions, and solutions to problem \eqref{PP2}.

\begin{definition}\label{definition2}
	A function $u \in \textbf{WL} : = W^{1,p}_{0}(\Omega) \cap L^{2(m+1)}(\Omega)$, with $u \geq 0$ and $u \not\equiv 0$, is said to be a \text{weak sub-solution} (resp. \text{weak super-solution}) of \eqref{PP2} if it satisfies
	\begin{equation}\label{weakform1}
	\begin{aligned}
	& \int_{\Omega} u^{2m+1} \, \varphi \, dx 
	+ \lambda \int_{\Omega} |\nabla u|^{p-2} \nabla u \cdot \nabla \varphi \, dx   + \lambda \iint_{\mathbb{R}^{2N}} 
	\frac{\left[ u(x)-u(y)\right] ^{p-1} \big(\varphi(x)-\varphi(y)\big)}{|x-y|^{N+sp}} dx dy \\
	& \qquad \leq \ (\text{resp. } \geq) \int_{\Omega} \Big( g_{0}(x)\, u^{m} + \lambda \,m\,  d(x)^{-\gamma}\, u^{\delta} \Big) \varphi  dx,
	\end{aligned}
	\end{equation}
	for all test functions \( \varphi \in \textbf{WL} \). 
	A function that is simultaneously a sub-solution and a super-solution of \eqref{PP2} is referred to as a \text{weak solution} of \eqref{PP2}.
\end{definition}

\noindent 
We present the main results related to problem \eqref{PP2}. In particular, we first establish a comparison principle, which serves as a fundamental tool for proving the uniqueness of solutions and deriving additional estimates.
\begin{theorem}\label{theorem2}
	Let $\gamma < \frac{1}{\textbf{2}}$. Suppose that $\underline{u}$ and $\overline{u} \in \textbf{WL} $ be nonnegative functions representing a sub-solution and a super-solution of problem \eqref{PP2}, respectively. Then 
	$ \underline{u} \leq \overline{u}  $ a.e. in $ \Omega. $
\end{theorem}

\noindent 
We next examine the existence, regularity, and uniqueness of the weak solution to \eqref{PP2} under two distinct assumptions on the potential function $g_{0}$. Specifically, we establish the following theorem, beginning with the case where $g_{0} \in L^{\infty}(\Omega)$.
\begin{theorem} \label{theorem3}
	Let $\gamma < \frac{1}{\textbf{d}}$, and assume that $g_{0} \in L^{\infty}(\Omega)$ satisfies condition~\eqref{equ3}. Then, for every $\lambda > 0$, there exists a positive weak solution  	$ u \in C^{1, \xi}(\overline{\Omega}) \cap \mathcal{M}^{1}_{1, \alpha}(\Omega) $ 	to problem~\eqref{PP2}, for some $\xi \in (0, 1)$ and  for every \( \alpha \in [s,1) \) with \( \alpha \neq \frac{ps}{p-1} \). Moreover, if $u_{1}$ and $u_{2}$ denote two weak solutions of~\eqref{PP2} corresponding to $g_{1}, g_{2} \in L^{\infty}(\Omega)$ satisfying~\eqref{equ3}, respectively, then the following contraction property holds:
	\begin{equation}\label{contraction}
	\left\| \left( u_{1}^{m+1} - u_{2}^{m+1}\right) ^{+} \right\| _{L^{2}(\Omega)} \leq \left\|  \left( g_{1} - g_{2}\right) ^{+} \right\| _{L^{2}(\Omega)}.
	\end{equation}
\end{theorem}

\noindent 
Now, we extend the previous theorem to the case where $ g_{0} \in L^{2}(\Omega) $.
\begin{theorem} \label{theorem4}
	Let $\gamma < \frac{1}{\textbf{d}}$, and assume that \( g_0 \in L^2(\Omega) \) satisfies condition~\eqref{equ3}.
	Then, there exists a positive weak solution \( u \in \mathbf{WL} \) to problem~\eqref{PP2}. 
	Furthermore, if \( g_0 \in L^r(\Omega) \) for some \( r > \frac{N}{s p} \) when \( \gamma \leq 0 \), and \( \frac{N}{s p} < r < \frac{1}{\gamma} \) when \( \gamma > 0 \), then \( u \in L^{\infty}(\Omega) \).
	Moreover, let \( u_1, u_2 \) be two weak solutions of~\eqref{PP2} corresponding to \( g_1, g_2 \in L^2(\Omega) \), respectively, both satisfying~\eqref{equ3}. Then, the following stability estimate holds:
	\begin{equation}\label{equ63}
	\left\| \left( u_{1}^{m+1} - u_{2}^{m+1}\right) ^{+} \right\|_{L^{2}(\Omega)} 
	\leq 
	\left\| \left( g_{1} - g_{2}\right) ^{+} \right\|_{L^{2}(\Omega)}.
	\end{equation}
\end{theorem}

\noindent 
It is important to highlight that, based on the boundary behavior of the weak solution to~\eqref{PP2} established in Theorems~\ref{theorem3} and~\ref{theorem4}, we study the following perturbed problem associated with the operator $\mathcal{T}_m$ (see~\eqref{equ37}), which arises naturally from the parabolic formulation presented in~\eqref{0equ4} for $m > 0$:
\begin{equation}\label{equ22}
\begin{cases}
v + \lambda \mathcal{T}_m v = g_{0}, & \text{in } \Omega,\\[1mm]
v > 0, & \text{in } \Omega,\\[1mm]
v = 0, & \text{on } \mathbb{R}^N \setminus \Omega.
\end{cases}
\end{equation}
Moreover, by invoking Theorem~\ref{theorem3}, we derive the $T$-accretivity property of  $\mathcal{T}_m$, as formulated below.

\begin{corollary}\label{cor1}
Let $\gamma < \frac{1}{\textbf{d}}$, let $m>0$, and assume that $g_0 \in L^{\infty}(\Omega)$ satisfies~\eqref{equ3}. Then, problem~\eqref{equ22} admits a unique solution 
	$v \in C^{1}(\overline{\Omega})$. Moreover, $v \in \dot{V}_{+}^{m+1} \cap \mathcal{M}_{1,\alpha}^{\frac{1}{m+1}}(\Omega)$  for every \( \alpha \in [s,1) \) with \( \alpha \neq \frac{ps}{p-1} \), and $v$ satisfies
	\begin{equation}\label{equ24}
	\begin{aligned}
	&\int_{\Omega} v\,\Psi \, dx 
	+ \lambda \int_{\Omega} \left|\nabla \left( v^{\frac{1}{m+1}} \right)\right|^{p-1} 
	\nabla \left( v^{\frac{1}{m+1}}\right)\cdot 
	\nabla \left(\frac{\Psi}{v^{\frac{m}{m+1}}}\right) dx \\
	&\quad + \lambda \iint_{\mathbb{R}^{2N}}
	\frac{\left[ v^{\frac{1}{m+1}}(x)-v^{\frac{1}{m+1}}(y)\right] ^{p-1} }{|x-y|^{N+sp}}  
	\left[ \left( \frac{\Psi}{v^{\frac{m}{m+1}}}\right)(x)
	- \left( \frac{\Psi}{v^{\frac{m}{m+1}}}\right)(y) \right]
	\, dx \, dy  \\
	&= \int_{\Omega} g_0(x)\,\Psi \, dx
	+ \lambda m \int_{\Omega} d(x)^{-\gamma} 
	v^{\frac{\delta-m}{m+1}}\,\Psi \, dx ,
	\end{aligned}
	\end{equation}
	for every test function $\Psi$ such that
	\[
	|\Psi| \in L^{\infty}_{d^{m+1}}(\Omega)
	\quad \text{and} \quad
	\frac{|\nabla \Psi|}{d(\cdot)^{m}} \in L^{p}(\Omega).
	\]
	Furthermore, if $v_1$ and $v_2$ are solutions of problem~\eqref{equ22}
	corresponding to data $g_1$ and $g_2$, respectively, then the following
	stability estimate holds:
	\begin{equation}\label{equ27}
	\left\| \left( v_1 - v_2 \right)^{+} \right\|_{L^2(\Omega)}
	\leq 
	\left\| \left( v_1 - v_2 
	+ \lambda \big(\mathcal{T}_m(v_1)-\mathcal{T}_m(v_2)\big) \right)^{+}
	\right\|_{L^2(\Omega)}.
	\end{equation}
\end{corollary}

\noindent 
The next result extends Corollary~\ref{cor1} to the case of an $L^{2}$-potential, as stated below.
\begin{corollary}\label{cor22}
Let $\gamma < \frac{1}{\textbf{d}}$, let $m>0$, and assume that  $g_{0}\in L^{2}(\Omega)\cap L^{r}(\Omega)$ for some
	$r>\frac{N}{sp}$, satisfying \eqref{equ3}. Then problem~\eqref{equ22} admits a unique
	solution $v$. More precisely,
	\[
	v\in \dot{V}_{+}^{m+1}\cap L^{\infty}(\Omega),
	\]
	and $v$ satisfies
	
	\begin{equation}\label{equu}
	\begin{aligned}
	&\int_{\Omega} v\,\Psi \, dx
	+\lambda \int_{\Omega}
	\left|\nabla\!\left(v^{\frac{1}{m+1}}\right)\right|^{p-1}
	\nabla\!\left(v^{\frac{1}{m+1}}\right)\!\cdot\!
	\nabla\!\left(\frac{\Psi}{v^{\frac{m}{m+1}}}\right) dx  \\
	& \quad+\lambda \iint_{\mathbb{R}^{2N}}
	\frac{\left[ v^{\frac{1}{m+1}}(x)-v^{\frac{1}{m+1}}(y)\right] ^{p-1}
		}
	{|x-y|^{N+sp}}  
	\left[
	\left(\frac{\Psi}{v^{\frac{m}{m+1}}}\right)(x)
	-\left(\frac{\Psi}{v^{\frac{m}{m+1}}}\right)(y)
	\right]\, dx\, dy  \\
	&\qquad= \int_{\Omega} g_{0}(x)\,\Psi \, dx
	+ \lambda m \int_{\Omega} d(x)^{-\gamma}
	v^{\frac{\delta-m}{m+1}}\,\Psi \, dx ,
	\end{aligned}
	\end{equation}
	for every test function $\Psi$ such that
\begin{equation}\label{equuu}
	|\Psi|\in L_{d^{m+1}}^{\infty}(\Omega)
	\quad \text{and} \quad
	\frac{|\nabla \Psi|}{d(\cdot)^{m}}\in L^{p}(\Omega).
\end{equation}
	Moreover, there exists a constant $c>0$ such that
	\[
	v(x)\geq c\, d^{m+1}(x)\qquad \text{a.e. in } \Omega .
	\]
	Furthermore, if $v_{1}$ and $v_{2}$ are two solutions of~\eqref{equ22}
	corresponding to $g_{1},g_{2}\in L^{2}(\Omega)$ satisfying \eqref{equ3},
	then the following estimate holds:
	\begin{equation}\label{1equ27}
	\left\| (v_{1}-v_{2})^{+} \right\|_{L^{2}(\Omega)}
	\leq
	\left\|
	\big( v_{1}-v_{2}
	+\lambda \big(\mathcal{T}_{m}(v_{1})-\mathcal{T}_{m}(v_{2})\big)
	\big)^{+}
	\right\|_{L^{2}(\Omega)} .
	\end{equation}
\end{corollary}

\begin{remark}\label{remark2}
	We emphasize that, in the case $m=0$, one can introduce the nonlinear operator
\begin{equation}\label{1equ1}
	\mathcal{T}_0 : D(\mathcal{T}_0) \subset L^{2}(\Omega) \longrightarrow L^{2}(\Omega),
	\qquad
	\mathcal{T}_0 w := -\Delta_{p} w + (-\Delta)_{p}^{s} w,
\end{equation}
	whose domain is defined by
	\[
	D(\mathcal{T}_0)
	:= \Big\{\, w : \Omega \to \mathbb{R}^{+} \ \text{measurable} \;\big|\;
	w \in W_0^{1,p}(\Omega) \cap L^{2}(\Omega)
	\ \text{and} \
	\mathcal{T}_0 w \in L^{2}(\Omega) \,\Big\}.
	\]
It follows from~\eqref{contraction} in the case $g_0 \in L^{\infty}(\Omega)$ and from~\eqref{equ63} in the case $g_0 \in L^{2}(\Omega)$ that the operator $\mathcal{T}_0$ is $T$-accretive in $L^{2}(\Omega)$. More precisely, for every $\lambda>0$, let $u_1$ and $u_2$ be the solutions of~\eqref{PP2} corresponding to the data $g_1$ and $g_2$, respectively. Then,
	\[
	\bigl\|(u_1-u_2)^{+}\bigr\|_{L^{2}(\Omega)}
	\leq
	\Bigl\|
	\bigl(
	u_1-u_2
	+\lambda\bigl(\mathcal{T}_0(u_1)-\mathcal{T}_0(u_2)\bigr)
	\bigr)^{+}
	\Bigr\|_{L^{2}(\Omega)}.
	\]
\end{remark}

\noindent
We now turn to the second stationary problem, which plays a fundamental role in the characterization of the asymptotic behavior of the trajectories associated with problem~\eqref{P1}.
\begin{equation}\label{P7}
\left\{
\begin{aligned}
- \Delta_{p} u + ( -\Delta ) ^{s}_{p} u &= b(x)\, u^{m} + m\, d(x)^{-\gamma} u^{\delta} 
\quad \text{in } \Omega,\\
u &> 0 \quad \text{in } \Omega,\\
u &= 0 \quad \text{in } \mathbb{R}^N \setminus \Omega.
\end{aligned}
\right.
\end{equation}
Here, $b \in L^{\infty}(\Omega) \setminus \{0\}$ is a nonnegative function.   We now introduce the notion of a weak solution.

\begin{definition}
	Let $m \ge 0$. A function $u$ such that  $ u \in W^{1,p}_0(\Omega) $  if  $ m=0, $ and $ u \in W^{1,p}_0(\Omega)\cap L^{2(m+1)}(\Omega)  $ if $ m>0, $
	with $u>0$ a.e. in $\Omega$, is said to be a  weak solution of problem~\eqref{P7} if
	\begin{equation}\label{equuuu}
	\begin{aligned}
	& \int_{\Omega} |\nabla u|^{p-2}\nabla u \cdot \nabla \varphi \, dx 
	+ \iint_{\mathbb{R}^{2N}} 
	\frac{\left[ u(x)-u(y)\right] ^{p-1}\big(\varphi(x)-\varphi(y)\big)}
	{|x-y|^{N+sp}} \, dx \, dy \\
	& \qquad = \int_{\Omega} \left( b(x)\, u^{m} + m\, d(x)^{-\gamma} u^{\delta} \right)\varphi \, dx,
	\end{aligned}
	\end{equation}
	for every test function  $	\varphi \in W^{1,p}_0(\Omega) $ if $ m=0, $
	and $ \varphi \in W^{1,p}_0(\Omega)\cap L^{2(m+1)}(\Omega) $ if $ m>0. $
\end{definition}

\begin{theorem}\label{the2}
Assume that $\gamma < \frac{1}{\textbf{d}}$. Then problem~\eqref{P7} admits a unique weak solution  $ u \in C^{1}(\overline{\Omega}) \cap \mathcal{M}^{1}_{1,\alpha}(\Omega), $ for every $\alpha \in [s,1)$ with $\alpha \neq \frac{ps}{p-1}$.
\end{theorem}

\noindent
Next, following the strategy employed in the proof of Corollary~\ref{cor1} (see below) and leveraging the boundary behavior of the weak solution to~\eqref{P7} established in Theorem~\ref{the2}, we consider the perturbed problem associated with the operator $\mathcal{T}_m$, corresponding to the problem~\eqref{P7} for $m > 0$:
\begin{equation}\label{equu22}
\begin{cases}
\mathcal{T}_m v = b, & \text{in } \Omega,\\[1mm]
v > 0, & \text{in } \Omega,\\[1mm]
v = 0, & \text{on } \mathbb{R}^N \setminus \Omega.
\end{cases}
\end{equation} 

\begin{corollary}
Under the assumptions of Theorem~\ref{the2}, problem~\eqref{equu22} admits a unique weak solution
\[
v \in \dot{V}_{+}^{m+1} \cap \mathcal{M}^{\frac{1}{m+1}}_{1,\alpha}(\Omega),
\]
 for every \( \alpha \in [s,1) \) with \( \alpha \neq \frac{ps}{p-1} \). Furthermore, this solution satisfies the integral identity
	\begin{equation*}
	\begin{aligned}
	&\int_{\Omega} \left|\nabla \left( v^{\frac{1}{m+1}} \right)\right|^{p-1} 
	\nabla \left( v^{\frac{1}{m+1}}\right) \cdot 
	\nabla \left(\frac{\Psi}{v^{\frac{m}{m+1}}}\right) \, dx \\
	&\quad + \iint_{\mathbb{R}^{2N}}
	\frac{\left[  v^{\frac{1}{m+1}}(x)-v^{\frac{1}{m+1}}(y)\right] ^{p-1}}{|x-y|^{N+sp}}  
	\left[ \left( \frac{\Psi}{v^{\frac{m}{m+1}}}\right)(x)
	- \left( \frac{\Psi}{v^{\frac{m}{m+1}}}\right)(y) \right]
	\, dx \, dy  \\
	&= \int_{\Omega} b(x)\,\Psi \, dx
	+ m \int_{\Omega} d(x)^{-\gamma} 
	v^{\frac{\delta-m}{m+1}}\,\Psi \, dx,
	\end{aligned}
	\end{equation*}
	for all test functions $\Psi$ satisfying
	\[
	|\Psi| \in L^{\infty}_{d^{m+1}}(\Omega)
	\quad \text{and} \quad
	\frac{|\nabla \Psi|}{d(\cdot)^{m}} \in L^{p}(\Omega).
	\]
\end{corollary}
\subsection{Main Results for Problem~\eqref{P1}}
We begin by presenting the main results concerning the existence, uniqueness, regularity, and qualitative properties of weak solutions to problem~\eqref{PP1}.

\begin{theorem}\label{2thm1}
Let $p>1$ and $s\in(0,1)$ be such that $p(1-s)<1$. Let $T>0$ be fixed and let $m\in[0,p-1)$. In the case $m>0$, assume that the parameters $\delta$ and $\gamma$ satisfy $ \delta\in(0,m) $ $ 2\delta - \textbf{d}\,\gamma > 2m - 1 $ for some $\textbf{d}>N$. In addition, assume that $g\in L^{\infty}(Q_T)$ satisfies condition~\eqref{equ1}, and that
\begin{equation}\label{1equ2}
u_0 \in W^{1,p}_0(\Omega)\cap \mathcal{M}^{1}_{\alpha',\alpha'}(\Omega),
\qquad
\alpha'=\frac{p}{p-1+\vartheta},
\end{equation}
where the parameter $\vartheta$ fulfills $ 1+p\Bigl(\frac{1-s}{s}\Bigr)<\vartheta<2+\frac{1}{p-1}. $ Then problem~\eqref{PP1} admits a weak solution $u$ in the sense of Definition~\ref{definition1}. Moreover, this solution enjoys the regularity property $ u\in C\bigl([0,T];L^{r}(\Omega)\bigr) $ for all  $ 1\le r<\infty, $ and there exists a constant $C>0$ such that, for every $t\in[0,T]$,
\begin{equation}\label{4equ41}
C^{-1} d(x)\le u(t,x)\le C\, d(x)^{\alpha'}
\quad \text{for a.e. } x\in\Omega.
\end{equation}
	In addition, $ u\in C\bigl([0,T];W^{1,p}_0(\Omega)\bigr), $ and, for every $t\in[0,T]$, the following energy identity holds:
	\begin{equation}\label{4equ31}
	\begin{aligned}
	&\int_{0}^{t}\!\!\int_{\Omega} \bigl(\partial_{t}u^{m+1}\bigr)^{2}\,dx\,ds+ \frac{m+1}{p}\Big(
	\|u(t)\|_{W^{1,p}_0(\Omega)}^{p}
	+ \|u(t)\|_{W^{s,p}_0(\Omega)}^{p}
	\Big) \\
	&= \int_{0}^{t}\!\!\int_{\Omega} g(s, x)\,\partial_{t}\bigl(u^{m+1}\bigr)\,dx\,ds
	+ m \int_{0}^{t}\!\!\int_{\Omega} d(x)^{-\gamma} u^{\delta - m}\,
	\partial_{t}\bigl(u^{m+1}\bigr)\,dx\,ds \\
	&\quad + \frac{m+1}{p}\Big(
	\|u(t_{0})\|_{W^{1,p}_0(\Omega)}^{p}
	+ \|u(t_{0})\|_{W^{s,p}_0(\Omega)}^{p}
	\Big).
	\end{aligned}
	\end{equation}
	Furthermore, let $u$ and $\tilde u$ be weak solutions to~\eqref{PP1}
	corresponding to initial data $u_0,\tilde u_0\in L^{2(m+1)}(\Omega)$ with
	$u_0,\tilde u_0\ge0$, and  $g,\tilde g\in L^2(Q_T)$. Then the following $T$--accretivity estimate in $L^2(\Omega)$ holds for all $t\in[0,T]$:
	\begin{equation}\label{equ34}
	\|u^{m+1}(t)-\tilde u^{m+1}(t)\|_{L^2(\Omega)}
	\le
	\|u_0^{m+1}-\tilde u_0^{m+1}\|_{L^2(\Omega)}
	+ \int_0^{t}\|g(s)-\tilde g(s)\|_{L^2(\Omega)}\,ds.
	\end{equation}
\end{theorem}
\begin{remark}
		We make the following two observations: 
		\begin{enumerate}
			\item The condition $p(1-s) < 1$ guarantees that  $ 1 + p\Bigl(\frac{1-s}{s}\Bigr) < 2 + \frac{1}{p-1}. $ As established in Theorem~\ref{theorem1}, the boundary behavior of solutions is not determined by the same power of the distance function. This discrepancy introduces challenges in deriving uniform estimates near the boundary for solutions of problem~\eqref{PP1}. To address this issue, the aforementioned condition is employed to construct an appropriate sub-solution whose behavior is suitably aligned with the distance function, thereby allowing the selection of a compatible exponent.
			
			\item In the proof of Theorem~\ref{2thm1}, we utilize a semi-discretization in time via the implicit Euler method. To obtain the required a priori estimates, we impose the additional condition  $ 2\delta - \mathbf{d}\,\gamma > 2m - 1 $ for some $\mathbf{d} > N$. This condition also ensures the desired regularity of solutions.  Although effective, this condition is not optimal.
		\end{enumerate}
\end{remark}

\noindent
The uniqueness of the weak solution, in the sense of Definition~\ref{definition1} and as established in Theorem~\ref{2thm1}, follows directly from~\eqref{equ34}.  We emphasize that this uniqueness result holds without imposing the restriction $p(1-s)<1$ and under weaker assumptions on the initial datum $u_{0}$ and the source term $g$. 
More precisely, we have the following result.

\begin{corollary}\label{cor3}
	Let $u$ and $\tilde u$ be weak solutions of problem~\eqref{PP1} in the sense of Definition~\ref{definition1}, corresponding to the same nonnegative initial datum $u_0 \in L^{2(m+1)}(\Omega)$ and source term $g \in L^2(Q_T)$. Then,  $ u \equiv \tilde u $ in $ Q_T.	 $
\end{corollary}

\noindent
On the other hand, by combining Theorem~\ref{2thm1} with Corollary~\ref{cor3}, we obtain the following existence result for the parabolic problem~\eqref{0equ4} involving the operator $\mathcal{T}_{m}$, with $m>0$.

\begin{theorem}\label{thm6}
Under the assumptions of Theorem~\ref{2thm1}, let $v_0$ be such that  $ v_0^{\frac{1}{m+1}} \in W^{1,p}_0(\Omega) \cap \mathcal{M}^{1}_{\alpha',\alpha'}(\Omega), $ where 
	\[
	\alpha' = \frac{p}{\,p - 1 + \vartheta\,}, \quad \text{with} \quad \vartheta > p \left(\frac{1-s}{s}\right),
	\] 
	which ensures that $\alpha' \in (0,s)$. Then there exists a unique weak solution $v\in L^{\infty}(Q_T)$ to problem~\eqref{0equ4}. In particular, $ v^{1/m+1} \in L^{\infty}\bigl(0,T; W^{1,p}_0(\Omega)\bigr), $ and $ \partial_t v \in L^2(Q_T), $ and there exists  $C>0$ such that, for any $t\in[0,T]$,
	\[
	C^{-1} d(x) \le v^{1/m+1}(t,x) \le C\, d(x)^{\alpha'}
	\quad \text{a.e. in } \Omega.
	\]
	Moreover, for any $t\in[0,T]$, the function $v$ satisfies
	\begin{align*}
	&\int_0^{t}\!\!\int_{\Omega} \partial_t v \,\Psi \,dx\,ds
	+ \int_0^{t}\!\!\int_{\Omega}
	\left|\nabla\!\left(v^{\frac{1}{m+1}}\right)\right|^{p-1}
	\nabla\!\left(v^{\frac{1}{m+1}}\right)\!\cdot\!
	\nabla\!\left(\frac{\Psi}{v^{\frac{m}{m+1}}}\right) dx\,ds  \\
	&\quad
	+ \int_0^{t}\!\!\iint_{\mathbb{R}^{2N}}
	\frac{\left[ v^{\frac{1}{m+1}}(x)-v^{\frac{1}{m+1}}(y)\right] ^{p-1}}
	{|x-y|^{N+sp}}
	\left[
	\left(\frac{\Psi}{v^{\frac{m}{m+1}}}\right)(x)
	-\left(\frac{\Psi}{v^{\frac{m}{m+1}}}\right)(y)
	\right]\,dx\,dy\,ds \\
	&= \int_0^{t}\!\!\int_{\Omega} g(x,s)\,\Psi \,dx\,ds
	+ m \int_0^{t}\!\!\int_{\Omega} d(x)^{-\gamma}
	v^{\frac{\delta-m}{m+1}}\,\Psi \,dx\,ds,
	\end{align*}
	for any test function $\Psi\in L^2(Q_T)$ such that
	\[
	|\Psi| \in L^{\infty}\bigl(0,T; L^{\infty}_{d^{m+1}}(\Omega)\bigr)
	\quad \text{and} \quad
	\frac{|\nabla \Psi|}{d(\cdot)^m} \in L^{1}\bigl(0,T; L^{p}(\Omega)\bigr).
	\]
	Finally, for any $1\le r<\infty$, the solution $v$ belongs to $ C\bigl([0,T]; L^r(\Omega)\bigr). $
\end{theorem}

\noindent
In the particular case $m=0$, by invoking the theory of maximal accretive operators in Banach spaces and following the approach developed in \cite[Theorems~4.2 and~4.4]{barbu1976nonlinear}, one can derive the following regularity result for solutions of problem~\eqref{PP1}.  
To this end, we introduce the operator
\[
\mathcal{T}_0^{\star} : D^{\star}(\mathcal{T}_0^{\star}) \subset L^{\infty}(\Omega) \longrightarrow L^{\infty}(\Omega),
\qquad
\mathcal{T}_0^{\star} w := -\Delta_{p} w + (-\Delta)_{p}^{s} w,
\]
whose domain is defined by
\[
D^{\star}(\mathcal{T}_0^{\star})
:= \Big\{\, w : \Omega \to [0,\infty) \ \text{measurable} \;\big|\;
w \in W_0^{1,p}(\Omega) \cap L^{\infty}(\Omega)
\ \text{and} \
\mathcal{T}_0^{\star} w \in L^{\infty}(\Omega) \,\Big\}.
\]

\begin{theorem}\label{thereom1}
	Let $m=0$, $T>0$, and assume that $g\in L^{\infty}(Q_T)$. Suppose further that $u_0 \in \overline{D^{\star}(\mathcal{T}_0^{\star})}^{\,L^{\infty}}$. Then the following assertions hold:
	\begin{itemize}
		\item[(i)] The unique weak solution $u$ to problem~\eqref{PP1}, obtained in Theorem~\ref{2thm1}, satisfies	$ u \in C\big([0,T]; C_0(\overline{\Omega})\big). $
		\item[(ii)] Let $\tilde{u}$ be another mild solution to~\eqref{PP1} with $\tilde{u}_0 \in D^{\star}(\mathcal{T}_0^{\star})\cap L^{\infty}(\Omega)$ and right-hand side $\tilde{g}\in L^{\infty}(Q_T)$. Then, for all $t\in[0,T]$, the following stability estimate holds:
		\begin{equation*}
		\|u(t)-\tilde{u}(t)\|_{L^{\infty}(\Omega)}
		\le
		\|u_0-\tilde{u}_0\|_{L^{\infty}(\Omega)}
		+ \int_0^{t}\|g(s)-\tilde{g}(s)\|_{L^{\infty}(\Omega)}\,ds.
		\end{equation*}
		\item[(iii)] If $u_0\in D^{\star}(\mathcal{T}^{\star}_0)$, $g\in W^{1,1}(0,T;L^{\infty}(\Omega))$, then $ 	u\in W^{1,\infty}(0,T;L^{\infty}(\Omega)),
		\,
		-\Delta_p u + (-\Delta)^s_p u \in L^{\infty}(Q_T), $	and the following estimate is satisfied:
		\begin{equation*}
		\left\| \partial_t u(\cdot,t) \right\|_{L^{\infty}(\Omega)}
		\le
		\left\| -\Delta_p u_0 + (-\Delta)^s_p u_0 + g(\cdot,0) \right\|_{L^{\infty}(\Omega)}
		+ \int_0^T \left\| \partial_t g(\cdot,s) \right\|_{L^{\infty}(\Omega)}\,ds.
		\end{equation*}
	\end{itemize}
\end{theorem}

\noindent
Using the $T$-accretivity of $\mathcal{T}_m$ in $L^2(\Omega)$ for $m \ge 0$  (cf.~\eqref{1equ1} for $m=0$, \eqref{equ37} for $m>0$) and under additional regularity of the initial data, 
we obtain the following stabilization result for the weak solutions of \eqref{PP1}.

\begin{theorem} \label{the3}
	Assume that the hypotheses of Theorem \ref{2thm1} are satisfied for any $T>0$.  
	Let $u$ denote the weak solution of \eqref{PP1} with initial data   $ u_0 \in W^{1,p}_0(\Omega)\cap \mathcal{M}^{1}_{\alpha',\alpha'}(\Omega), $  where $\alpha' \in (0,s)$ is defined by \eqref{1equ2}. Furthermore, assume that there exists a function $g_{\infty} \in L^{\infty}(\Omega)$ such that, for some $t_{0} \geq 0$,
	\begin{equation}\label{equ51}
	\exp(t - t_{0})\,\| g(t, \cdot) - g_{\infty}\|_{L^2(\Omega)} = O(1) \quad \text{as } t \to \infty.
	\end{equation}
	Then, for any \(r \geq 1\) and \(m \geq 0\), we have
	\begin{equation*}
	\| u^{m+1}(t, \cdot) - u_{\infty}^{m+1} \|_{L^r(\Omega)} \longrightarrow 0 
	\quad \text{as } t \to \infty,
	\end{equation*}
	where \(u_{\infty}\) denotes the unique stationary solution of \eqref{P7} corresponding to the potential \(b = g_{\infty}\).
\end{theorem}

\noindent
The article is organized as follows:\\[4pt]
 In \textbf{Section}~\ref{sec1}, we provide rigorous proofs of the main results for the stationary problems~\eqref{PP2} and~\eqref{P7}, which are associated with the parabolic problem~\eqref{PP1}. \textbf{Section}~\ref{sec2} is devoted to the proofs of the main results concerning problem~\eqref{PP1}.

\section{Stationary Problems Associated with Problem \eqref{PP1}}\label{sec1}
In this section, we provide rigorous proofs of the main results previously stated, which are concerned with the stationary problems corresponding to \eqref{PP1}. In particular, we establish the following fundamental results.
\subsection{Weak Comparison Principle for Problem~\eqref{PP2}}

\begin{proof}[\textbf{Proof of Theorem} \ref{theorem2}] 
	For any nonnegative functions $\Psi, \Phi \in \textbf{WL} $, the following inequalities hold:
	\begin{equation}\label{equ5}
	\begin{aligned}
	& \int_{\Omega} \underline{u}^{2m+1}\, \Psi \, dx 
	+ \lambda \int_{\Omega} |\nabla \underline{u}|^{p-2} \nabla \underline{u} \cdot \nabla \Psi \, dx  + \lambda \, \iint_{\mathbb{R}^{2N}} 
	\frac{\left[ \underline{u}(x)-\underline{u}(y)\right] ^{p-1} \big(\Psi(x)-\Psi(y)\big)}{|x-y|^{N+sp}} \, dx \, dy \\
	& \qquad \leq \int_{\Omega} \Big( g_{0}(x)\, \underline{u}^{m} + \lambda m\, d(x)^{-\gamma}\, \underline{u}^{\delta} \Big) \Psi \, dx,
	\end{aligned}
	\end{equation}
	and
	\begin{equation}\label{equ6}
	\begin{aligned}
	& \int_{\Omega} \overline{u}^{2m+1}\, \Phi \, dx 
	+ \lambda \int_{\Omega} |\nabla \overline{u}|^{p-2} \nabla \overline{u} \cdot \nabla \Phi \, dx  + \lambda \, \iint_{\mathbb{R}^{2N}} 
	\frac{\left[ \overline{u}(x)-\overline{u}(y)\right] ^{p-1} \big(\Phi(x)-\Phi(y)\big)}{|x-y|^{N+sp}} \, dx \, dy \\
	& \qquad \geq \int_{\Omega} \Big( g_{0}(x)\, \overline{u}^{m} + \lambda m\, d(x)^{-\gamma}\, \overline{u}^{\delta} \Big) \Phi \, dx.
	\end{aligned}
	\end{equation}
	We begin by considering the case \(m=0\). Subtracting the previously mentioned inequalities and choosing the test functions $ 	\Psi(x) = \big(\underline{u}(x) - \overline{u}(x)\big)^{+},\,  \Phi(x) = \big(\overline{u}(x) - \underline{u}(x)\big)^{-}  \in W^{1,p}_{0}(\Omega) \cap L^{2(m+1)}(\Omega), $ we obtain
\begin{equation*}
		\begin{aligned}
		& \int_{\{ \underline{u} > \overline{u} \}} 
		\big( \underline{u} - \overline{u} \big)^{2} \, dx + \lambda \int_{\Omega} 
		\Big( |\nabla \underline{u}|^{p-2} \nabla \underline{u} 
		- |\nabla \overline{u}|^{p-2} \nabla \overline{u} \Big)  
		\cdot \nabla \big(\underline{u} - \overline{u}\big)^{+} \, dx  \\
		&  + \lambda \iint_{\mathbb{R}^{2N}} 
		\frac{ \Big( \big[\underline{u}(x)-\underline{u}(y)\big]^{p-1} 
			- \big[\overline{u}(x)-\overline{u}(y)\big]^{p-1} \Big) 
			\Big( \big(\underline{u}-\overline{u}\big)^{+}(x) 
			- \big(\underline{u}-\overline{u}\big)^{+}(y) \Big) }
		{|x-y|^{N+sp}}  dx dy \\
		& \leq 0.
		\end{aligned}
		\end{equation*}
	
	\noindent 
	By applying the classical elliptic estimates from \cite{pucci2007maximum}, we infer that the local term is nonnegative. Moreover, since the mapping $x \mapsto x^{+}$ is monotone nondecreasing, the fractional term is likewise nonnegative. Hence, we arrive at
	\begin{equation*}
	\int_{\{ \underline{u} > \overline{u} \}} 
	\big( \underline{u} - \overline{u} \big)^{2} \, dx  = 0,
	\end{equation*}
	which yields $ \underline{u} \leq \overline{u} $ a. e. in $\Omega$. We now consider the case $m > 0$. By subtracting \eqref{equ6} from \eqref{equ5} and introducing the test functions
 \[
		\Psi_{\epsilon, k} = \frac{\mathbf{T}_{k}\Big( \big((\underline{u} + \epsilon)^{m+1} - (\overline{u} + \epsilon)^{m+1}\big)^{+} \Big)}{(\underline{u} + \epsilon)^{m}}, 
		\,
		\Phi_{\epsilon, k} = \frac{\mathbf{T}_{k}\Big( \big((\overline{u} + \epsilon)^{m+1} - (\underline{u} + \epsilon)^{m+1}\big)^{-} \Big)}{(\overline{u} + \epsilon)^{m}},
		\]
	
	\noindent 
	for fixed $\epsilon \in (0,1)$ and $k > 0$, and observing that $\Psi_{\epsilon, k}, \Phi_{\epsilon, k} \in W^{1,p}_{0}(\Omega) \cap L^{2(m+1)}(\Omega)$ (see \cite[Remark~2.5]{durastanti2022comparison}), we obtain
\begin{equation}\label{equ8}
		\begin{aligned}
		& \int_{\Omega} \left( \underline{u}^{2m+1}\, \Psi_{\epsilon, k} - \overline{u}^{2m+1}\, \Phi_{\epsilon, k} \right) dx 
		+ \lambda \int_{\Omega} \left( |\nabla \underline{u}|^{p-2} \nabla \underline{u} \cdot \nabla \Psi_{\epsilon, k} 
		- |\nabla \overline{u}|^{p-2} \nabla \overline{u} \cdot \nabla \Phi_{\epsilon, k} \right) dx \\
		&+ \lambda \iint_{\mathbb{R}^{2N}} 
		\left( 
		\frac{\left[ \underline{u}(x)-\underline{u}(y)\right] ^{p-1} \big(\Psi_{\epsilon, k}(x)-\Psi_{\epsilon, k}(y)\big) - \left[ \overline{u}(x)-\overline{u}(y)\right] ^{p-1} \big(\Phi_{\epsilon, k}(x)-\Phi_{\epsilon, k}(y)\big)}{|x-y|^{N+sp}} \right) dxdy \\
		& \quad \leq \int_{\Omega} g_{0}(x)\, \big(\underline{u}^{m} \Psi_{\epsilon, k} - \overline{u}^{m} \Phi_{\epsilon, k}\big) dx 
		+ \lambda \, m \, \int_{\Omega} d(x)^{-\gamma} \big(\underline{u}^{\delta} \Psi_{\epsilon, k} - \overline{u}^{\delta} \Phi_{\epsilon, k}\big) dx .
		\end{aligned}
		\end{equation}
	
	\noindent 
	Consequently, we obtain the following estimates: \\[4pt]
	Firstly, observe that
	\[
	\dfrac{\overline{u}^{2m+1}}{\left( \overline{u}+ \epsilon\right)^{m}} \leq 
	\dfrac{\underline{u}^{2m+1}}{\left( \underline{u}+ \epsilon\right)^{m}} 
	\quad \text{in } \{ \underline{u} > \overline{u} \}.
	\]
	Moreover, noting that $\mathbf{T}_{k}(t) \leq t$ for all $t \geq 0$, we deduce on the set $\{ \underline{u} > \overline{u} \}$ that
	\begin{equation*}
	\begin{aligned}
	0 &\leq \underline{u}^{2m+1}\, \Psi_{\epsilon, k} - \overline{u}^{2m+1}\, \Phi_{\epsilon, k} \\
	&= \left( \dfrac{\underline{u}^{2m+1}}{(\underline{u}+ \epsilon)^{m}} -  
	\dfrac{\overline{u}^{2m+1}}{(\overline{u}+ \epsilon)^{m}} \right) 
	\mathbf{T}_{k}\!\Big( (\underline{u} + \epsilon)^{m+1} - (\overline{u} + \epsilon)^{m+1} \Big) \\[2mm]
	&\leq \left( \underline{u}^{m+1}\left(\dfrac{\underline{u}}{\underline{u} + \epsilon}\right)^{m} 
	- \overline{u}^{m+1}\left(\dfrac{\overline{u}}{\overline{u} + \epsilon}\right)^{m}\right) 
	\Big((\underline{u} + \epsilon)^{m + 1} - (\overline{u} + \epsilon)^{m + 1}\Big) \\[1mm]
	&\leq \underline{u}^{m+1}\,\Big((\underline{u} + \epsilon)^{m + 1} - (\overline{u} + \epsilon)^{m + 1}\Big) 
	\leq \underline{u}^{m+1}\, (\underline{u} + 1)^{m + 1} \in L^{1}(\{ \underline{u} > \overline{u} \}).
	\end{aligned}
	\end{equation*}
	Passing to the limit as $\epsilon \to 0$, we obtain
	\begin{equation*}
	\begin{aligned}
	&\left( \dfrac{\underline{u}^{2m+1}}{(\underline{u}+ \epsilon)^{m}} -  
	\dfrac{\overline{u}^{2m+1}}{(\overline{u}+ \epsilon)^{m}} \right) 
	\mathbf{T}_{k}\!\Big( (\underline{u} + \epsilon)^{m+1} - (\overline{u} + \epsilon)^{m+1} \Big)  \\[1mm]
	&\quad \longrightarrow \; 
	\big(\underline{u}^{m+1} - \overline{u}^{m+1}\big)\,
	\mathbf{T}_{k}\!\big(\underline{u}^{m+1} - \overline{u}^{m+1}\big).
	\end{aligned}
	\end{equation*}
	Hence, by the Dominated Convergence Theorem, it follows that
\begin{equation*}
		\begin{aligned}
		\lim_{\epsilon \to 0} \int_{\{ \underline{u} > \overline{u} \}} 
		\big( \underline{u}^{2m+1} \Psi_{\epsilon, k} - \overline{u}^{2m+1} \Phi_{\epsilon, k} \big) dx 
		= \int_{\{ \underline{u} > \overline{u} \}} 
		\big( \underline{u}^{m+1} - \overline{u}^{m+1} \big)
		\mathbf{T}_{k}\!\big( \underline{u}^{m+1} - \overline{u}^{m+1} \big) dx.
		\end{aligned}
		\end{equation*}
	
	\noindent 
	In addition, consider the sequence
	\[
	U_{k} := \big( \underline{u}^{m+1} - \overline{u}^{m+1} \big)\,
	\mathbf{T}_{k}\!\big( \underline{u}^{m+1} - \overline{u}^{m+1} \big) \;\geq\; 0 
	\quad \text{in } \{ \underline{u} > \overline{u} \}.
	\]
	Clearly, $(U_{k})$ is nondecreasing. Hence, by the Monotone Convergence Theorem, we obtain
	\begin{equation*}
	\lim_{k \to +\infty}\int_{\{ \underline{u} > \overline{u} \}} 
	U_{k}\, dx 
	=
	\int_{\{ \underline{u} > \overline{u} \}} 
	\big( \underline{u}^{m+1} - \overline{u}^{m+1} \big)^{2}\, dx.
	\end{equation*}
	Hence, by passing to the limits $\epsilon \to 0$ and $k \to \infty$, we obtain
	\begin{equation}\label{equ7}
	\begin{aligned}
	\int_{\{ \underline{u} > \overline{u} \}} 
	\Big( \underline{u}^{2m+1}\, \Psi_{\epsilon, k} - \overline{u}^{2m+1}\, \Phi_{\epsilon, k} \Big)\, dx 
	\longrightarrow \int_{\{ \underline{u} > \overline{u} \}} 
	\Big( \underline{u}^{m+1} - \overline{u}^{m+1} \Big)^{2} \, dx.
	\end{aligned}
	\end{equation}
	Analogously, by observing that 
	\[
	\dfrac{\overline{u}}{\overline{u}+ \epsilon} \leq 
	\dfrac{\underline{u}}{ \underline{u}+ \epsilon}  < 1
	\quad \text{in } \{ \underline{u} > \overline{u} \},
	\]
	we obtain
	\begin{equation*}
	\begin{aligned}
	0 &\leq g_{0}(x)\,\big(\underline{u}^{m}\, \Psi_{\epsilon, k} - \overline{u}^{m}\, \Phi_{\epsilon, k}\big) \\[4pt]
	&= g_{0}(x)\,\left(\left(\dfrac{\underline{u}}{\underline{u} + \epsilon}\right)^{m}   
	- \left(\dfrac{\overline{u}}{\overline{u} + \epsilon}\right)^{m} \right)\,
	\mathbf{T}_{k}\!\Big((\underline{u} + \epsilon)^{m + 1} - (\overline{u} + \epsilon)^{m + 1}\Big) \\[4pt]
	&\leq g_{0}(x)\,\left(\left(\dfrac{\underline{u}}{\underline{u} + \epsilon}\right)^{m}   
	- \left(\dfrac{\overline{u}}{\overline{u} + \epsilon}\right)^{m} \right)\,
	\Big((\underline{u} + \epsilon)^{m + 1} - (\overline{u} + \epsilon)^{m + 1}\Big) \\[4pt]
	&\leq g_{0}(x)\, (\underline{u} + 1)^{m + 1} \;\in\; L^{1}(\{ \underline{u} > \overline{u} \}).
	\end{aligned}
	\end{equation*}
	Moreover, we have
	\begin{equation*}
	\begin{aligned}
	\lim_{\epsilon \to 0} \left( g_{0}(x)\, \left(\left(\dfrac{\underline{u}}{\underline{u} + \epsilon}\right)^{m}   
	- \left(\dfrac{\overline{u}}{\overline{u} + \epsilon}\right)^{m} \right)\,
	\mathbf{T}_{k}\!\Big((\underline{u} + \epsilon)^{m + 1} - (\overline{u} + \epsilon)^{m + 1}\Big)\right)  = 0,
	\end{aligned}
	\end{equation*}
	and consequently, by the Dominated Convergence Theorem, we conclude that
	\begin{equation}\label{equ21}
	\lim_{k \to \infty}  \lim_{\epsilon \to 0}\int_{\{ \underline{u} > \overline{u} \}} 
	g_{0}(x)\, \big(\underline{u}^{m} \Psi_{\epsilon, k} - \overline{u}^{m} \Phi_{\epsilon, k}\big)\, dx  = 0.
	\end{equation}
	Now, by applying Lagrange's Theorem, we deduce that
	\begin{equation*}
	\begin{gathered}
	\begin{aligned}
	&  d(x)^{-\gamma}  \left(\dfrac{\underline{u}^{\delta}}{\left(\underline{u} + \epsilon\right) ^{m}} - \dfrac{\overline{u}^{\delta}}{\left(\overline{u} + \epsilon\right) ^{m}} \right)^{+} \textbf{T}_{k} \left(\left(\underline{u} + \epsilon\right) ^{m+1} - \left( \overline{u} + \epsilon\right) ^{m+1} \right) \\
	& \leq d(x)^{-\gamma} \, \underline{u}^{\delta} \left( \dfrac{\left(\underline{u} + \epsilon\right) ^{m+1} - \left( \overline{u} + \epsilon\right) ^{m+1}}{\left(\underline{u} + \epsilon\right) ^{m}}\right)   \\
	& \leq  d(x)^{-\gamma} \, \underline{u}^{\delta} \,\left( \dfrac{\left(\underline{u} + \epsilon\right) ^{m+1} - \epsilon ^{m+1}}{\left(\underline{u} + \epsilon\right) ^{m}}\right) \leq (m+1) d(x)^{-\gamma} \underline{u}^{\delta +1} \in L^{1}(\{ \underline{u} > \overline{u} \}).
	\end{aligned}
	\end{gathered}
	\end{equation*}
	Invoking the Dominated Convergence Theorem once more, we obtain that, as $\epsilon \to 0$, 
	\begin{equation}\label{equ18}
	\begin{aligned}
	&\int_{\{\underline{u} > \overline{u}\}} d(x)^{-\gamma}  
	\left(
	\dfrac{\underline{u}^{\delta}}{(\underline{u} + \epsilon)^{m}} - \dfrac{\overline{u}^{\delta}}{(\overline{u} + \epsilon)^{m}} 
	\right)^{+} 
	\mathbf{T}_{k} 
	\Big((\underline{u} + \epsilon)^{m+1} - (\overline{u} + \epsilon)^{m+1} \Big) \, dx \\[2mm]
	&\quad \longrightarrow 
	\int_{\{\underline{u} > \overline{u}\}} d(x)^{-\gamma}  
	\left(\underline{u}^{\delta - m} - \overline{u}^{\delta - m}\right)^{+} 
	\mathbf{T}_{k} \Big( \underline{u}^{m+1} - \overline{u}^{m+1} \Big) \, dx.
	\end{aligned}
	\end{equation}
	
	\noindent
	On the other hand, by applying Fatou’s lemma, we obtain
		\begin{equation}\label{equ19}
		\begin{gathered}
		\begin{aligned}
		\liminf_{\epsilon \to 0} 
		&{\displaystyle\int_{\{\underline{u} > \overline{u}\}} d(x)^{-\gamma}
			\left(
			\dfrac{\underline{u}^{\delta}}{(\underline{u} + \epsilon)^{m}}
			- \dfrac{\overline{u}^{\delta}}{(\overline{u} + \epsilon)^{m}}
			\right)^{-}
			\mathbf{T}_{k}\!\left(
			(\underline{u} + \epsilon)^{m+1} - (\overline{u} + \epsilon)^{m+1}
			\right) dx } \\[4pt]
		&\geq {\displaystyle\int_{\{\underline{u} > \overline{u}\}}}
		d(x)^{-\gamma}
		\left( \underline{u}^{\delta - m} - \overline{u}^{\delta - m} \right)^{-}
		\mathbf{T}_{k}\!\left(
		\underline{u}^{m+1} - \overline{u}^{m+1}
		\right) dx.
		\end{aligned}
		\end{gathered}
		\end{equation}
	
	\noindent
	Taking the $\liminf$ as $ \epsilon \to 0 $ and using \eqref{equ18} and \eqref{equ19}, we get that
	\begin{equation*}
	\begin{gathered}
	\begin{aligned}
	\liminf_{\epsilon \to 0}&
	\int_{\{\underline{u} > \overline{u}\}} d(x)^{-\gamma}
	\big(\underline{u}^{\delta} \Psi_{\epsilon, k}
	- \overline{u}^{\delta} \Phi_{\epsilon, k}\big) dx\\
	&\leq {\displaystyle\int_{\{\underline{u} > \overline{u}\}}}
	d(x)^{- \gamma}
	\left(\underline{u}^{\delta - m} - \overline{u}^{\delta - m}\right)
	\mathbf{T}_{k}\!\left(
	\underline{u}^{m+1} - \overline{u}^{m+1}
	\right) dx.
	\end{aligned}
	\end{gathered}
	\end{equation*}
	
	\noindent
	Moreover, by the monotonicity property, we have
	\[
	U_{k} := - d(x)^{- \gamma}
	\left(\underline{u}^{\delta - m} - \overline{u}^{\delta - m}\right)
	\mathbf{T}_{k}\!\left(
	\underline{u}^{m+1} - \overline{u}^{m+1}
	\right) \geq 0 \quad \text{in } \{ \underline{u} > \overline{u} \}.
	\]
	Since $(U_{k}) $ is an increasing sequence, the Monotone Convergence Theorem yields
	\[
	\lim_{k \to \infty} \int_{\Omega} U_{k} \, dx
	= - \int_{\Omega} d(x)^{- \gamma}
	\left(\underline{u}^{\delta - m} - \overline{u}^{\delta - m}\right)
	(\underline{u}^{m+1} - \overline{u}^{m+1})^{+} \, dx.
	\]
	Consequently, since $ \delta < m $, we obtain
	\begin{equation}\label{equ20}
	\begin{gathered}
	\begin{aligned}
	\lim_{k \to \infty} \liminf_{\epsilon \to 0}&
	\int_{\Omega} d(x)^{-\gamma}
	\big(\underline{u}^{\delta} \Psi_{\epsilon, k}
	- \overline{u}^{\delta} \Phi_{\epsilon, k}\big) dx\\
	&\leq
	\int_{\Omega} d(x)^{- \gamma}
	\left(\underline{u}^{\delta - m} - \overline{u}^{\delta - m}\right)
	(\underline{u}^{m+1} - \overline{u}^{m+1})^{+} \, dx
	\leq 0.
	\end{aligned}
	\end{gathered}
	\end{equation}
	Secondly, for the remaining estimates, we adapt the argument used in the proof of Claim~2.3 of \cite[Theorem~1.9]{gouasmia2024uniqueness}. To this end, we introduce the following subdomains:
	\begin{equation*}
	\begin{gathered}
	\begin{aligned}
	& \Omega_{1} = \left\lbrace x \in \Omega : \, \left(\underline{u} + \epsilon\right)^{m+1}(x) - \left(\overline{u}  + \epsilon \right)^{m+1}(x) \leq 0 \right\rbrace, \\[4pt]
	& \Omega_{2} = \left\lbrace x \in \Omega : \, 0 <\left(\underline{u} + \epsilon\right)^{m+1}(x) - \left(\overline{u}  + \epsilon \right)^{m+1}(x)  < k \right\rbrace, \\[4pt]
	& \Omega_{3} = \left\lbrace x \in \Omega : \, \left(\underline{u} + \epsilon\right)^{m+1}(x) - \left(\overline{u}  + \epsilon \right)^{m+1}(x)  \geq k \right\rbrace.
	\end{aligned}
	\end{gathered}
	\end{equation*}
	Hence, we derive the following expression:
	
\begin{equation*}
		\begin{aligned}
		& \int_{\Omega} \left( |\nabla \underline{u}|^{p-2} \nabla \underline{u} \cdot \nabla \Psi_{\epsilon, k} 
		- |\nabla \overline{u}|^{p-2} \nabla \overline{u} \cdot \nabla \Phi_{\epsilon, k} \right) dx \\
		&+ \iint_{\mathbb{R}^{2N}} 
		\left( 
		\frac{\left[ \underline{u}(x)-\underline{u}(y)\right] ^{p-1} \big(\Psi_{\epsilon, k}(x)-\Psi_{\epsilon, k}(y)\big) - \left[ \overline{u}(x)-\overline{u}(y)\right] ^{p-1} \big(\Phi_{\epsilon, k}(x)-\Phi_{\epsilon, k}(y)\big)}{|x-y|^{N+sp}} \right) dxdy \\
		& = \int_{\Omega_{1}} \left( |\nabla \underline{u}|^{p-2} \nabla \underline{u} \cdot \nabla \Psi_{\epsilon, k} 
		- |\nabla \overline{u}|^{p-2} \nabla \overline{u} \cdot \nabla \Phi_{\epsilon, k} \right) dx \\
		&+\int_{\Omega_{2}} \left( |\nabla \underline{u}|^{p-2} \nabla \underline{u} \cdot \nabla \Psi_{\epsilon, k} 
		- |\nabla \overline{u}|^{p-2} \nabla \overline{u} \cdot \nabla \Phi_{\epsilon, k} \right) dx\\
		&+\int_{\Omega_{3}} \left( |\nabla \underline{u}|^{p-2} \nabla \underline{u} \cdot \nabla \Psi_{\epsilon, k} 
		- |\nabla \overline{u}|^{p-2} \nabla \overline{u} \cdot \nabla \Phi_{\epsilon, k} \right) dx \\
		&+\iint_{\Omega_{1} \times \Omega_{1}} 
		\left( 
		\frac{\left[ \underline{u}(x)-\underline{u}(y)\right] ^{p-1} \big(\Psi_{\epsilon, k}(x)-\Psi_{\epsilon, k}(y)\big) - \left[ \overline{u}(x)-\overline{u}(y)\right] ^{p-1} \big(\Phi_{\epsilon, k}(x)-\Phi_{\epsilon, k}(y)\big)}{|x-y|^{N+sp}} \right) dxdy \\
		& +\iint_{\Omega_{2} \times \Omega_{2}} 
		\left( 
		\frac{\left[ \underline{u}(x)-\underline{u}(y)\right] ^{p-1} \big(\Psi_{\epsilon, k}(x)-\Psi_{\epsilon, k}(y)\big) - \left[ \overline{u}(x)-\overline{u}(y)\right] ^{p-1} \big(\Phi_{\epsilon, k}(x)-\Phi_{\epsilon, k}(y)\big)}{|x-y|^{N+sp}} \right) dxdy \\
		&+\iint_{\Omega_{3} \times \Omega_{3}} 
		\left( 
		\frac{\left[ \underline{u}(x)-\underline{u}(y)\right] ^{p-1} \big(\Psi_{\epsilon, k}(x)-\Psi_{\epsilon, k}(y)\big) - \left[ \overline{u}(x)-\overline{u}(y)\right] ^{p-1} \big(\Phi_{\epsilon, k}(x)-\Phi_{\epsilon, k}(y)\big)}{|x-y|^{N+sp}} \right) dxdy \\
		&+2 \iint_{\Omega_{2} \times \Omega_{1}} 
		\left( 
		\frac{\left[ \underline{u}(x)-\underline{u}(y)\right] ^{p-1} \big(\Psi_{\epsilon, k}(x)-\Psi_{\epsilon, k}(y)\big) - \left[ \overline{u}(x)-\overline{u}(y)\right] ^{p-1} \big(\Phi_{\epsilon, k}(x)-\Phi_{\epsilon, k}(y)\big)}{|x-y|^{N+sp}} \right) dxdy \\
		& +2 \iint_{\Omega_{3} \times \Omega_{1}} 
		\left( 
		\frac{\left[ \underline{u}(x)-\underline{u}(y)\right] ^{p-1} \big(\Psi_{\epsilon, k}(x)-\Psi_{\epsilon, k}(y)\big) - \left[ \overline{u}(x)-\overline{u}(y)\right] ^{p-1} \big(\Phi_{\epsilon, k}(x)-\Phi_{\epsilon, k}(y)\big)}{|x-y|^{N+sp}} \right) dxdy \\
		&+2 \iint_{\Omega_{3} \times \Omega_{2}} 
		\left( 
		\frac{\left[ \underline{u}(x)-\underline{u}(y)\right] ^{p-1} \big(\Psi_{\epsilon, k}(x)-\Psi_{\epsilon, k}(y)\big) - \left[ \overline{u}(x)-\overline{u}(y)\right] ^{p-1} \big(\Phi_{\epsilon, k}(x)-\Phi_{\epsilon, k}(y)\big)}{|x-y|^{N+sp}} \right) dxdy.
		\end{aligned}
		\end{equation*}
	
	\noindent
	We first apply Lemma~\ref{Lem2} to obtain the following inequality:
		\begin{equation}\label{equ17}
		\begin{aligned}
		& \int_{\Omega_{2}} |\nabla \underline{u}|^{p-2} \nabla \underline{u} \cdot
		\nabla \!\left( \frac{\underline{u}^{m+1} - \overline{u}^{m+1}}{\underline{u}^{m}} \right) dx 
		+ \int_{\Omega_{2}} |\nabla \overline{u}|^{p-2} \nabla \overline{u} \cdot
		\nabla \!\left( \frac{\overline{u}^{m+1} - \underline{u}^{m+1}}{\overline{u}^{m}} \right) dx \\[2mm]
		& \quad + \iint_{\Omega_{2} \times \Omega_{2}} 
		\dfrac{\big[ \underline{u}(x) - \underline{u}(y) \big]^{p-1}}{|x-y|^{N+sp}}
		\left( \frac{\underline{u}(x)^{m+1} - \overline{u}(x)^{m+1}}{\underline{u}(x)^{m}} 
		- \frac{\underline{u}(y)^{m+1} - \overline{u}(y)^{m+1}}{\underline{u}(y)^{m}} \right) dxdy \\[1mm]
		& \quad + \iint_{\Omega_{2} \times \Omega_{2}} 
		\dfrac{\big[ \overline{u}(x) - \overline{u}(y) \big]^{p-1}}{|x-y|^{N+sp}}
		\left( \frac{\overline{u}(x)^{m+1} - \underline{u}(x)^{m+1}}{\overline{u}(x)^{m}} 
		- \frac{\overline{u}(y)^{m+1} - \underline{u}(y)^{m+1}}{\overline{u}(y)^{m}} \right) dxdy 
		\geq 0.
		\end{aligned}
		\end{equation}
	
	\noindent
	Next, we estimate each of the remaining integrals separately. First, we observe that
	\[
	\Psi_{\epsilon, k}(x) = \Phi_{\epsilon, k}(x) = 0 \quad \text{for all } x \in \Omega_{1},
	\]
	which immediately implies that
	\begin{equation}\label{equ11}
	\int_{\Omega_{1}} 
	\left( 
	|\nabla \underline{u}|^{p-2} \nabla \underline{u} \cdot \nabla \Psi_{\epsilon, k} 
	- |\nabla \overline{u}|^{p-2} \nabla \overline{u} \cdot \nabla \Phi_{\epsilon, k} 
	\right) dx = 0,
	\end{equation}
	and
		\begin{equation}\label{equ12}
		\begin{aligned}
		\iint_{\Omega_{1} \times \Omega_{1}} 
		\frac{
			\big[ \underline{u}(x) - \underline{u}(y) \big]^{p-1} \big( \Psi_{\epsilon, k}(x) - \Psi_{\epsilon, k}(y) \big)
			- \big[ \overline{u}(x) - \overline{u}(y) \big]^{p-1} \big( \Phi_{\epsilon, k}(x) - \Phi_{\epsilon, k}(y) \big)
		}{|x-y|^{N+sp}}  dxdy = 0.
		\end{aligned}
		\end{equation}
	
	\noindent
	Since $\text{meas}(\Omega_{3}) \to 0$ as $k \to \infty$, we infer that
	\begin{equation}\label{equ9}
	\begin{gathered}
	\begin{aligned}
	\int_{\Omega_{3}} &
	\left( 
	|\nabla \underline{u}|^{p-2} \nabla \underline{u} \cdot \nabla \Psi_{\epsilon, k} 
	- |\nabla \overline{u}|^{p-2} \nabla \overline{u} \cdot \nabla \Phi_{\epsilon, k} 
	\right) dx \\[4pt]
	&= k\, m \int_{\Omega_{3}} 
	\left( 
	\dfrac{|\nabla \overline{u}|^{p}}{(\overline{u} + \epsilon)^{m+1}} 
	- \dfrac{|\nabla \underline{u}|^{p}}{(\underline{u} + \epsilon)^{m+1}} 
	\right) dx \\[4pt]
	&\geq -m \int_{\Omega_{3}} |\nabla \underline{u}|^{p} \, dx 
	\longrightarrow 0 \quad \text{as } k \to \infty.
	\end{aligned}
	\end{gathered}
	\end{equation}
	
	\noindent
	Furthermore, by applying Lagrange’s mean value theorem, we obtain
		\begin{equation*}
		\begin{aligned}
		&\iint_{\Omega_{3} \times \Omega_{3}} 
		\frac{
			\left[ \underline{u}(x)-\underline{u}(y)\right]^{p-1} \big(\Psi_{\epsilon, k}(x)-\Psi_{\epsilon, k}(y)\big)
			- \left[ \overline{u}(x)-\overline{u}(y)\right]^{p-1} \big(\Phi_{\epsilon, k}(x)-\Phi_{\epsilon, k}(y)\big)
		}{|x-y|^{N+sp}} \, dxdy \\[4pt]
		&= 
		\iint_{\Omega_{3} \times \Omega_{3}} 
		\dfrac{\big[ \underline{u}(x) - \underline{u}(y) \big]^{p-1}}{|x-y|^{N+sp}} 
		\left( \dfrac{k}{(\underline{u}(x) + \epsilon)^{m}} - \dfrac{k}{(\underline{u}(y) + \epsilon)^{m}} \right) dxdy \\[4pt]
		&\quad - 
		\iint_{\Omega_{3} \times \Omega_{3}} 
		\dfrac{\big[ \overline{u}(x) - \overline{u}(y) \big]^{p-1}}{|x-y|^{N+sp}} 
		\left( \dfrac{k}{(\overline{u}(x) + \epsilon)^{m}} - \dfrac{k}{(\overline{u}(y) + \epsilon)^{m}} \right) dxdy \\[4pt]
		&= 
		- k \iint_{\Omega_{3} \times \Omega_{3}}
		\dfrac{\big[ \underline{u}(x) - \underline{u}(y) \big]^{p-1}}{|x-y|^{N+sp}} 
		\left( 
		\dfrac{(\underline{u}(x) + \epsilon)^{m} - (\underline{u}(y) + \epsilon)^{m}}{(\underline{u}(x) + \epsilon)^{m} (\underline{u}(y) + \epsilon)^{m}} 
		\right) dxdy \\[4pt]
		&\quad + 
		k \underbrace{
			\iint_{\Omega_{3} \times \Omega_{3}} 
			\dfrac{\big[ \overline{u}(x) - \overline{u}(y) \big]^{p-1}}{|x-y|^{N+sp}} 
			\left( 
			\dfrac{(\overline{u}(x) + \epsilon)^{m} - (\overline{u}(y) + \epsilon)^{m}}{(\overline{u}(x) + \epsilon)^{m} (\overline{u}(y) + \epsilon)^{m}} 
			\right) dxdy 
		}_{> 0} \\[4pt]
		&\geq 
		- k\, m \iint_{\Omega_{3} \times \Omega_{3}} 
		\dfrac{\left| \underline{u}(x) - \underline{u}(y) \right|^{p}}{|x - y|^{N + sp}} 
		\left( 
		\dfrac{\max\big\{ (\underline{u}(x) + \epsilon)^{m-1}, (\underline{u}(y) + \epsilon)^{m-1} \big\}}
		{(\underline{u}(x) + \epsilon)^{m} (\underline{u}(y) + \epsilon)^{m}} 
		\right) dxdy.
		\end{aligned}
		\end{equation*}
	
	\noindent
	Now, define
	\[
	\omega_{\underline{u}} := \big\{ (x, y) \in \Omega_{3} \times \Omega_{3} \,\big|\, \underline{u}(x) > \underline{u}(y) \big\}.
	\]
	Since $\text{meas}(\Omega_{3}) \to 0$ as $k \to \infty$, we have
\begin{equation}\label{equ13}
		\begin{aligned}
		\iint_{\Omega_{3} \times \Omega_{3}} &
		\frac{
			\left[ \underline{u}(x)-\underline{u}(y)\right]^{p-1} \big(\Psi_{\epsilon, k}(x)-\Psi_{\epsilon, k}(y)\big)
			- \left[ \overline{u}(x)-\overline{u}(y)\right]^{p-1} \big(\Phi_{\epsilon, k}(x)-\Phi_{\epsilon, k}(y)\big)
		}{|x-y|^{N+sp}} \, dxdy \\[4pt]
		&\geq 
		- m \iint_{(\Omega_{3} \times \Omega_{3}) \cap \omega_{\underline{u}}} 
		\dfrac{\left| \underline{u}(x) - \underline{u}(y) \right|^{p}}{|x - y|^{N + sp}} 
		\left( \dfrac{\underline{u}(y) + \epsilon}{\underline{u}(x) + \epsilon} \right) dxdy \\[4pt]
		&\quad - 
		m \iint_{(\Omega_{3} \times \Omega_{3}) \cap \omega^{c}_{\underline{u}}} 
		\dfrac{\left| \underline{u}(x) - \underline{u}(y) \right|^{p}}{|x - y|^{N + sp}} 
		\left( \dfrac{\underline{u}(x) + \epsilon}{\underline{u}(y) + \epsilon} \right) dxdy \\[4pt]
		&\geq 
		- m \iint_{\Omega_{3} \times \Omega_{3}} 
		\dfrac{\left| \underline{u}(x) - \underline{u}(y) \right|^{p}}{|x - y|^{N + sp}} dxdy 
		\longrightarrow 0 \quad \text{as } k \to \infty.
		\end{aligned}
		\end{equation}
	
	\noindent  
	By noting that the condition $m < p-1$ ensures the monotonicity of the function  
	\[
	\tau \mapsto \frac{\left[ \tau - \tau_{0}\right] ^{p-1}}{\tau^{m}},
	\]
	we deduce that

\begin{equation}\label{equ10}
		\begin{gathered}
		\begin{aligned}
		& \iint_{\Omega_{2} \times \Omega_{1}} 
		\left( 
		\frac{\left[ \underline{u}(x)-\underline{u}(y)\right] ^{p-1} \big(\Psi_{\epsilon, k}(x)-\Psi_{\epsilon, k}(y)\big) - \left[ \overline{u}(x)-\overline{u}(y)\right] ^{p-1} \big(\Phi_{\epsilon, k}(x)-\Phi_{\epsilon, k}(y)\big)}{|x-y|^{N+sp}} \right) dxdy  \\[4pt]
		&= \iint_{\Omega_{2} \times \Omega_{1}} \frac{(\underline{u}(x) + \epsilon)^{m+1} - (\overline{u}(x) + \epsilon)^{m+1}}{\vert x-y \vert^{N+ s p}} \left( \frac{\left[\underline{u}(x) - \underline{u}(y)\right]^{p-1}}{(\underline{u}(x) + \epsilon)^{m}} - \frac{\left[\overline{u} (x) - \overline{u} (y)\right]^{p-1}}{(\overline{u}  + \epsilon)^{m}} \right) dx dy \geq 0.
		\end{aligned}
		\end{gathered}
		\end{equation}
	
	\noindent  
	Similarly, we have
\begin{equation}\label{equ14}
		\begin{gathered}
		\begin{aligned}
		&\iint_{\Omega_{3} \times \Omega_{1}} 
		\left( 
		\frac{\left[ \underline{u}(x)-\underline{u}(y)\right] ^{p-1} \big(\Psi_{\epsilon, k}(x)-\Psi_{\epsilon, k}(y)\big) - \left[ \overline{u}(x)-\overline{u}(y)\right] ^{p-1} \big(\Phi_{\epsilon, k}(x)-\Phi_{\epsilon, k}(y)\big)}{|x-y|^{N+sp}} \right) dxdy \\[4pt]
		&= \iint_{\Omega_{3} \times \Omega_{1}} \frac{k}{\vert x-y \vert^{N+ s p}} \left( \frac{\left[\underline{u}(x) - \underline{u}(y)\right]^{p-1}}{(\underline{u}(x) + \epsilon)^{m}} - \frac{\left[\overline{u}(x) - \overline{u}(y)\right]^{p-1}}{(\overline{u} + \epsilon)^{m}} \right) dx dy \geq 0.
		\end{aligned}
		\end{gathered}
		\end{equation}
	
	\noindent
	Finally, for the last integral, we obtain
\begin{equation*}
		\begin{gathered}
		\begin{aligned}
		&\iint_{\Omega_{3} \times \Omega_{2}} 
		\left( 
		\frac{\left[ \underline{u}(x)-\underline{u}(y)\right] ^{p-1} \big(\Psi_{\epsilon, k}(x)-\Psi_{\epsilon, k}(y)\big) - \left[ \overline{u}(x)-\overline{u}(y)\right] ^{p-1} \big(\Phi_{\epsilon, k}(x)-\Phi_{\epsilon, k}(y)\big)}{|x-y|^{N+sp}} \right) dxdy  \\[4pt]
		&= \underbrace{ \iint_{\Omega_{3} \times \Omega_{2}} \dfrac{\left[ \underline{u}(x) - \underline{u}(y)\right] ^{p-1}}{\left| x -y\right| ^{N+sp}} \left(\dfrac{k}{(\underline{u}(x) + \epsilon)^{m}} - \dfrac{(\underline{u}(y) + \epsilon)^{m+1} - (\overline{u}(y) +\epsilon)^{m+1}}{(\underline{u}(y) + \epsilon)^{m}}\right) dx dy }_{\textbf{(1)}} \\[4pt]
		&\underbrace{ -  \iint_{\Omega_{3} \times \Omega_{2}} \dfrac{\left[ \overline{u}(x) - \overline{u}(y)\right] ^{p-1}}{\left| x -y\right| ^{N+sp}} \left(\dfrac{k}{(\overline{u}(x)+ \epsilon)^{m}} - \dfrac{(\underline{u}(y) + \epsilon)^{m+1} - (\overline{u}(y) + \epsilon)^{m+1}}{(\overline{u}(y) + \epsilon)^{m}}\right) dx dy}_{\textbf{(2)}} .
		\end{aligned}
		\end{gathered}
		\end{equation*}
	
	\noindent
	Now, we define
\[
		\textbf{S}_{\underline{u}} := \left\{ (x, y) \in \Omega_{3} \times \Omega_{2} \; : \; \underline{u}(x) > \underline{u}(y) \right\}
		\text{ and } 
		\textbf{S}_{\overline{u}} := \left\{ (x, y) \in \Omega_{3} \times \Omega_{2} \; : \; \overline{u}(x) > \overline{u}(y) \right\},
		\]
	
	\noindent 
	and by employing Lagrange’s theorem in conjunction with the fractional Picone inequality \cite[Proposition~4.2]{brasco2014convexity}, and subsequently applying H\"{o}lder’s and Young’s inequalities with exponents $ \frac{p}{m+1} $ and $ \frac{p}{p - m -1} $, we deduce that
\begin{equation*}
		\begin{gathered}
		\begin{aligned}
		&\textbf{(1)} \geq  \iint_{\textbf{S}_{\underline{u}} } \dfrac{\left[ \underline{u}(x) - \underline{u}(y)\right] ^{p-1}}{\left| x -y\right| ^{N+sp}}\left(\dfrac{(\underline{u}(y) + \epsilon)^{m+1} - (\overline{u}(y) + \epsilon)^{m+1}}{(\underline{u}(x) + \epsilon)^{m}} - \dfrac{(\underline{u}(y) + \epsilon)^{m+1} - (\overline{u} (y)+ \epsilon)^{m+1}}{(\underline{u}(y) + \epsilon)^{m}}\right) dx dy\\[4pt]
		& + \iint_{\textbf{S}^{\text{c}}_{\underline{u}} } \dfrac{\left[ \underline{u}(x) - \underline{u}(y)\right] ^{p-1}}{\left| x -y\right| ^{N+sp}} \left(\dfrac{(\underline{u}(x) + \epsilon)^{m+1} - (\overline{u}(x) + \epsilon)^{m+1}}{(\underline{u}(x) + \epsilon)^{m}} - \dfrac{(\underline{u}(y) + \epsilon)^{m+1} - (\overline{u}(y) + \epsilon)^{m+1}}{(\underline{u}(y) + \epsilon)^{m}}\right) dx dy\\[4pt]
		&\geq - m \iint_{\textbf{S}_{\underline{u}} } \dfrac{\left| \underline{u}(x) - \underline{u}(y)\right| ^{p}}{\left| x -y\right| ^{N+sp}} \left( \dfrac{\underline{u}(y) + \epsilon}{\underline{u}(x) + \epsilon}\right) dx dy + \iint_{\textbf{S}^{\text{c}}_{\underline{u}} } \dfrac{\left| \underline{u}(x) - \underline{u}(y)\right| ^{p}}{\left| x -y\right| ^{N+sp}} dx dy\\[4pt]
		& - \iint_{\textbf{S}^{\text{c}}_{\underline{u}} } \dfrac{\left[ \underline{u}(x) - \underline{u}(y)\right] ^{p-1}}{\left| x -y\right| ^{N+sp}} \left(\dfrac{(\overline{u}(x) + \epsilon)^{m+1}}{(\underline{u}(x) + \epsilon)^{m}} - \dfrac{(\overline{u}(y) + \epsilon)^{m+1}}{(\underline{u}(y) + \epsilon)^{m}}\right) dx dy\\[4pt]
		&\geq - m \iint_{\textbf{S}_{\underline{u}} } \dfrac{\left| \underline{u}(x) - \underline{u}(y)\right| ^{p}}{\left| x -y\right| ^{N+sp}} dx dy+ \iint_{\textbf{S}^{\text{c}}_{\underline{u}} } \dfrac{\left| \underline{u}(x) - \underline{u}(y)\right| ^{p}}{\left| x -y\right| ^{N+sp}} dx dy  \\[4pt]
		&- \iint_{\textbf{S}^{\text{c}}_{\underline{u}} } \dfrac{\left|\overline{u}(x) - \overline{u}(y)\right| ^{m+1}}{\left| x -y\right| ^{N\frac{m+1}{p}+s(m+1)}} \dfrac{\left| \underline{u}(x) - \underline{u}(y)\right| ^{p -m-1}}{\left| x -y\right| ^{N\frac{p -m -1}{p}+s(p-m-1)}} dx dy\\[4pt]
		&\geq - m \iint_{\textbf{S}_{\underline{u}} } \dfrac{\left| \underline{u}(x) - \underline{u}(y)\right| ^{p}}{\left| x -y\right| ^{N+sp}} dx dy - \dfrac{m+1}{p - m -1} \iint_{\textbf{S}^{\text{c}}_{\underline{u}} } \dfrac{\left| \underline{u}(x) - \underline{u}(y)\right| ^{p}}{\left| x -y\right| ^{N+sp}} dx dy -  \dfrac{m+1}{p} \iint_{\textbf{S}^{\text{c}}_{\underline{u}} } \dfrac{\left| \overline{u}(x) - \overline{u}(y)\right| ^{p}}{\left| x -y\right| ^{N+sp}} dx dy,
		\end{aligned}
		\end{gathered}
		\end{equation*}
and
 
 \begin{equation*}
		\begin{gathered}
		\begin{aligned}
		&\textbf{(2)} \geq -  \iint_{\textbf{S}_{\overline{u}} } \dfrac{\left[ \overline{u}(x) - \overline{u}(y)\right] ^{p-1}}{\left| x -y\right| ^{N+sp}} \left(\dfrac{(\underline{u}(x) + \epsilon)^{m+1} - (\overline{u}(x) +\epsilon)^{m+1}}{(\overline{u}(x) + \epsilon)^{m}} - \dfrac{(\underline{u}(y) + \epsilon)^{m+1} - (\overline{u}(y) + \epsilon)^{m+1}}{(\overline{u}(y) +  \epsilon)^{m}} \right) dx dy\\[4pt]
		& -  \iint_{\textbf{S}^{\text{c}}_{\overline{u}} } \dfrac{\left[ \overline{u}(x) - \overline{u}(y)\right] ^{p-1}}{\left| x -y\right| ^{N+sp}} \left(\dfrac{k}{(\overline{u}(x) +\epsilon)^{m}} - \dfrac{k}{(\overline{u}(y) + \epsilon)^{m}}\right) dx dy\\[4pt]
		& \geq - \iint_{\textbf{S}_{\overline{u}} } \dfrac{\left|\underline{u}(x) - \underline{u}(y)\right| ^{m+1}}{\left| x -y\right| ^{N\frac{m+1}{p}+s(m+1)}} \dfrac{\left| \overline{u}(x) - \overline{u}(y)\right| ^{p -m-1}}{\left| x -y\right| ^{N\frac{p -m -1}{p}+s(p-m-1)}} dx dy + \iint_{\textbf{S}_{\overline{u}}} \dfrac{\left| \overline{u}(x) - \overline{u}(y)\right| ^{p}}{\left| x -y\right| ^{N+sp}} dx dy\\[4pt]
		& \geq  - \dfrac{m+1}{p} \iint_{\textbf{S} _{\overline{u}} } \dfrac{\left| \underline{u}(x) - \underline{u}(y)\right| ^{p}}{\left| x -y\right| ^{N+sp}} dx dy  - \dfrac{p}{p - m - 1}  \iint_{\textbf{S}_{\overline{u}} } \dfrac{\left| \overline{u}(x) - \overline{u}(y)\right| ^{p}}{\left| x -y\right| ^{N+sp}} dx dy + \iint_{\textbf{S}_{\overline{u}}} \dfrac{\left| \overline{u}(x) - \overline{u}(y)\right| ^{p}}{\left| x -y\right| ^{N+sp}} dx dy.
		\end{aligned}
		\end{gathered}
		\end{equation*}
	
	\noindent Then, we obtain 
		\begin{equation}\label{equ15}
		\begin{gathered}
		\begin{aligned}
		&\iint_{\Omega_{3} \times \Omega_{2}} 
		\left( 
		\frac{
			\big[\underline{u}(x)-\underline{u}(y)\big]^{p-1}\big(\Psi_{\epsilon, k}(x)-\Psi_{\epsilon, k}(y)\big)
			- 
			\big[\overline{u}(x)-\overline{u}(y)\big]^{p-1}\big(\Phi_{\epsilon, k}(x)-\Phi_{\epsilon, k}(y)\big)
		}{
			|x-y|^{N+sp}
		} 
		\right) dxdy \\
		&\quad \geq 
		- m \iint_{\mathbf{S}_{\underline{u}} } 
		\dfrac{|\underline{u}(x) - \underline{u}(y)|^{p}}{|x -y|^{N+sp}} \, dx dy 
		- \dfrac{m+1}{p - m -1} 
		\iint_{\mathbf{S}^{c}_{\underline{u}} } 
		\dfrac{|\underline{u}(x) - \underline{u}(y)|^{p}}{|x -y|^{N+sp}} \, dx dy \\
		&\quad 
		- \dfrac{m+1}{p} 
		\iint_{\mathbf{S}^{c}_{\underline{u}} } 
		\dfrac{|\overline{u}(x) - \overline{u}(y)|^{p}}{|x -y|^{N+sp}} \, dx dy 
		- \dfrac{m+1}{p} 
		\iint_{\mathbf{S}_{\overline{u}} } 
		\dfrac{|\underline{u}(x) - \underline{u}(y)|^{p}}{|x -y|^{N+sp}} \, dx dy \\
		&\quad 
		- \dfrac{p}{p - m - 1}  
		\iint_{\mathbf{S}_{\overline{u}} } 
		\dfrac{|\overline{u}(x) - \overline{u}(y)|^{p}}{|x -y|^{N+sp}} \, dx dy 
		+ 
		\iint_{\mathbf{S}_{\overline{u}}} 
		\dfrac{|\overline{u}(x) - \overline{u}(y)|^{p}}{|x -y|^{N+sp}} \, dx dy \\[4pt]
		&\quad \to 0 \quad \text{as } k \to \infty.
		\end{aligned}
		\end{gathered}
		\end{equation}
	
	\noindent 
	Hence, combining \eqref{equ17}–\eqref{equ15}, we deduce that
		\begin{equation}\label{equ16}
		\begin{gathered}
		\begin{aligned}
		& \int_{\Omega} 
		\Big( 
		|\nabla \underline{u}|^{p-2} \nabla \underline{u} \cdot \nabla \Psi_{\epsilon, k} 
		- 
		|\nabla \overline{u}|^{p-2} \nabla \overline{u} \cdot \nabla \Phi_{\epsilon, k} 
		\Big) dx \\
		& + 
		\iint_{\mathbb{R}^{2N}} 
		\left( 
		\frac{
			\big[\underline{u}(x)-\underline{u}(y)\big]^{p-1}\big(\Psi_{\epsilon, k}(x)-\Psi_{\epsilon, k}(y)\big)
			- 
			\big[\overline{u}(x)-\overline{u}(y)\big]^{p-1}\big(\Phi_{\epsilon, k}(x)-\Phi_{\epsilon, k}(y)\big)
		}{
			|x-y|^{N+sp}
		} 
		\right) dxdy \\
		&\qquad \geq 0 \quad \text{as } k \to \infty.
		\end{aligned}
		\end{gathered}
		\end{equation}
	
	\noindent 
	Finally, substituting \eqref{equ7}, \eqref{equ21}, \eqref{equ20}, and \eqref{equ16}, and letting 
	$ \lim_{k \to \infty} \liminf_{\epsilon \to 0} $, we arrive at
	\begin{equation*}
	\begin{aligned}
	0 \leq \int_{\{ \underline{u} > \overline{u} \}} 
	\big( 
	\underline{u}^{m+1} - \overline{u}^{m+1} 
	\big)^{2} \, dx 
	\leq 0,
	\end{aligned}
	\end{equation*}
	which implies that 
	$ \underline{u} \leq \overline{u} $ a.e. in $\Omega$.
\end{proof}
\subsection{Existence and Uniqueness of Weak Solutions for Problem~\eqref{PP2}}

\begin{proof}[\textbf{Proof of Theorem} \ref{theorem3}] 
	We introduce the energy functional $\mathcal{J}$ associated with problem~\eqref{PP2}, which is defined on the space $ \textbf{WL} := W^{1,p}_{0}(\Omega) \cap L^{2(m+1)}(\Omega), $ endowed with the following Cartesian norm:
	\[
	\|u\|_{\textbf{WL}} = \|u\|_{W^{1,p}_{0}(\Omega)} + \|u\|_{L^{2(m+1)}(\Omega)}.
	\]
	The functional $\mathcal{J}$ is expressed as
		\begin{equation} \label{3equ1}
		\begin{aligned}
		\mathcal{J}(u)
		&= \frac{1}{2(m+1)} \int_{\Omega} u^{2(m+1)}\,dx
		+ \frac{\lambda}{p} \int_{\Omega} |\nabla u|^{p}\,dx
		+ \frac{\lambda}{p} \iint_{\mathbb{R}^{2N}}
		\frac{|u(x) - u(y)|^{p}}{|x - y|^{N + sp}}\,dx\,dy \\
		&\quad - \frac{1}{m+1} \int_{\Omega} g_{0}(x)\,(u^{+})^{m+1}\,dx
		- \frac{\lambda m}{\delta + 1} \int_{\Omega} d(x)^{-\gamma} (u^{+})^{\delta + 1}\,dx.
		\end{aligned}
		\end{equation}
	
	\noindent 
	Therefore, it follows directly that the functional $\mathcal{J}$ is well defined and weakly lower semicontinuous on $\mathbf{WL}$. Moreover, when $m = 0$, we have
	\begin{align*}
	\mathcal{J}(u)
	&\geq \frac{1}{2}\|u\|_{L^{2 }(\Omega)}^{2}
	+ \frac{\lambda}{p}\|u\|_{W_0^{1,p}(\Omega)}^{p}
	- \|g_0\|_{L^2(\Omega)}\|u\|_{L^{2}(\Omega)}\\
	&\geq \frac{\lambda}{p}\|u\|_{W_0^{1,p}(\Omega)}^{p}
	+ \|u\|_{L^{2}(\Omega)}\!\left[\frac{1}{2}\|u\|_{L^{2}(\Omega)} - \|g_{0}\|_{L^{2}(\Omega)}\right].
	\end{align*}
	For the case $m > 0$, since $\gamma < \tfrac{1}{2}$, applying H\"{o}lder’s inequality yields
	\begin{align*}
	\mathcal{J}(u)
	&\geq \frac{1}{2(m+1)}\|u\|_{L^{2(m+1)}(\Omega)}^{2(m+1)}
	+ \frac{\lambda}{p}\|u\|_{W_0^{1,p}(\Omega)}^{p}
	- \frac{1}{m+1}\|g_0\|_{L^2(\Omega)}\|u\|_{L^{2(m+1)}(\Omega)}^{m+1} \\
	&\quad - \frac{\lambda m\, |\Omega|^{\frac{m-\delta}{2(m+1)}}}{\delta + 1}
	\left(\int_{\Omega} d(x)^{-2\gamma}\,dx\right)^{\!\frac{1}{2}}
	\|u\|_{L^{2(m+1)}(\Omega)}^{\delta + 1} \\[4pt]
	&\geq \frac{\lambda}{p}\|u\|_{W_0^{1,p}(\Omega)}^{p}
	+ \|u\|_{L^{2(m+1)}(\Omega)}^{m+1}
	\Bigg[
	\frac{1}{2(m+1)}\|u\|_{L^{2(m+1)}(\Omega)}^{m+1}
	- \frac{1}{m+1}\|g_{0}\|_{L^{2}(\Omega)} \\
	&\hspace{4.2cm}
	- \frac{C\,\lambda m\, |\Omega|^{\frac{m-\delta}{2(m+1)}}}{\delta + 1}
	\|u\|_{L^{2(m+1)}(\Omega)}^{\delta - m}
	\Bigg],
	\end{align*}
	where the constant $C$ is independent of $u$. Consequently, in both cases, since $\delta < m$, the functional $\mathcal{J}$ is coercive on $\mathbf{WL}$. Therefore, $\mathcal{J}$ admits a global minimizer on this space, denoted by $u_0$.  Using the notation $t = t^{+} - t^{-}$, we obtain
	\begin{align*}
	&\mathcal{J}(u_0)
	= \mathcal{J}(u_0^{+}) 
	+ \frac{1}{2(m+1)} \int_{\Omega} (u_{0}^{-})^{2(m+1)}\,dx 
	+ \frac{\lambda}{p} \int_{\Omega} \left| \nabla (u_{0}^{-})\right|^{p}\,dx \\
	&\,
	+ \frac{\lambda}{p} \iint_{\mathbb{R}^{2N}}
	\frac{|u_{0}^{-}(x) - u_{0}^{-}(y)|^{p}}{|x - y|^{N + sp}}\,dx\,dy  
	+ \frac{2\lambda}{p} \iint_{\mathbb{R}^{2N}}
	\frac{|u_{0}^{-}(x) - u_{0}^{+}(y)|^{p}}{|x - y|^{N + sp}}\,dx\,dy \geq \mathcal{J}(u_0^{+}).
	\end{align*}
	Consequently, we deduce that $u_{0} \geq 0$ in $\Omega$. To demonstrate that $u_{0} \not\equiv 0$ in $\Omega$, we construct an appropriate test function $u \in \mathbf{WL}$ such that $\mathcal{J}(u) < \mathcal{J}(0) = 0$.  
	Since $m < p - 1$, let $\phi \in C_{c}^{1}(\Omega)$ be a nonnegative, nontrivial function satisfying  $\operatorname{supp}(\phi) \subset \operatorname{supp}(g_{0})$.   Then we obtain
	\[
	\mathcal{J}(t\phi) \leq t^{m+1} \big[\, C_{1} t^{m+1} + C_{2} t^{p - m - 1} - C_{3} \,\big], \quad \forall\, t > 0,
	\]
	where the constants $C_{1}$, $C_{2}$, and $C_{3}$ are independent of $t$, with $C_{3} > 0$ since $g_{0} \geq \lambda \underline{g} \not\equiv 0$.   Hence, for sufficiently small $t > 0$, we have $\mathcal{J}(t\phi) < 0$.   As $\mathcal{J}(0) = 0$, it follows that $u_{0} \not\equiv 0$.  By the Gâteaux differentiability of $\mathcal{J}$, the function $u_{0}$ satisfies the following weak formulation:
	\begin{equation}\label{3weakform1}
	\begin{aligned}
	& \int_{\Omega} u^{2m+1} \varphi\,dx
	+ \lambda \int_{\Omega} |\nabla u|^{p-2} \nabla u \cdot \nabla \varphi\,dx  \\
	&\quad 
	+ \lambda \iint_{\mathbb{R}^{2N}} 
	\frac{[u(x) - u(y)]^{p-1}\big(\varphi(x) - \varphi(y)\big)}{|x - y|^{N + sp}}\,dx\,dy \\
	&= \int_{\Omega} \Big( g_{0}(x) u^{m} + \lambda m\, d(x)^{-\gamma} u^{\delta} \Big) \varphi\,dx, 
	\quad \forall\, \varphi \in \textbf{WL}.
	\end{aligned}
	\end{equation}
	On the other hand, regarding the regularity and positivity of weak solutions, we assert that every weak solution to problem~\eqref{PP2} lies in $L^{\infty}(\Omega)$. To demonstrate this property, we adapt the approach developed in~\cite[Theorem~3.2]{franzina2014fractional} to the mixed local–nonlocal framework. Specifically, let $u_0 \in \textbf{WL}$ denote a weak solution of~\eqref{PP2}. We set
	\begin{equation}\label{equation}
	v_{0} = \frac{u_{0}}{\rho \left\| u_{0}\right\|_{L^{p}(\Omega)}} 
	\quad \text{with} \quad 
	\rho = \max \left\{ 1,\dfrac{1}{\left\| u_{0} \right\|_{L^{p}(\Omega)} }\right\}.
	\end{equation}
	Since $v_{0} \in \textbf{WL}$ and $\|v_{0}\|_{L^{p}(\Omega)} = \rho^{-1}$, we define the sequence $(w_{k})$ by
	\begin{equation*}
	\begin{cases}
	w_0(x) := (v_0(x))^{+}, \\[4pt]
	w_k(x) := \left( v_0(x) - 1 + 2^{-k}\right) ^{+}, & \text{for } k \in \mathbb{N}^{*}.
	\end{cases}
	\end{equation*}
	We begin by noting the following straightforward properties of the sequence $(w_k)$:
	$$
	w_k = 0
	\text{ a.e. in } \mathbb{R}^{N}\setminus \Omega, 
	\quad \text{and} \quad 
	w_k \in \textbf{WL},
	$$
	together with
	\begin{equation}\label{equ53}
	\begin{cases}
	0 \leq w_{k+1}(x) \leq w_k(x), & \text{a.e. in } \mathbb{R}^N,\\[4pt]
	v_0(x) < (2^{k+1} - 1)\, w_k(x), & \text{for } x \in \{ w_{k+1} > 0 \}.
	\end{cases}
	\end{equation}
	Moreover, the following inclusion holds:
	\begin{equation*}
	\{ w_{k+1} > 0 \} \subseteq \{ w_k > 2^{-(k+1)} \},
	\quad \text{for all } k \in \mathbb{N}.
	\end{equation*}
	Consequently, we obtain  
	\begin{align}\label{equ54}
	\left| \{ w_{k+1} > 0 \} \right|
	\leq 2^{p(k+1)}  \| w_{k} \|_{L^{p}(\Omega)}^{p}.
	\end{align}
	We set \( U_{k} := \| w_{k} \|_{L^{p}(\Omega)}^{p}.\) It is worth observing that  $ U_k \to 0 $  as $ k \to \infty. $ Indeed, since $\delta < m < p-1$ and $\rho \, \| u_0 \|_{L^p(\Omega)} \geq 1$, by applying the inequality
	\begin{equation*}
	|x^+ - y^+|^p \leq [x - y]^{p-1} (x^+ - y^+), \quad \text{for all } x, y \in \mathbb{R},
	\end{equation*}
	we obtain
	\begin{align*}
	&\lambda \| w_{k+1}\|^p_{W^{s,p}_0(\Omega)} 
	=\lambda   \iint_{\mathbb{R}^{2N}} \frac{|w_{k+1}(x) - w_{k+1}(y)|^p}{|x - y|^{N+sp}} \, dx dy \\[2mm]
	&\leq \lambda \int_{\Omega} |\nabla w_{k+1}|^p \, dx + \lambda  \iint_{\mathbb{R}^{2N}} \frac{|w_{k+1}(x) - w_{k+1}(y)|^p}{|x - y|^{N+sp}} \, dx dy \\[1mm]
	&\leq \left( \rho \, \| u_0 \|_{L^p(\Omega)}\right) ^{1-p} \int_{\Omega} |\nabla u_0|^{p-2} \nabla u_0 \cdot \nabla w_{k+1} \, dx \\[1mm]
	&\quad + \left( \rho \, \| u_0 \|_{L^p(\Omega)}\right) ^{1-p} \iint_{\mathbb{R}^{2N}} \frac{[u_0(x) - u_0(y)]^{p-1} (w_{k+1}(x) - w_{k+1}(y))}{|x - y|^{N+sp}} \, dx dy \\[1mm]
	&\leq  (\rho \, \| u_0 \|_{L^p(\Omega)})^{1-p+m}  \int_{\Omega} g_{0}(x) v_{0}^{m} w_{k+1} dx + \lambda m (\rho  \| u_0 \|_{L^p(\Omega)})^{1-p+\delta} \int_{\Omega} d(x)^{-\gamma} v_0^{\delta} w_{k+1}  dx \\[1mm]
	&\leq  \int_{\Omega} g_{0}(x) v_{0}^{m} w_{k+1} \, dx + \lambda m \int_{\Omega} d(x)^{-\gamma} v_0^{\delta} w_{k+1} \, dx.
	\end{align*}
	In the case \( m = 0 \), combining \eqref{equ53} with \eqref{equ54} yields  
	\begin{equation}\label{3equ2}
	\begin{aligned}
	\lambda \| w_{k+1}\|_{W^{s,p}_0(\Omega)}^{p} 
	&\leq C 
	\left|\{ w_{k+1} > 0 \}\right|^{1 - \frac{1}{p}}
	U_{k}^{\frac{1}{p}} \\
	&\leq C \, 2^{p(k+1)\left(1 - \frac{1}{p}\right)} U_{k} \\
	&\leq C \left( 2^{k+1} + 1 \right)^{p - 1} U_{k}.
	\end{aligned}
	\end{equation}
	When \( m > 0 \), by applying Hardy’s inequality together with \eqref{equ53} and \eqref{equ54} once again, we obtain
	\begin{equation}\label{3equ3}
	\begin{aligned}
	\lambda \| w_{k+1}\|_{W^{s,p}_0(\Omega)}^{p} 
	&\leq C \left( 2^{k+1} - 1 \right)^{m}
	\left|\{ w_{k+1} > 0 \}\right|^{1 - \frac{m+1}{p}}
	U_{k}^{\frac{m+1}{p}} \\
	&\quad + C \lambda m 
	\left( 2^{k+1} - 1 \right)^{\delta}
	\left|\{ w_{k+1} > 0 \}\right|^{1 - \frac{\delta+1}{p}}
	U_{k}^{\frac{\delta+1}{p}} \\
	&\leq C \Big[
	\left( 2^{k+1} - 1 \right)^{m}
	2^{p(k+1)\left(1 - \frac{m+1}{p}\right)} \\
	&\quad + \lambda m 
	\left( 2^{k+1} - 1 \right)^{\delta}
	2^{p(k+1)\left(1 - \frac{\delta+1}{p}\right)}
	\Big] U_{k} \\
	&\leq C \left( 2^{k+1} + 1 \right)^{p - 1} U_{k},
	\end{aligned}
	\end{equation}
	where \( C > 0 \) denotes a positive constant that may vary from line to line. Moreover, by H\"{o}lder's inequality and Theorem \ref{thm3}, we obtain
	\begin{equation}\label{equ25}
	U_{k+1} = \int_{\{ w_{k+1} > 0 \}} w_{k+1}^{p} \, dx
	\leq C \left| \{ w_{k+1} > 0 \} \right|^{1 - \frac{p}{p_{s}^{*}}}
	\| w_{k+1} \|_{W_{0}^{s,p}(\Omega)}^{p}.
	\end{equation}
	Thus, from both cases \eqref{3equ2} and \eqref{3equ3}, we obtain
	\begin{equation}\label{equ26}
	U_{k+1} \leq C^{k} U_k^{1 + \vartheta}, \quad \text{for all } k \in \mathbb{N},
	\end{equation}
	for some constant $C > 1$ and $\vartheta = \dfrac{s p}{N}$. Consequently,
	\begin{equation}\label{equ28}
	U_k \leq \dfrac{\eta^{k}}{\rho^{p}}, \quad \text{for all } k \in \mathbb{N},
	\end{equation}
	where $\eta = C^{-1 / \vartheta}$ and  $ \rho = \max \left\lbrace 1, \, \left\| u_0 \right\|_{L^{p}(\Omega)}^{-1}, \, C^{1 / (p \vartheta^{2})} \right\rbrace. $ Indeed, the claim follows by induction. From the construction in \eqref{equation}, it is evident that
	
	\[
	U_0 = \left\| v_0^{+} \right\|_{L^{p}(\Omega)}^{p} 
	\leq \left\| v_0 \right\|_{L^{p}(\Omega)}^{p} 
	= \left\| \frac{u_0}{\rho \, \left\| u_0 \right\|_{L^{p}(\Omega)}} \right\|_{L^{p}(\Omega)}^{p} 
	= \frac{1}{\rho^{p}}.
	\]
	Assume now that \eqref{equ28} holds for some $k \in \mathbb{N}$. Using \eqref{equ26} and standard computations, we obtain
	\[
	U_{k+1} \leq C^{k} U_k^{1 + \vartheta} \leq \dfrac{\eta^{k+1}}{\rho^{p}}.
	\]
	Since $\eta \in (0,1)$, we deduce that
	\begin{equation}\label{equ56}
	\lim_{k \to \infty} U_k = 0.
	\end{equation}
	Furthermore, since $w_k \to (v_0 - 1)^{+}$ a.e.\ in $\mathbb{R}^{N}$, relation \eqref{equ56} implies that $w_k \to 0$ a.e.\ in $\Omega$. Consequently, we obtain $v_0 \leq 1$ a.e.\ in $\Omega$, and therefore,
	\[
	\| u_0 \|_{L^{\infty}(\Omega)} \leq \rho\, \| u_0 \|_{L^{p}(\Omega)}.
	\]
Hence, we conclude that $u_0 \in L^{\infty}(\Omega)$. Moreover, since $\gamma < \frac{1}{\textbf{d}}$, by \cite[Theorem~1.1]{antonini2025global} we infer that  $ u_0 \in C^{1,\xi}(\overline{\Omega}) $ for some $\xi \in (0,1)$.  On the other hand, in the case $m=0$, or when $m>0$ with $m+1<p$, there exists a sufficiently small constant $\varepsilon>0$ such that the function $\varepsilon \phi^{s}_{1,p}$ serves as a sub-solution of problem \eqref{PP2}. Indeed, for such a choice of $\varepsilon>0$ small enough, one has
	\begin{equation*}
	\begin{aligned}
	&(\varepsilon \phi^{s}_{1,p})^{2m+1} 
	- \lambda \Delta_{p}(\varepsilon \phi^{s}_{1,p}) 
	+ \lambda (-\Delta)^{s}_{p}(\varepsilon \phi^{s}_{1,p}) \\
	&\qquad \leq g_0(x)\, (\varepsilon \phi^{s}_{1,p})^{m} 
	+ \lambda m\, d(x)^{-\gamma} \left( \varepsilon \phi_{1,p}^{s} \right)^{\delta}
	\quad \text{in } \Omega.
	\end{aligned}
	\end{equation*}
	By applying the comparison principle (Theorem~\ref{theorem2}), we obtain 
	\( \varepsilon\, \phi^{s}_{1,p} \leq u_{0} \) in \( \Omega \). Hence, by Remark~\ref{remark1}, it follows that \( u_{0} \geq c_{1}\, d(x) \) for some constant \( c_{1} > 0 \). Thus, we conclude that \( u_{0} > 0 \) in \( \Omega \). Moreover, according to Theorem \ref{theorem1} we have 
	\( u_{0} \leq c_{2}\, d(x)^{\alpha} \) in \( \Omega \)  for every \( \alpha \in [s,1) \) with \( \alpha \neq \frac{ps}{p-1} \), where 
	\( c_{2} > 0 \) is a constant. Consequently, we deduce that 
	\( u_{0} \in \mathcal{M}_{1,\alpha}^{1}(\Omega) \). Finally, to establish the contraction property \eqref{contraction}, let
	\( u_{1}, u_{2} \in C^{1,\xi}(\overline{\Omega}) \cap \mathcal{M}^{1}_{1,\alpha}(\Omega) \) 
	be two weak solutions of \eqref{PP2} corresponding to 
	\( g_{1} \) and \( g_{2} \), respectively, for some \( \xi \in (0,1) \) and  for every \( \alpha \in [s,1) \) with \( \alpha \neq \frac{ps}{p-1} \).  That is, for any \( \Phi, \Psi \in \mathbf{WL} \), we have
	\begin{equation*}
	\begin{aligned}
	& \int_{\Omega} u_{1}^{2m+1} \Phi \, dx 
	+ \lambda \int_{\Omega} |\nabla u_{1}|^{p-2} \nabla u_{1} \cdot \nabla \Phi \, dx  \\
	& \quad + \lambda \iint_{\mathbb{R}^{2N}} 
	\frac{\left[ u_{1}(x)-u_{1}(y)\right]^{p-1} \big(\Phi(x)-\Phi(y)\big)}{|x-y|^{N+sp}} \, dx \, dy \\
	& \qquad = \int_{\Omega} 
	\Big( g_{1}(x)\, u_{1}^{m} + \lambda m\, d(x)^{-\gamma}\, u_{1}^{\delta} \Big) \Phi \, dx,
	\end{aligned}
	\end{equation*}
	and
	\begin{equation*}
	\begin{aligned}
	& \int_{\Omega} u_{2}^{2m+1} \Psi \, dx 
	+ \lambda \int_{\Omega} |\nabla u_{2}|^{p-2} \nabla u_{2} \cdot \nabla \Psi \, dx  \\
	& \quad + \lambda \iint_{\mathbb{R}^{2N}} 
	\frac{\left[ u_{2}(x)-u_{2}(y)\right]^{p-1} \big(\Psi(x)-\Psi(y)\big)}{|x-y|^{N+sp}} \, dx \, dy \\
	& \qquad = \int_{\Omega} 
	\Big( g_{2}(x)\, u_{2}^{m} + \lambda m\, d(x)^{-\gamma}\, u_{2}^{\delta} \Big) \Psi \, dx.
	\end{aligned}
	\end{equation*}
By choosing appropriate test functions $\Phi$ and $\Psi$, according to the two cases $m=0$ and $m>0$ as in the proof of Theorem~\ref{theorem2}, and by following the same line of arguments, we obtain
\[
\int_{\Omega} \big( (u_{1}^{m+1} - u_{2}^{m+1})^{+} \big)^{2} \, dx
\leq \int_{\Omega} \big( g_{1}(x) - g_{2}(x) \big) \big( u_{1}^{m+1} - u_{2}^{m+1} \big)^{+} \, dx.
\]
	By applying H\"{o}lder’s inequality, we obtain
	\begin{equation*}
	\left\| \left( u_{1}^{m+1} - u_{2}^{m+1}\right)^{+} \right\|_{L^{2}(\Omega)} 
	\leq \left\| \left( g_{1} - g_{2}\right) ^{+} \right\|_{L^{2}(\Omega)}.
	\end{equation*}
\end{proof}
\begin{proof}[\textbf{Proof of Theorem} \ref{theorem4}] 
	Consider $g_{0} \in L^{2}(\Omega)$. Since $C^{\infty}_{c}(\Omega)$ is dense in $L^{2}(\Omega)$, there exists a sequence $(\breve{g}_n) \subset C^{\infty}_{c}(\Omega)$ such that $\breve{g}_n \to g_0$ in $L^{2}(\Omega)$.  
	We define  $ g_n = \max(\breve{g}_n, \lambda \underline{g}). $ According to Theorem~\ref{theorem3}, for any $\varphi \in \mathbf{WL}$ and for all $n \geq n_0$, the problem
	
	\begin{equation}\label{3equ4}
	\begin{aligned}
	& \int_{\Omega} u_{n}^{2m+1} \varphi \, dx 
	+ \lambda \int_{\Omega} |\nabla u_{n}|^{p-2} \nabla u_{n} \cdot \nabla \varphi \, dx  \\
	& \quad + \lambda \iint_{\mathbb{R}^{2N}} 
	\frac{\big[u_{n}(x)-u_{n}(y)\big]^{p-1} \big(\varphi(x)-\varphi(y)\big)}{|x-y|^{N+sp}} \, dx \, dy \\
	& \qquad = \int_{\Omega} \Big( g_{n}(x)\, u_{n}^{m} + \lambda m\, d(x)^{-\gamma}\, u_{n}^{\delta} \Big) \varphi \, dx,
	\end{aligned}
	\end{equation}
	admits a unique positive weak solution  $ u_n \in C^{1, \xi}(\overline{\Omega}) \cap \mathcal{M}^{1}_{1, \alpha}(\Omega), $ for some $\xi \in (0, 1)$ and  for every \( \alpha \in [s,1) \) with \( \alpha \neq \frac{ps}{p-1} \).  Using the inequality
	\begin{equation}\label{equ23}
	\left( a - b \right)^{2(r+1)} \leq \left( a^{r+1} - b^{\,r+1} \right)^{2} 
	\quad \text{for all } r \geq 0,\; a,b \geq 0,
	\end{equation}
	and combining it with \eqref{contraction}, we derive the following properties.  
	First, for any $i, j \in \mathbb{N}\setminus\{0\}$,
	\begin{equation*}
	\left\| \left( u_i - u_j \right)^{+} \right\|_{L^{2(m+1)}(\Omega)}
	\leq \left\| \left( u_i^{m+1} - u_j^{m+1} \right)^{+} \right\|_{L^{2}(\Omega)}^{\frac{1}{m+1}}
	\leq \left\| \left( g_i - g_j \right)^{+} \right\|_{L^{2}(\Omega)}^{\frac{1}{m+1}}.
	\end{equation*}
	Hence, the sequence $(u_n)$ is Cauchy in the Banach space $L^{2(m+1)}(\Omega)$, and therefore it converges to some limit $u \in L^{2(m+1)}(\Omega)$. Second, the limit $u$ does not depend on the particular choice of the approximating sequence $(g_n)$.  
	Indeed, let $(\tilde{g}_n) \subset C_c^{\infty}(\Omega)$ be another sequence such that $\tilde{g}_n \to g_0$ in $L^{2}(\Omega)$, and let $\tilde{u}_n$ denote the corresponding positive solutions to \eqref{3equ4}, converging to $\tilde{u}$. For any $n \in \mathbb{N}$, we have
	\begin{equation*}
	\left\| \left( u_n - \tilde{u}_n \right)^{+} \right\|_{L^{2(m+1)}(\Omega)}^{m+1}
	\leq \left\| \left( u_n^{m+1} - \tilde{u}_n^{m+1} \right)^{+} \right\|_{L^{2}(\Omega)}
	\leq \left\| \left( g_n - \tilde{g}_n \right)^{+} \right\|_{L^{2}(\Omega)}.
	\end{equation*}
	Passing to the limit yields $u \leq \tilde{u}$; interchanging the roles of $u$ and $\tilde{u}$ then shows that $u = \tilde{u}$. Third, for $n \in \mathbb{N}\setminus\{0\}$, let $g_n = \min\{ g_0,\, n \lambda \underline{g} \}$. It is straightforward to verify that the sequence $(u_n)$ is nondecreasing and satisfies $u_n \leq u$ a.e.\ in $\Omega$ for all $n \in \mathbb{N}\setminus\{0\}$. Consequently,
	\begin{equation}\label{equ65}
	0 < C\, d(x) \leq u_1(x) \leq u(x) \quad \text{in } \Omega,
	\end{equation}
	for some constant $C > 0$ independent of $n$. After these properties have been established, taking $\varphi = u_n$ in \eqref{3equ4}, we obtain
	\begin{equation*}
	\lambda \int_{\Omega} \left|\nabla u_{n} \right|^{p} dx 
	\leq \int_{\Omega} g_{n}(x) u_{n}^{m+1} dx 
	+ \lambda\, m \int_{\Omega} d(x)^{-\gamma} u_{n}^{\delta + 1} dx.
	\end{equation*}
	When \( m = 0 \), it follows that
	\begin{equation*}
	\lambda \int_{\Omega} \left|\nabla u_{n} \right|^{p} dx 
	\leq \| g_{n}\|_{L^{\infty}(\Omega)} 
	|\Omega|^{1 - \frac{1}{2}} 
	\| u_{n}\|_{L^{2}(\Omega)}.
	\end{equation*}
	On the other hand, when \( m > 0 \), we have
	\begin{equation*}
	\begin{aligned}
	\lambda \int_{\Omega} \left|\nabla u_{n} \right|^{p} dx 
	&\leq \| g_{n}\|_{L^{\infty}(\Omega)} 
	|\Omega|^{1 - \frac{1}{2}} 
	\| u_{n}\|_{L^{2(m+1)}(\Omega)}^{m+1} \\
	&\quad + \lambda m 
	\left( \int_{\Omega} d(x)^{-2\gamma} dx \right)^{\frac{1}{2}} 
	|\Omega|^{\frac{1}{2}\left(1 - \frac{\delta + 1}{m+1}\right) } 
	\| u_{n}\|_{L^{2(m+1)}(\Omega)}^{\delta + 1}.
	\end{aligned}
	\end{equation*}
	Hence, in both cases,  \( (u_n)_{n\in\mathbb{N}} \) is uniformly bounded in \( W^{1, p}_0(\Omega) \). Consequently, $ u_n \rightharpoonup u \ \text{in } W^{1,p}_{0}(\Omega), 
	\, u_n \to u \ \text{in } L^{l}(\Omega) \ \text{for } 1 \leq l < p^{*}, 
	\, \text{and} \, u_n(x) \to u(x) \ \text{a.e. in } \Omega. $ In fact, our goal is to pass to the limit in \eqref{3equ4}. From \eqref{equ65}, the sequence \( (u_n) \) remains uniformly bounded away from zero in the interior of the domain \( \Omega \). This enables us to apply \cite[Theorem~2.1 and Remark~2.2]{boccardo1992almost}, which imply that
	\[
	\nabla u_n(x) \to \nabla u(x) \quad \text{a.e. in } \Omega \text{ as } n \to \infty.
	\]
	Consequently,
	\[
	\left| \nabla u_{n}\right| ^{p-2} \nabla u_{n} \to \left| \nabla u\right| ^{p-2} \nabla u 
	\quad \text{a.e. in } \Omega,
	\]
	and the sequence \( \left( |\nabla u_{n}|^{p-1}\right)  \) is bounded in \( L^{\frac{p}{p-1}}(\Omega) \). Hence,
	\[
	\left| \nabla u_{n}\right| ^{p-2} \nabla u_{n} 
	\rightharpoonup 
	\left| \nabla u\right| ^{p-2} \nabla u 
	\quad \text{weakly in } L^{\frac{p}{p-1}}(\Omega).
	\]
	Since \( \varphi \in \mathbf{WL} \), we then deduce that
	\begin{equation}
	\int_{\Omega} \left| \nabla u_{n}\right| ^{p-2} \nabla u_{n} \cdot \nabla \varphi \, dx 
	\to 
	\int_{\Omega} \left| \nabla u\right| ^{p-2} \nabla u \cdot \nabla \varphi \, dx 
	\quad \text{as } n \to \infty.
	\end{equation}
	For the nonlocal term, we observe that
	\[
	\left( 
	\frac{\left[ u_n(x) - u_n(y)\right]^{p-1}}{|x - y|^{\frac{(p-1)(N + sp)}{p}}}
	\right)
	\quad \text{is bounded in } L^{\frac{p}{p-1}}(\mathbb{R}^{2N}).
	\]
	Therefore,
	\[
	\frac{\left[ u_n(x) - u_n(y)\right]^{p-1}}{|x - y|^{\frac{(p-1)(N + sp)}{p}}}
	\rightharpoonup
	\frac{\left[ u(x) - u(y)\right]^{p-1}}{|x - y|^{\frac{(p-1)(N + sp)}{p}}}
	\quad \text{weakly in } L^{\frac{p}{p-1}}(\mathbb{R}^{2N}).
	\]
	The latter convergence follows from the fact that \( u_n(x) \to u(x) \) pointwise, and thus
	\[
	\frac{\left[ u_n(x) - u_n(y)\right]^{p-1}}{|x - y|^{\frac{(p-1)(N + sp)}{p}}}
	\to 
	\frac{\left[ u(x) - u(y)\right]^{p-1}}{|x - y|^{\frac{(p-1)(N + sp)}{p}}}
	\quad \text{a.e. in } \mathbb{R}^{2N}.
	\]
	Hence, for all \( \varphi \in \mathbf{WL} \),
	\begin{align*}
	&\iint_{\mathbb{R}^{2N}} 
	\frac{\left[ u_n(x) - u_n(y)\right] ^{p-1}
		(\varphi(x) - \varphi(y))}{|x - y|^{N + sp}} \, dx \, dy  \to 
	\iint_{\mathbb{R}^{2N}} 
	\frac{\left[ u(x) - u(y)\right] ^{p-1}
		(\varphi(x) - \varphi(y))}{|x - y|^{N + sp}} \, dx \, dy 
	\, \text{as } n \to \infty.
	\end{align*}
	Using similar arguments together with H\"{o}lder's inequality, the sequences  
	\( (u_n^{2m+1}) \), \( (g_n u_n^{m}) \), and \( (d(x)^{-\gamma} u_{n}^{\delta}) \)  
	are uniformly bounded in \( L^{\frac{2(m+1)}{2m+1}}(\Omega) \).  
	For any \( \varphi \in \mathbf{WL} \), we obtain
	\begin{gather*}
	\lim_{n \to \infty} \int_{\Omega} u_n^{2m+1} \varphi \, dx 
	= \int_{\Omega} u^{2m+1} \varphi \, dx, \quad
	\lim_{n \to \infty} \int_{\Omega} g_n u_n^{m} \varphi \, dx 
	= \int_{\Omega} g_{0} u^{m} \varphi \, dx, \\
	\lim_{n \to \infty} \int_{\Omega} d(x)^{-\gamma} u_{n}^{\delta} \varphi \, dx 
	= \int_{\Omega} d(x)^{-\gamma} u^{\delta} \varphi \, dx.
	\end{gather*}
	Therefore, by passing to the limit in \eqref{3equ4}, we conclude that \( u \) is a weak solution of \eqref{PP2}, that is,  
	\begin{equation}\label{3weakform2}
	\begin{aligned}
	& \int_{\Omega} u^{2m+1} \varphi\,dx
	+ \lambda \int_{\Omega} |\nabla u|^{p-2} \nabla u \cdot \nabla \varphi\,dx  \\
	&\quad 
	+ \lambda \iint_{\mathbb{R}^{2N}} 
	\frac{[u(x) - u(y)]^{p-1}\big(\varphi(x) - \varphi(y)\big)}{|x - y|^{N + sp}}\,dx\,dy \\
	&= \int_{\Omega} \Big( g_{0}(x) u^{m} + \lambda m\, d(x)^{-\gamma} u^{\delta} \Big) \varphi\,dx, 
	\quad \forall\, \varphi \in \mathbf{WL}.
	\end{aligned}
	\end{equation}
	Now, the boundedness \( u \in L^{\infty}(\Omega) \) follows by reproducing the main steps of the proof in \cite[Theorem 3.1]{brasco2016second}.  
	For every fixed \( \epsilon > 0 \), we introduce \( u_{\epsilon} = u + \epsilon \).  
	Given \( \beta \geq 1 \), we select the test function \( \varphi_{\epsilon} = u_{\epsilon}^{\beta} - \epsilon^{\beta} \in \mathbf{WL} \) in \eqref{3weakform2}, which yields
	\begin{equation*}
	\begin{aligned}
	& \lambda \underbrace{\int_{\Omega} |\nabla u|^{p-2} \nabla u \cdot \nabla \varphi_{\epsilon} \,dx}_{\textbf{I}_{1}}    + \lambda \underbrace{\iint_{\mathbb{R}^{2N}} 
		\frac{[u(x) - u(y)]^{p-1}\big(\varphi_{\epsilon} (x) - \varphi_{\epsilon} (y)\big)}{|x - y|^{N + sp}}\,dx\,dy}_{\textbf{I}_{2}}  \\
	&\qquad\leq \underbrace{\int_{\Omega} \Big( g_{0}(x) u^{m} + \lambda m\, d(x)^{-\gamma} u^{\delta} \Big) \varphi_{\epsilon} \,dx}_{\textbf{I}_{3}} .
	\end{aligned}
	\end{equation*}
	\textbf{Estimate of} \( \textbf{I}_{1}\):  
	\[
	\textbf{I}_{1} = \beta \left(\dfrac{p}{\beta + p -1}\right)^{p} 
	\int_{\Omega} \big| \nabla u_{\epsilon}^{\frac{\beta + p -1}{p}}\big|^{p} dx \geq 0.
	\]
	\textbf{Estimate of} \( \textbf{I}_{2} \):  
	Applying the inequality established in \cite[Lemma A.2]{brasco2016second}, we obtain
	\begin{equation*}
	\textbf{I}_{2} \geq 
	\beta \left(\dfrac{p}{\beta + p - 1}\right)^{p} 
	\iint_{\mathbb{R}^{2N}}
	\frac{\big| u_{\epsilon}(x)^{\frac{\beta + p - 1}{p}} - u_{\epsilon}(y)^{\frac{\beta + p - 1}{p}} \big|^{p}}
	{|x - y|^{N + s p}}\,dx\,dy.
	\end{equation*}
	By virtue of Theorem~\ref{thm3}, we deduce that
	\begin{align*}
	\textbf{C}\left( \int_{\Omega} 
	\big( u_{\epsilon}^{\frac{\beta + p - 1}{p}} - \epsilon^{\frac{\beta + p - 1}{p}} \big)^{p^{*}_{s}} dx \right)^{p/p^*_{s}}
	&\leq 
	\iint_{\mathbb{R}^{2N}}
	\frac{\big| u_{\epsilon}(x)^{\frac{\beta + p - 1}{p}} - u_{\epsilon}(y)^{\frac{\beta + p - 1}{p}} \big|^{p}}
	{|x - y|^{N + s p}}\,dx\,dy,
	\end{align*}
	where \( \textbf{C} > 0 \) denotes a positive constant.  
	Applying the triangle inequality, the left-hand side of the above estimate can be bounded from below as
	
	\begin{equation*}
	\left(  \int_{\Omega} \left(  u_{\epsilon}^{\frac{\beta + p - 1}{p}} \right) ^{p^{*}_{s}} dx \right) ^{p/p^*_{s}} 
	- \epsilon^{\beta + p - 1} |\Omega|^{p/p^*_{s}}
	\leq 
	\left(  \int_{\Omega} 
	\left(  u_{\epsilon}^{\frac{\beta + p - 1}{p}} - \epsilon^{\frac{\beta + p - 1}{p}} \right) ^{p^{*}_{s}} dx\right) ^{p/p^*_{s}}.
	\end{equation*}
	It is straightforward to observe that
	\begin{equation*}
	\epsilon^{\beta + p -1} |\Omega| ^{p/p^*_s}
	\leq 
	\frac{1}{\beta}\left( \frac{\beta + p - 1}{p}\right) ^{p}
	|\Omega| ^{\frac{p}{p^{*}_{s}}-1} 
	\int_{\Omega }u_{\epsilon}^{\beta + p -1}\,dx.
	\end{equation*}
	Consequently, we obtain
	\begin{equation*}
	\textbf{I}_{2} + \textbf{C} |\Omega| ^{\frac{p}{p^{*}_{s}}-1}  \int_{\Omega }u_{\epsilon}^{\beta + p -1}\,dx
	\geq  
	\textbf{C}\, \beta \left(\frac{p}{\beta + p - 1}\right)^{p} 
	\left( \int_{\Omega} \left(  u_{\epsilon}^{\frac{\beta + p - 1}{p}} \right) ^{p^{*}_{s}} dx \right) ^{p/p^*_{s}}.
	\end{equation*}
	\textbf{Estimate of} \( \textbf{I}_{3} \). We distinguish two cases. In the first case, when \( m = 0 \), we employ the inequality \( 1 \leq \epsilon^{1 - p} u_{\epsilon}^{p - 1} \), together with Hölder’s and interpolation inequalities, where 
	\[
	\frac{1}{p r'} = \frac{\alpha}{p^{*}_{s}} + \frac{1 - \alpha}{p}, \quad 0 \leq \alpha \leq 1,
	\]
	and we note that for \( r > \frac{N}{s p} \), it ensures \( p < p r' < p^{*}_{s} \), with \( r' = \frac{r}{r - 1} \). It follows that
	\begin{equation*}
	\begin{aligned}
	\textbf{I}_{3}  
	&\leq \int_{\Omega} g_{0}(x) u_{\epsilon}^{\beta} \, dx 
	\leq \epsilon^{1 - p} \int_{\Omega} g_{0}(x) u_{\epsilon}^{\beta + p - 1} \, dx \\
	&\leq \epsilon^{1 - p} \| g_{0}\|_{L^r(\Omega)} 
	\left( \int_{\Omega} \left( u_{\epsilon}^{\frac{\beta + p - 1}{p}} \right)^{p r'} dx \right)^{1/r'} \\
	&\leq \epsilon^{1 - p} \| g_{0}\|_{L^r(\Omega)} 
	\left( \int_{\Omega} u_{\epsilon}^{\beta + p - 1} dx \right)^{1 - \alpha} 
	\left( \int_{\Omega} \left( u_{\epsilon}^{\frac{\beta + p - 1}{p}} \right)^{p^{*}_{s}} dx \right)^{\frac{p \alpha}{p^{*}_{s}}}.
	\end{aligned}
	\end{equation*}
	Applying Young’s inequality, we obtain
	\begin{equation*}
	\begin{aligned}
	\textbf{I}_{3} 
	\leq \epsilon^{1 - p} \| g_{0}\|_{L^r(\Omega)} 
	\left[ \eta 
	\left( \int_{\Omega} \left( u_{\epsilon}^{\frac{\beta + p - 1}{p}} \right)^{p^{*}_{s}} dx \right)^{p / p^{*}_{s}}
	+ C_{\eta} \int_{\Omega} u_{\epsilon}^{\beta + p - 1} dx \right],
	\end{aligned}
	\end{equation*}
	where \( C_{\eta} = \eta^{- \frac{1}{\alpha - 1}} \). Hence, we deduce that
	\begin{equation*}
	\begin{aligned}
	&\textbf{C}\, \beta 
	\left( \frac{p}{\beta + p - 1} \right)^{p} 
	\left( \int_{\Omega} \left( u_{\epsilon}^{\frac{\beta + p - 1}{p}} \right)^{p^{*}_{s}} dx \right)^{p / p^{*}_{s}} \\
	&\leq \epsilon^{1 - p} \| g_{0}\|_{L^r(\Omega)} 
	\left[ \eta 
	\left( \int_{\Omega} \left( u_{\epsilon}^{\frac{\beta + p - 1}{p}} \right)^{p^{*}_{s}} dx \right)^{p / p^{*}_{s}}
	+ C_{\eta} \int_{\Omega} u_{\epsilon}^{\beta + p - 1} dx \right] \\
	&\quad + \textbf{C} |\Omega|^{\frac{p}{p^{*}_{s}} - 1} 
	\int_{\Omega} u_{\epsilon}^{\beta + p - 1} dx.
	\end{aligned}
	\end{equation*}
	Now, by choosing
	\[
	\eta = \frac{\mathbf{C} \, \beta \, \epsilon^{\,p - 1}}{2 \, \| g_{0}\|_{L^r(\Omega)}} 
	\left( \frac{p}{\beta + p - 1} \right)^{p} > 0,
	\]
	we deduce the following estimate:
	\begin{equation*}
	\begin{aligned}
	\left( \int_{\Omega} \left( u_{\epsilon}^{\frac{\beta + p - 1}{p}} \right)^{p^{*}_{s}} \, dx \right)^{\frac{p}{p^{*}_{s}}}
	&\leq \frac{2}{\mathbf{C} \, \beta} 
	\left( \frac{\beta + p - 1}{p} \right)^{p} \\
	&\quad \times \Bigg[ 
	\epsilon^{1 - p} \, \| g_{0}\|_{L^r(\Omega)} \, C_{\eta} 
	+ \mathbf{C} \, |\Omega|^{\frac{p}{p^{*}_{s}} - 1} 
	\Bigg] 
	\int_{\Omega} u_{\epsilon}^{\beta + p - 1} \, dx.
	\end{aligned}
	\end{equation*}
	Next, we set
	\[
	\epsilon = 
	\left(
	\frac{\mathbf{C} \, |\Omega|^{\frac{p}{p^{*}_{s}} - 1}}{\| g_{0}\|_{L^r(\Omega)} \, C_{\eta}}
	\right)^{\frac{1}{1 - p}}.
	\]
	It then follows that
	\begin{equation*}
	\begin{aligned}
	\left( \int_{\Omega} \left( u_{\epsilon}^{\frac{\beta + p - 1}{p}} \right)^{p^{*}_{s}} \, dx \right)^{\frac{p}{p^{*}_{s}}} 
	&\leq \frac{4}{\beta} 
	\left( \frac{\beta + p - 1}{p} \right)^{p} |\Omega|^{\frac{p}{p^{*}_{s}} - 1} 
	\int_{\Omega} u_{\epsilon}^{\beta + p - 1} \, dx.
	\end{aligned}
	\end{equation*}
	In the second case, when \( m > 0 \), noting that \( \delta < m < p - 1 \), we use the estimates
	\[
	u_{\epsilon}^{m+\beta} \leq \epsilon^{m+1-p} u_{\epsilon}^{\beta + p - 1}, 
	\quad 
	u_{\epsilon}^{\delta+\beta} \leq \epsilon^{\delta+1-p} u_{\epsilon}^{\beta + p - 1},
	\]
	together with Hölder's and interpolation inequalities, where
	\[
	\frac{1}{p r'} = \frac{\alpha}{p^{*}_{s}} + \frac{1 - \alpha}{p}, \quad 0 \leq \alpha \leq 1.
	\]
	Observe that for \( r > \frac{N}{s p} \), it holds that \( p < p r' < p^{*}_{s} \), and we also assume \( r < \frac{1}{\gamma} \) when \( \gamma > 0 \). Consequently, we obtain
	\begin{equation*}
	\begin{aligned}
	\mathbf{I}_{3} 
	&\leq \int_{\Omega} g_{0}(x) \, u_{\epsilon}^{m + \beta} \, dx
	+ \lambda m \int_{\Omega} d(x)^{- \gamma} \, u_{\epsilon}^{\delta + \beta} \, dx \\[1ex]
	&\leq \epsilon^{m + 1 - p} \int_{\Omega} g_{0}(x) \, u_{\epsilon}^{\beta + p - 1} \, dx
	+ \lambda m \, \epsilon^{\delta + 1 - p} \int_{\Omega} d(x)^{- \gamma} \, u_{\epsilon}^{\beta + p - 1} \, dx \\[1ex]
	&\leq 
	\left[ \epsilon^{m + 1 - p} \, \| g_{0} \|_{L^r(\Omega)} 
	+ \lambda m \epsilon^{\delta + 1 - p} \Big( \int_{\Omega} d(x)^{-r\gamma} \, dx \Big)^{1/r} \right] 
	\left( \int_{\Omega} \Big( u_{\epsilon}^{\frac{\beta + p - 1}{p}} \Big)^{p r'} \, dx \right)^{1/r'} \\[1ex]
	&\leq 
	\left[ \epsilon^{m + 1 - p} \, \| g_{0} \|_{L^r(\Omega)} 
	+\lambda m \epsilon^{\delta + 1 - p} \left(  \int_{\Omega} d(x)^{-r\gamma} \, dx \right) ^{1/r} \right] \\ 
	&\qquad \times
	\left( \int_{\Omega} u_{\epsilon}^{\beta + p - 1} \, dx \right)^{1 - \alpha} 
	\left( \int_{\Omega} \left( u_{\epsilon}^{\frac{\beta + p - 1}{p}} \right) ^{p^{*}_{s}} \, dx \right)^{\frac{p \alpha}{p^{*}_{s}}}.
	\end{aligned}
	\end{equation*}
	Applying Young's inequality, we deduce
	\begin{equation*}
	\begin{aligned}
	\mathbf{I}_{3} 
	&\leq 
	\Big(
	\epsilon^{m + 1 - p} \, \| g_{0} \|_{L^r(\Omega)} 
	+ \lambda m \epsilon^{\delta + 1 - p} 
\left( \int_{\Omega} d(x)^{-r\gamma} \, dx \Big)^{1/r} 
\right) \\[1ex]
	&\quad \times
	\Bigg[
	\eta
	\left( \int_{\Omega} \Big( u_{\epsilon}^{\frac{\beta + p - 1}{p}} \Big)^{p^{*}_{s}} \, dx \right)^{p / p^{*}_{s}}
	+ C_{\eta} \int_{\Omega} u_{\epsilon}^{\beta + p - 1} \, dx
	\Bigg],
	\end{aligned}
	\end{equation*}
	where \( C_{\eta} = \eta^{- \frac{1}{\alpha- 1}} \). Hence, we arrive at
	\begin{equation*}
	\begin{aligned}
	&\mathbf{C}\, \beta 
	\left( \frac{p}{\beta + p - 1} \right)^{p}
	\left( \int_{\Omega} \Big( u_{\epsilon}^{\frac{\beta + p - 1}{p}} \Big)^{p^{*}_{s}} \, dx \right)^{p / p^{*}_{s}} \\[1ex]
	&\leq 
\left[ 
	\epsilon^{m + 1 - p} \, \| g_{0} \|_{L^r(\Omega)} 
	+ \lambda m \epsilon^{\delta + 1 - p} 
	\Big( \int_{\Omega} d(x)^{-r\gamma} \, dx \Big)^{1/r}
\right] \\[1ex]
	&\quad \times
\left[ 
	\eta
	\left( \int_{\Omega} \Big( u_{\epsilon}^{\frac{\beta + p - 1}{p}} \Big)^{p^{*}_{s}} \, dx \right)^{p / p^{*}_{s}}
	+ C_{\eta} \int_{\Omega} u_{\epsilon}^{\beta + p - 1} \, dx
	\right] + \textbf{C} |\Omega|^{\frac{p}{p^{*}_{s}} - 1} 
	\int_{\Omega} u_{\epsilon}^{\beta + p - 1} dx.
	\end{aligned}
	\end{equation*}
	Choosing
	\[
	\eta =
	\frac{\mathbf{C} \, \beta}{2 \left( \epsilon^{m + 1 - p} \| g_{0}\|_{L^r(\Omega)} + \lambda m \epsilon^{\delta + 1 - p}
		\left( \int_{\Omega} d(x)^{-r\gamma} \, dx \right)^{1/r}\right) }
	\left( \frac{p}{\beta + p - 1} \right)^{p} > 0,
	\]
	we obtain
	
	\begin{equation*}
	\begin{aligned}
	&\left( \int_{\Omega} \Big( u_{\epsilon}^{\frac{\beta + p - 1}{p}} \Big)^{p^{*}_{s}} \, dx \right)^{p / p^{*}_{s}} \\
	&\leq 
	\frac{2}{\mathbf{C} \, \beta} 
	\left( \frac{\beta + p - 1}{p} \right)^{p} \\
	&\quad \times
	\Bigg[
	\epsilon^{m + 1 - p} \, \| g_{0} \|_{L^r(\Omega)} \, C_{\eta}
	+ \lambda m \epsilon^{\delta + 1 - p} 
	\Big( \int_{\Omega} d(x)^{-2\gamma} \, dx \Big)^{1/2} \, C_{\eta}
	+ \mathbf{C} \, |\Omega|^{\frac{p}{p^{*}_{s}} - 1}
	\Bigg] \\
	&\quad \times \int_{\Omega} u_{\epsilon}^{\beta + p - 1} \, dx.
	\end{aligned}
	\end{equation*}
	Depending on the magnitude of the ratio
	\[
	\dfrac{\mathbf{C} |\Omega|^{\frac{p}{p^{*}_{s}} - 1}}{C_{\eta}  \left( 
		\| g_{0}\|_{L^r(\Omega)} + \lambda m \left( \int_{\Omega} d(x)^{-2\gamma} \, dx \right)^{1/2}\right) },
	\]
	we choose \(\epsilon\) as
	\[
	\epsilon =
	\begin{cases}
	\left(
	\dfrac{\mathbf{C} |\Omega|^{\frac{p}{p^{*}_{s}} - 1}}{C_{\eta}  \left( 
		\| g_{0}\|_{L^r(\Omega)} + \lambda m \left( \int_{\Omega} d(x)^{-2\gamma} \, dx \right)^{1/2}\right)}
	\right)^{1/(\delta + 1 - p)}, & \text{if the ratio} \geq 1, \\[20pt]
	\left(
	\dfrac{\mathbf{C} |\Omega|^{\frac{p}{p^{*}_{s}} - 1}}{C_{\eta}  \left( 
		\| g_{0}\|_{L^r(\Omega)} + \lambda m \left( \int_{\Omega} d(x)^{-2\gamma} \, dx \right)^{1/2}\right)}
	\right)^{1/(m + 1 - p)}, & \text{otherwise},
	\end{cases}
	\]
	we get that
	\begin{equation*}
	\begin{aligned}
	&\left( \int_{\Omega} \left( u_{\epsilon}^{\frac{\beta + p - 1}{p}} \right)^{p^{*}_{s}} dx \right)^{\!\!p / p^{*}_{s}}\leq \frac{4}{\beta} 
	\left( \frac{\beta + p - 1}{p} \right)^{p} 
	|\Omega|^{\frac{p}{p^{*}_{s}} - 1}
	\int_{\Omega} u_{\epsilon}^{\beta + p - 1} \, dx.
	\end{aligned}
	\end{equation*}
Finally, setting $\nu = \beta + p - 1$, we deduce in both cases that
\begin{equation}\label{3equ5}
\left( \int_{\Omega} u_{\epsilon}^{\left( \frac{p^{*}_{s}}{p} \right)\nu} \, dx \right)^{\!\frac{1}{\left( \frac{p^{*}_{s}}{p} \right)\nu}}
\leq 
\left[ 4\, |\Omega|^{\frac{p}{p^{*}_{s}} - 1} \right]^{\!\frac{1}{\nu}}
\left( \frac{\nu}{p} \right)^{\!\frac{p}{\nu}}
\left( \int_{\Omega} u_{\epsilon}^{\nu} \, dx \right)^{\!\frac{1}{\nu}}.
\end{equation}
Next, we iterate this estimate by considering the sequence of exponents
\[
\begin{cases}
\nu_{0} = 1, \\
\nu_{n+1} = \left( \dfrac{p^{*}_{s}}{p} \right)\nu_{n}
= \left( \dfrac{p^{*}_{s}}{p} \right)^{n+1}, \quad n \geq 0.
\end{cases}
\]
Clearly,
\[
\nu_{n} \to \infty \quad \text{as } n \to \infty,
\]
and the recursive inequality becomes
\begin{equation}\label{3equ6}
\| u_{\epsilon} \|_{L^{\nu_{n+1}}(\Omega)}
\leq
\left[ 4\, |\Omega|^{\frac{p}{p^{*}_{s}} - 1} \right]^{\!\frac{1}{\nu_{n}}}
\left( \frac{\nu_{n}}{p} \right)^{\!\frac{p}{\nu_{n}}}
\| u_{\epsilon} \|_{L^{\nu_{n}}(\Omega)} .
\end{equation}
Iterating \eqref{3equ5}–\eqref{3equ6} yields
\[
\| u_{\epsilon} \|_{L^{\nu_{n+1}}(\Omega)}
\leq
\big[ 4\, |\Omega|^{\frac{p}{p^{*}_{s}} - 1} \big]^{\sum_{i=0}^{n} \frac{1}{\nu_i}}
\prod_{i=0}^{n} \left( \frac{\nu_i}{p} \right)^{\!\frac{p}{\nu_i}}
\| u_{\epsilon} \|_{L^{1}(\Omega)}.
\]
Observe that
\[
\sum_{n=0}^{\infty} \frac{1}{\nu_n}
= \sum_{n=0}^{\infty} \left( \frac{p}{p^{*}_{s}} \right)^{n}
= \frac{p^{*}_{s}}{p^{*}_{s} - p},
\qquad
\prod_{n=0}^{\infty}
\left( \frac{\nu_n}{p} \right)^{\!\frac{p}{\nu_n}}
< \infty.
\]
Passing to the limit as $n \to \infty$ in \eqref{3equ6}, we obtain
\[
\| u_{\epsilon} \|_{L^{\infty}(\Omega)}
\leq
\frac{C'}{|\Omega|}
\| u_{\epsilon} \|_{L^{1}(\Omega)}
\leq
\frac{C'}{|\Omega|}
\big( \| u \|_{L^{1}(\Omega)} + \epsilon |\Omega| \big)
< \infty,
\]
for some constant $C' > 0$.  
Finally, to establish \eqref{equ63}, we apply the same argument as in the proof of the weak comparison principle presented in Theorem~\ref{theorem2}.
\end{proof} 

\subsection{Existence and Uniqueness of Weak Solutions and Accretivity Results for Problem~\eqref{equ22}}

\begin{proof}[\textbf{Proof of Corollary} \ref{cor1}] 
	Let $m>0$. We define the energy functional $\xi$ on 
	$\dot{V}_{+}^{m+1} \cap L^{2}(\Omega)$ by
	\begin{equation*}
	\begin{aligned}
	\xi(v)
	&:= \mathcal{J}\!\left(v^{\frac{1}{m+1}}\right)
	\quad \text{(see~\eqref{3equ1})} \\
	&= \frac{1}{2(m+1)} \int_{\Omega} v^{2}\,dx
	+ \frac{\lambda}{p} \int_{\Omega}
	\left|\nabla \left(v^{\frac{1}{m+1}}\right)\right|^{p}\,dx
	+ \frac{\lambda}{p} \iint_{\mathbb{R}^{2N}}
	\frac{\left|v^{\frac{1}{m+1}}(x)-v^{\frac{1}{m+1}}(y)\right|^{p}}
	{|x-y|^{N+sp}}\,dx\,dy \\
	&\quad - \frac{1}{m+1} \int_{\Omega} g_{0}(x)\,v \,dx
	- \frac{\lambda m}{\delta+1} \int_{\Omega}
	d(x)^{-\gamma} v^{\frac{\delta+1}{m+1}} \,dx .
	\end{aligned}
	\end{equation*}
Let $u_{0}$ be the weak solution of \eqref{PP2} obtained via the minimization approach established in Theorem~\ref{theorem3}.
We define $v_{0}=u_{0}^{m+1}$. By the regularity and positivity properties of the solution $u_{0}$, it follows that,  for every \( \alpha \in [s,1) \) with \( \alpha \neq \frac{ps}{p-1} \),
\[
v_{0}\in \dot{V}_{+}^{m+1}\cap \mathcal{M}_{1,\alpha}^{\frac{1}{m+1}}(\Omega).
\]
Moreover, it is straightforward to verify that $v_{0}$ is a minimizer of the functional $\xi$ over the set $\dot{V}_{+}^{m+1}\cap L^{2}(\Omega)$. We also observe that $\xi$ is Gâteaux differentiable. Let $\Psi\geq 0$ be such that
\[
\Psi\in L_{d^{m+1}}^{\infty}(\Omega).
\]
Then there exists $t_{0}=t_{0}(\Psi)>0$ such that $v_{0}+t\Psi>0$ for all $t\in(-t_{0},t_{0})$.
Consequently, we have
\begin{align*}
0
&=\lim_{t\to 0}\frac{\xi(v_{0}+t\Psi)-\xi(v_{0})}{t}
:=\left.\frac{d}{dt}\xi(v_{0}+t\Psi)\right|_{t=0}  \\
&=\frac{1}{m+1}\int_{\Omega} v_{0}\,\Psi \, dx
+\frac{\lambda}{m+1}\int_{\Omega}
\left|\nabla\left(v_{0}^{\frac{1}{m+1}}\right)\right|^{p-1}
\nabla\left(v_{0}^{\frac{1}{m+1}}\right)\cdot
\nabla\left(\frac{\Psi}{v_{0}^{\frac{m}{m+1}}}\right)\, dx  \\
&\quad
+\frac{\lambda}{m+1}\iint_{\mathbb{R}^{2N}}
\frac{\left[ v_{0}^{\frac{1}{m+1}}(x)-v_{0}^{\frac{1}{m+1}}(y)\right] ^{p-1}
	\big(v_{0}^{\frac{1}{m+1}}(x)-v_{0}^{\frac{1}{m+1}}(y)\big)}
{|x-y|^{N+sp}}   \\
&\qquad\qquad\times
\left[
\left(\frac{\Psi}{v_{0}^{\frac{m}{m+1}}}\right)(x)
-\left(\frac{\Psi}{v_{0}^{\frac{m}{m+1}}}\right)(y)
\right]\, dx\, dy  \\
&-\frac{1}{m+1}\int_{\Omega} g_{0}(x)\,\Psi \, dx
- \frac{\lambda m}{m+1}\int_{\Omega} d(x)^{-\gamma}
v_{0}^{\frac{\delta-m}{m+1}}\,\Psi \, dx .
\end{align*}
This implies that $v_{0}$ satisfies \eqref{equ24}. On the other hand, let $v_{1},v_{2}\in \dot{V}_{+}^{m+1}\cap
\mathcal{M}_{1,\alpha}^{\frac{1}{m+1}}(\Omega)$ be two solutions of problem~\eqref{equ22}
satisfying \eqref{equ24}. Setting $v_{1}=u_{1}^{m+1}$ and $v_{2}=u_{2}^{m+1}$,
we see that $u_{1}$ and $u_{2}$ are solutions of \eqref{PP2} in the sense of
Definition~\ref{definition2} with potential $g_{0}$.
The contraction principle \eqref{contraction} then yields $v_{1}=v_{2}$.
Finally, \eqref{equ27} follows directly from \eqref{contraction}.
\end{proof}
\begin{proof}[\textbf{Proof of Corollary} \ref{cor22}] 
We begin by noting that the solution $u_{0} \in L^{\infty}(\Omega)$ corresponding to $g_{0} \in L^{2}(\Omega)$, obtained in Theorem~\ref{theorem4}, is a global minimizer of the functional $\mathcal{J}$ defined in \eqref{3equ1}. This follows from the global minimization method as in Theorem~\ref{theorem3}, together with the result of Theorem~\ref{theorem2}. As in the proof of Corollary \ref{cor1}, we define the energy functional $\xi$ on 
	$ \dot{V}_{+}^{m+1} \cap L^2(\Omega)$ by $ 	\xi(v) = \mathcal{J}(v^{\frac{1}{m+1}}). $ Setting $v_0 = u_0^{m+1}$, we observe that $v_0 \in \dot{V}_{+}^{m+1} \cap L^{\infty}(\Omega)$ and that $v_0$ is a minimizer of the functional $\xi$.
By \eqref{equ65}, it follows that 
	\[
	v_0(x) \geq c \, d^{m+1}(x) \quad \text{a.e. in } \Omega.
	\]
	Let $\Psi$ satisfy \eqref{equuu}, i.e.,
	\[
	|\Psi| \in L_{d^{m+1}}^{\infty}(\Omega)
	\quad \text{and} \quad
	\frac{|\nabla \Psi|}{d(\cdot)^{m}} \in L^{p}(\Omega).
	\]
	Then, for sufficiently small $t$, we have $v_0 + t \Psi > 0$. Using the Gâteaux differentiability of $\xi$, we deduce that $v_0$ satisfies \eqref{equu}. Finally, \eqref{1equ27} yields \eqref{equ63}.
\end{proof}
\subsection{Existence and Uniqueness of Weak Solutions for Problem~\eqref{P7}}
\begin{proof}[\textbf{Proof of Theorem}~\ref{the2}]
	To problem~\eqref{P7}, we associate the energy functional $\mathcal{L}$ defined on the space
	$W^{1,p}_{0}(\Omega)$ in the case $m=0$, and on the space $ \textbf{WL} := W^{1,p}_{0}(\Omega)\cap L^{2(m+1)}(\Omega) $ when $m>0$. The space $\textbf{WL}$ is endowed with the Cartesian norm $ 	\|u\|_{\textbf{WL}}
	:= \|u\|_{W^{1,p}_{0}(\Omega)} + \|u\|_{L^{2(m+1)}(\Omega)}. $ The functional $\mathcal{L}$ is defined by
	\[
	\begin{aligned}
	\mathcal{L}(u)
	&:= \frac{1}{p}\int_{\Omega}|\nabla u|^{p}\,dx
	+ \frac{1}{p}\iint_{\mathbb{R}^{2N}}
	\frac{|u(x)-u(y)|^{p}}{|x-y|^{N+sp}}\,dx\,dy \\
	&\quad
	- \frac{1}{m+1}\int_{\Omega} b(x)\,(u^{+})^{m+1}\,dx
	- \frac{m}{\delta+1}\int_{\Omega} d(x)^{-\gamma}\,(u^{+})^{\delta+1}\,dx .
	\end{aligned}
	\]
Arguing as in the proof of Theorem~\ref{theorem3}, we establish the existence of a nonnegative global minimizer of the functional $\mathcal{L}$. Moreover, following the same approach as in Theorem~\ref{theorem3}, we infer that the minimizer $u_{0}$ belongs to $L^{\infty}(\Omega)$. In addition, by \cite[Theorem~1.1]{antonini2025global}, we conclude that $u_{0}\in C^{1,\xi}(\overline{\Omega})$ for some $\xi\in(0,1)$. Furthermore, owing to \cite[Proposition~6.1]{antonini2025global}, we obtain that $u_{0}>0$ in $\Omega$. Hence, by Hopf's lemma (see the proof of \cite[Theorem~1.2]{antonini2025global}), there exists a constant $c_{1}>0$ such that $ u_{0}\ge c_{1}\,d(x) $ in $ \Omega. $ Now, by Theorem \ref{theorem1}, there exists a constant $c_{2}>0$ satisfying $ u_{0}\le c_{2}\,d(x)^{\alpha} $ in $ \Omega, $  for every \( \alpha \in [s,1) \) with \( \alpha \neq \frac{ps}{p-1} \). As a consequence, we deduce that $u_{0}\in \mathcal{M}^{1}_{1,\mu}(\Omega)$.   To prove the uniqueness result, let $u_{1},u_{2}\in C^{1,\xi}(\overline{\Omega})\cap \mathcal{M}^{1}_{1,\mu}(\Omega)$ be two weak solutions of~\eqref{P7}.  In the case $m=0$, we choose the test functions $ u_{1}-u_{2} $ and $ u_{2}-u_{1}, $ in~\eqref{equuuu}, satisfied by $u_{1}$ and $u_{2}$, respectively. Consequently, summing the resulting identities, we obtain
\begin{equation*}
\begin{aligned}
&\int_{\Omega} 
\Big( |\nabla u_{1}|^{p-2} \nabla u_{1}
- |\nabla u_{2}|^{p-2} \nabla u_{2} \Big)
\cdot \nabla (u_{1}-u_{2})\, dx \\
&\quad + \iint_{\mathbb{R}^{2N}} 
\frac{
	\Big( [u_{1}(x)-u_{1}(y)]^{p-1}
	- [u_{2}(x)-u_{2}(y)]^{p-1} \Big)
	\Big( (u_{1}-u_{2})(x)-(u_{1}-u_{2})(y) \Big)
}
{|x-y|^{N+sp}} \, dx\, dy
\leq 0.
\end{aligned}
\end{equation*}
By applying the elementary inequalities established in~\cite{farina2013monotonicity} (see also \cite[Section~10]{lindqvist2017notes}), we infer that
\begin{equation*}
\int_{\Omega} \big| \nabla (u_{1}-u_{2}) \big|^{p}\, dx \leq 0,
\qquad \text{for } p \geq 2,
\end{equation*}
and
\begin{equation*}
C \left(
\int_{\Omega} \big| \nabla (u_{1}-u_{2}) \big|^{p}\, dx
\left\|
\left( |\nabla u_{1}| + |\nabla u_{2}| \right)^{\frac{p(2-p)}{2}}
\right\|_{L^{\frac{2}{2-p}}(\Omega)}^{-1}
\right)^{\frac{2}{p}}
\leq 0,
\qquad \text{for } p < 2.
\end{equation*}
These estimates lead to the conclusion that $u_{1} \equiv u_{2}$ in $\mathbb{R}^{N}$.  When $m>0$, since $u_{1},u_{2}\in L^{\infty}(\Omega)$  for any fixed $\varepsilon \in (0,1)$, we choose the test functions
\[
\frac{(u_{1}+\varepsilon)^{m+1}
	- (u_{2}+\varepsilon)^{m+1}}{(u_{1}+\varepsilon)^{m}}
\quad \text{and} \quad
\frac{ (u_{2}+\varepsilon)^{m+1}
	- (u_{1}+\varepsilon)^{m+1}}{(u_{2}+\varepsilon)^{m}},
\]
in~\eqref{equuuu}, satisfied by $u_{1}$ and $u_{2}$, respectively. Summing the resulting identities and proceeding as in the proof of Theorem~\ref{theorem2} and Lemma~\ref{Lem2}, while also applying Fatou’s lemma, we may pass to the limit and infer that
\begin{equation*}
\begin{aligned}
0 \leq\; &
\int_{\Omega} |\nabla u_{1}|^{p-2} \nabla u_{1} \cdot
\nabla \!\left( \frac{u_{1}^{m+1}-u_{2}^{m+1}}{u_{1}^{m}} \right) dx
+ \int_{\Omega} |\nabla u_{2}|^{p-2} \nabla u_{2} \cdot
\nabla \!\left( \frac{u_{2}^{m+1}-u_{1}^{m+1}}{u_{2}^{m}} \right) dx \\[2mm]
& + \iint_{\mathbb{R}^{2N}}
\frac{[u_{1}(x)-u_{1}(y)]^{p-1}}{|x-y|^{N+sp}}
\left(
\frac{u_{1}(x)^{m+1}-u_{2}(x)^{m+1}}{u_{1}(x)^{m}}
- \frac{u_{1}(y)^{m+1}-u_{2}(y)^{m+1}}{u_{1}(y)^{m}}
\right) dx\, dy \\[1mm]
& + \iint_{\mathbb{R}^{2N}}
\frac{[u_{2}(x)-u_{2}(y)]^{p-1}}{|x-y|^{N+sp}}
\left(
\frac{u_{2}(x)^{m+1}-u_{1}(x)^{m+1}}{u_{2}(x)^{m}}
- \frac{u_{2}(y)^{m+1}-u_{1}(y)^{m+1}}{u_{2}(y)^{m}}
\right) dx\, dy
\leq 0.
\end{aligned}
\end{equation*}
Invoking Lemma~\ref{Lem2} once again, we conclude that $ u_{1} \equiv u_{2}  $ in $  \mathbb{R}^{N}, $ since $m < p-1$.
\end{proof}

\section{Existence, Uniqueness, Regularity, and Asymptotic Analysis of problem~\eqref{P1}}\label{sec2}
In this section, we provide rigorous proofs of the main results regarding the existence, uniqueness, and regularity of solutions, as well as their asymptotic behavior, for the parabolic problem~\eqref{PP1}, building upon the results established in Section~3. In particular, we establish the following key findings.
\subsection{Existence, Uniqueness, and Regularity of Problem~\eqref{P1}}

\begin{proof}[\textbf{Proof of Theorem}~\ref{2thm1}]
The proof is organized into three principal steps.\\[2mm]
\textbf{Step~1: Existence of a Weak Solution.}
Let $n_0 \in \mathbb{N}^{*}$ and $T>0$. We introduce a uniform time discretization by setting $\Delta t = \frac{T}{n_0}$ and $t_n = n\Delta t$ for $n \in \{1,\ldots,n_0\}$. First, we construct an appropriate approximation of the function \(g\).  For each \(n \in \{1, \ldots, n_0\}\), we define
\[
g_{\Delta t}(t,x) := g^n(x) := \frac{1}{\Delta t} \int_{t_{n-1}}^{t_n} g(s,x)\,ds,
\quad (t,x) \in [t_{n-1},t_n) \times \Omega.
\]
Applying Jensen’s inequality, we obtain
\begin{equation}\label{4equ2}
\begin{aligned}
\|g_{\Delta t}\|^{2}_{L^2(Q_T)} 
&= \Delta t \sum_{n = 1}^{n_{0}} \|g^{n}\|^{2}_{L^2(\Omega)} 
= \Delta t \sum_{n = 1}^{n_{0}} \left\| \frac{1}{\Delta t} \int_{t_{n-1}}^{t_n} g(s,\cdot)\,ds \right\|^{2}_{L^2(\Omega)}  \\
&\leq \sum_{n = 1}^{n_{0}} \int_{t_{n-1}}^{t_n} \| g(s,\cdot)\|^{2}_{L^2(\Omega)} \,ds 
\leq \|g\|_{L^2(Q_T)}^2.
\end{aligned}
\end{equation}
Consequently, $g_{\Delta t} \in L^2(Q_T)$ and, in particular, $g^{n} \in L^2(\Omega)$ for all $n = 1,\dots,n_0$. We next show that
\begin{equation}\label{4equ9}
g_{\Delta t} \longrightarrow g \quad \text{strongly in } L^{2}(Q_T) \quad \text{as } \Delta t \to 0.
\end{equation}
To this end, let $(g_i)_{i\in\mathbb{N}} \subset C_c^\infty(Q_T)$ be a sequence such that $ g_i \to g $ in $ L^{2}(Q_T)  $ as $ i \to \infty. $ In view of \eqref{4equ2}, it follows that
\[
\bigl\| (g_i)_{\Delta t} - g_{\Delta t} \bigr\|_{L^{2}(Q_T)}
\le \| g_i - g \|_{L^{2}(Q_T)} \xrightarrow[i\to\infty]{} 0.
\]
Moreover, for each fixed $i$, the regularity of $g_i$ implies that
\[
(g_i)_{\Delta t} \to g_i \quad \text{strongly in } L^{2}(Q_T) \quad \text{as } \Delta t \to 0.
\]
Therefore, we infer that
\[
\begin{aligned}
\| g_{\Delta t} - g \|_{L^{2}(Q_T)}
\le {} & \bigl\| (g_i)_{\Delta t} - g_{\Delta t} \bigr\|_{L^{2}(Q_T)}
+ \bigl\| (g_i)_{\Delta t} - g_i \bigr\|_{L^{2}(Q_T)}
+ \| g_i - g \|_{L^{2}(Q_T)}.
\end{aligned}
\]
Letting first $\Delta t \to 0$ and then $i \to \infty$, we deduce that
\[
\| g_{\Delta t} - g \|_{L^{2}(Q_T)} \to 0 \quad \text{as } \Delta t \to 0,
\]
which establishes the desired strong convergence. Finally, another application of Jensen’s inequality yields
\[
\|g_{\Delta t}\|_{L^{\infty}(Q_T)} \leq \|g\|_{L^{\infty}(Q_T)}.
\]
Hence, $g_{\Delta t} \in L^{\infty}(Q_T)$ and, in particular, $g^{n} \in L^{\infty}(\Omega)$ for all $n=1,\dots,n_0$.  Next, we discretize problem~\eqref{PP1} in time by means of the implicit Euler scheme. Let the initial datum satisfy
\begin{equation}\label{4equ4}
u_0 \in W^{1,p}_0(\Omega)\cap \mathcal{M}^{1}_{\alpha',\alpha'}(\Omega),
\qquad
\alpha'=\frac{p}{p-1+\vartheta},
\end{equation}
where the parameter $\vartheta$ is assumed to fulfill $ 1+p\Bigl(\frac{1-s}{s}\Bigr)<\vartheta<2+\frac{1}{p-1}. $ For each $n\geq 1$, we denote by $u_n$ the weak solution to the associated discrete problem, which is defined for all $m\geq 0$.
	\begin{equation*}
	\left\{
	\begin{aligned}
	\left( \frac{u_n^{m+1} - u_{n-1}^{m+1}}{\Delta_t} \right) u_n^{m} - \Delta_{p}u_{n} + (-\Delta)^s_p u_n\;&= g^n(x) u_n^{m} + m d(x)^{-\gamma} u_{n}^{\delta} && \text{in } \Omega, \\
	u_n \;&> 0 && \text{in } \Omega, \\
	u_n \;&= 0 && \text{in } \mathbb{R}^N \setminus \Omega.
	\end{aligned}
	\right.
	\end{equation*}	
	Equivalently, this can be written as
	\begin{equation}\label{4equ1}
	\left\{
	\begin{aligned}
	u_n^{2m + 1} - \Delta_{t} \, \Delta_{p} u_n + \Delta_{t} \, (-\Delta)^s_p u_n\;&= \left( u_{n-1}^{m+1} + \Delta_{t} \, g^n(x)\right)  u_n^{m} + \Delta_{t} \, m d(x)^{-\gamma} u_{n}^{\delta} && \text{in } \Omega, \\
	u_n \;&> 0 && \text{in } \Omega, \\
	u_n \;&= 0 && \text{in } \mathbb{R}^N \setminus \Omega.
	\end{aligned}
	\right.
	\end{equation}
It is straightforward to verify that the sequence $(u_n)_{n=1,2,\dots,n_0}$ is \textbf{well defined}. Indeed, we first consider the case $n=1$. The existence and uniqueness of $ u_1 \in C^{1}(\overline{\Omega}) \cap \mathcal{M}^{1}_{1,\alpha}(\Omega), $ for all  $ \alpha \in [s,1), $ follow from Theorem~\ref{theorem3}. Notice that the condition  $2\delta-\textbf{d}\, \gamma >2m-1$
implies $\gamma<\frac{1}{\textbf{d}}$. The theorem is applied with $\lambda=\Delta t$ and $ g_0= u_0^{m+1} + \Delta t\,  g^{1} \in L^{\infty}(\Omega). $ Moreover, we observe that $ u_0^{m+1} + \Delta t\, g^{1} \geq \Delta t\, \underline{g} $ for a.e. $ x\in\Omega. $ Proceeding by induction, we obtain in the same manner the existence and uniqueness of the solution $u_n$ for every $n=2,3,\dots,n_0$, where
$ u_n \in C^{1}(\overline{\Omega}) \cap \mathcal{M}^{1}_{1,\alpha}(\Omega), $
for all  $ \alpha \in [s,1). $ 
\begin{claim}\label{claim1}
	For every $n \in \{0, 1, 2, \dots, n_0\}$, there exist two positive constants $c_1$ and $c_2$, independent of $n$, such that the following estimate holds:
	\begin{equation}\label{4equ3}
	c_1\, d(x) \leq u_n(x) \leq c_2\, d(x)^{\alpha'}, 
	\quad \text{for all } x \in \Omega.
	\end{equation}
\end{claim}
	\noindent 
Indeed, it suffices to construct suitable sub- and super-solutions, denoted by $\underline{u}$ and $\overline{u}$, respectively, such that  $ \underline{u} \leq u_n \leq \overline{u} $ in $ \Omega, $ for all $ n \in \{0,1,2,\ldots,n_0\}. $ We begin by constructing a uniform sub-solution. Following the approach employed in the proof of Theorem~\ref{theorem3}, we show that for each $k \in \mathbb{N} \setminus \{0\}$, the  following problem
\begin{equation}\label{4equ6}
\left\{
\begin{aligned}
-\Delta_p \tilde{u} + (-\Delta)^s_p \tilde{u} 
&= k^{-1}\Big(\underline{g}(x)\, \tilde{u} ^{m}
+ m\, d(x)^{-\gamma} \tilde{u} ^{\delta}\Big)
&& \text{in } \Omega,\\
\tilde{u}  &> 0
&& \text{in } \Omega,\\
\tilde{u}  &= 0
&& \text{in } \mathbb{R}^N \setminus \Omega,
\end{aligned}
\right.
\end{equation}
admits a unique weak solution  $ \tilde{u}_k \in C^{1}(\overline{\Omega}) \cap \mathcal{M}^{1}_{1,\alpha}(\Omega), $ for all  $ \alpha \in [s,1), $ where $\underline{g}$ is defined as in \eqref{equ1}.  Next, we show that the sequence  $(\tilde{u}_k) $ is decreasing. Indeed, let $k \geq 1$, and denote by $\tilde{u}_k$ and $\tilde{u}_{k+1}$ the respective weak solutions of the problem \eqref{4equ6}. Then, for any test functions $\Phi, \Psi \in W^{1, p}_{0}(\Omega)$ when $m = 0$, and $\Phi, \Psi \in \mathbf{WL}$ when $m > 0$, we have
\begin{align*}
&\int_{\Omega} |\nabla \tilde{u}_{k+1}|^{p-2}\nabla \tilde{u}_{k+1} \cdot \nabla \Phi \, dx
+ \iint_{\mathbb{R}^{2N}}
\frac{\left[ \tilde{u}_{k+1}(x)-\tilde{u}_{k+1}(y)\right] ^{p-1}(\Phi(x)-\Phi(y))}{|x-y|^{N+sp}} \, dx \, dy \\
&\qquad= (k+1)^{-1} \int_{\Omega} \Big(\underline{g}(x)\, \tilde{u}_{k+1}^{m} + m\, d(x)^{-\gamma}\, \tilde{u}_{k+1}^{\delta}\Big) \, \Phi \, dx,
\end{align*}
and
\begin{align*}
&\int_{\Omega} |\nabla \tilde{u}_{k}|^{p-2}\nabla \tilde{u}_{k} \cdot \nabla \Psi \, dx
+ \iint_{\mathbb{R}^{2N}}
\frac{\left[ \tilde{u}_{k}(x)-\tilde{u}_{k}(y)\right] ^{p-1}(\Psi(x)-\Psi(y))}{|x-y|^{N+sp}} \, dx \, dy \\
&\qquad= k^{-1} \int_{\Omega} \Big(\underline{g}(x)\, \tilde{u}_{k}^{m} + m\, d(x)^{-\gamma}\, \tilde{u}_{k}^{\delta}\Big) \, \Psi \, dx.
\end{align*}
Subtracting the above identities, we choose the test functions as follows.  
In the case $m=0$, we take
\[
\Phi=\big(\tilde{u}_{k+1}-\tilde{u}_k\big)^{+},
\qquad
\Psi=\big(\tilde{u}_k-\tilde{u}_{k+1}\big)^{-}.
\]
On the other hand, if $m>0$, since $\tilde{u}_k, \tilde{u}_{k+1}\in L^{\infty}(\Omega)$ and for a fixed $\varepsilon>0$, we select
\[
\Phi
= \frac{\big((\tilde{u}_{k+1}+\varepsilon)^{m+1}
	-(\tilde{u}_k+\varepsilon)^{m+1}\big)^{+}}{(\tilde{u}_{k+1}+\varepsilon)^m},
\qquad
\Psi
= \frac{\big((\tilde{u}_k+\varepsilon)^{m+1}
	-(\tilde{u}_{k+1}+\varepsilon)^{m+1}\big)^{-}}{(\tilde{u}_k+\varepsilon)^m},
\]
and arguing as in the proof of the comparison principle (Theorem~\ref{theorem2}), we infer that $ \tilde{u}_{k+1} \leq \tilde{u}_k  $ in $ \Omega. $ Consequently, by \cite[Theorem~1.1]{antonini2025global}, there exists an integer $k_0 \geq 1$ such that $ \|\tilde{u}_k\|_{C^{1}(\overline{\Omega})} \leq C(k_0) $ for all  $ k \geq k_0. $ It follows that the sequence $(\tilde{u}_k)_{k \geq k_0}$ is uniformly bounded and equicontinuous in $C^{1}(\overline{\Omega})$. Hence, by the Arzel\`a--Ascoli theorem, there exist a subsequence, still denoted by $(\tilde{u}_k)$, and  $\tilde{u} \in C^{1}(\overline{\Omega})$ such that $ \tilde{u}_k \to \tilde{u} $ in $ C^{1}(\overline{\Omega})  $ as $ k \to \infty. $ Moreover, \cite[Lemma~3.1]{antonini2025global} guarantees that $ \|\tilde{u}_k\|_{L^{\infty}(\Omega)} \to  0 $
as $ k \to \infty, $ which implies that $ \tilde{u}_k \to 0 $ in $ C^{1}(\overline{\Omega}) $ as $ k \to \infty. $ Therefore, in view of \eqref{4equ4}, for $k$ sufficiently large, one can define  $ \underline{u} := \tilde{u}_k \in C^{1}(\overline{\Omega}) \cap \mathcal{M}^{1}_{1,\alpha}(\Omega), $ $ \alpha \in [s,1), $ such that $\underline{u} \leq u_0$ in $\Omega$. Furthermore, $\underline{u}$ is a sub-solution of the problem \eqref{4equ1} in the case $n=1$, namely,
\begin{align*}
&\int_\Omega \underline{u}^{2m+1}\varphi\,dx
+ \Delta_t \int_\Omega |\nabla \underline{u}|^{p-2}\nabla \underline{u}\cdot \nabla \varphi\,dx
+ \Delta_t \iint_{\mathbb{R}^{2N}}
\frac{\left[ \underline{u}(x)-\underline{u}(y)\right] ^{p-1}(\varphi(x)-\varphi(y))}{|x-y|^{N+sp}}\,dx\,dy \\
&\qquad \leq \int_\Omega u_0^{m+1}\underline{u}^m \varphi\,dx
+ \Delta_t \int_\Omega g^1 \underline{u}^m \varphi\,dx
+ \Delta_t\, m\,\int_\Omega d(x)^{-\gamma}\underline{u}^{\delta}\varphi\,dx,
\end{align*}
for all $\varphi \in \mathbf{WL}$ with $\varphi \geq 0$. On the other hand, the function $u_1$ satisfies
\begin{align*}
&\int_\Omega u_1^{2m+1}\varphi\,dx
+ \Delta_t \int_\Omega |\nabla u_1|^{p-2}\nabla u_1\cdot \nabla \varphi\,dx
+ \Delta_t \iint_{\mathbb{R}^{2N}}
\frac{\left[ u_1(x)-u_1(y)\right] ^{p-1}(\varphi(x)-\varphi(y))}{|x-y|^{N+sp}}\,dx\,dy \\
&\qquad = \int_\Omega u_0^{m+1}u_1^m \varphi\,dx
+ \Delta_t \int_\Omega g^1 u_1^m \varphi\,dx
+ \Delta_t\, m\, \int_\Omega  d(x)^{-\gamma}u_1^{\delta}\varphi\,dx.
\end{align*}
Hence, the weak comparison principle (see Theorem~\ref{theorem2}) yields $ \underline{u} \leq u_1  $ in $ \Omega $ and then by induction 
\[
\underline{u} \leq u_n \quad \text{in } \Omega,
\qquad \text{for all } n = 0,1,\ldots,n_0.
\]
We now proceed to the construction of a uniform super-solution. To this end, we consider the auxiliary problem
\[
\left\{
\begin{aligned}
-\Delta_p \tilde{u} + (-\Delta)^s_p \tilde{u} &= \tilde{u}^{-\vartheta}
\quad && \text{in } \Omega,\\
\tilde{u} &> 0
\quad && \text{in } \Omega,\\
\tilde{u} &= 0
\quad && \text{in } \mathbb{R}^N \setminus \Omega.
\end{aligned}
\right.
\]
Proceeding along the lines of the proofs of
\cite[Theorems~2.27 and~2.28]{bal2024regularity}, we obtain the existence of a unique function
$\tilde{u} \in W^{1,p}_{\mathrm{loc}}(\Omega)$ and a constant $c>0$ such that
\[
c^{-1} d(x)^{\frac{p}{p-1+\vartheta}}
\leq \tilde{u}(x)
\leq c\, d(x)^{\frac{p}{p-1+\vartheta}}
\quad \text{for all } x \in \Omega.
\]
Moreover, using the assumption  $ \vartheta < 2 + \frac{1}{p-1}, $ which is admissible since $p(1-s)<1$, we get that $ \tilde{u} \in W^{1,p}_0(\Omega). $ In addition, the regularity result of
\cite[Theorem~2.31]{bal2024regularity} ensures that $ \tilde{u} \in C^{0,\frac{p}{p-1+\vartheta}}(\overline{\Omega}). $ Now, for a fixed $M>0$, we define $\tilde{u}_M := M^{\frac{1}{p-1}} \tilde{u}$. A direct computation shows that $\tilde{u}_M$ is the unique weak solution of the problem
\begin{equation}\label{4equ100}
\left\{
\begin{aligned}
-\Delta_p \tilde{u}_M + (-\Delta)^s_p \tilde{u}_M
&= M^{1+\frac{\vartheta}{p-1}}\, \tilde{u}_M^{-\vartheta}
&& \text{in } \Omega,\\
\tilde{u}_M &> 0
&& \text{in } \Omega,\\
\tilde{u}_M &= 0
&& \text{in } \mathbb{R}^N \setminus \Omega.
\end{aligned}
\right.
\end{equation}
and satisfies
\begin{equation}\label{4equ5}
c^{-1} M^{\frac{1}{p-1}} d(x)^{\frac{p}{p-1+\vartheta}}
\leq \tilde{u}_M(x)
\leq c\, M^{\frac{1}{p-1}} d(x)^{\frac{p}{p-1+\vartheta}},
\end{equation}
for some constant $c>0$ independent of $M$. Since $m<p-1$ and $\frac{p}{p-1+\vartheta}\in(0,s)$, it follows from \eqref{4equ4} and \eqref{4equ5}, by choosing $M$
sufficiently large, that $ \overline{u} := \tilde{u}_M \geq u_0 $ in $ \Omega. $ Also, $\overline{u}$ is a super-solution of the problem
\begin{equation}\label{eq:main-super-problem}
\left\{
\begin{aligned}
-\Delta_p \hat{u} + (-\Delta)^s_p \hat{u}
&= \|g\|_{L^\infty(Q_T)} \hat{u}^m + m\, d(x)^{-\gamma} \hat{u}^\delta
\quad && \text{in } \Omega,\\
\hat{u} &> 0
\quad && \text{in } \Omega,\\
\hat{u} &= 0
\quad && \text{in } \mathbb{R}^N \setminus \Omega.
\end{aligned}
\right.
\end{equation}
Consequently, $\overline{u}$ is a super-solution of problem~\eqref{4equ1} for $n=1$, namely,
\begin{align*}
&\int_\Omega \overline{u}^{2m+1}\varphi\,dx
+ \Delta_t \int_\Omega |\nabla \overline{u}|^{p-2}
\nabla \overline{u}\cdot \nabla \varphi\,dx
+ \Delta_t \iint_{\mathbb{R}^{2N}}
\frac{\left[\overline{u}(x)-\overline{u}(y)\right]^{p-1}
	\big(\varphi(x)-\varphi(y)\big)}{|x-y|^{N+sp}}\,dx\,dy \\
&\qquad \geq \int_\Omega u_0^{m+1}\overline{u}^m \varphi\,dx
+ \Delta_t \int_\Omega g^1 \overline{u}^m \varphi\,dx
+ \Delta_t \, m\, \int_\Omega  d(x)^{-\gamma}
\overline{u}^\delta \varphi\,dx,
\end{align*}
for all $\varphi \in \mathbf{WL}$ with $\varphi \geq 0$. By the weak comparison principle (see Theorem~\ref{theorem2}), we infer that $ \overline{u} \geq u_1 $ in $ \Omega. $ Iterating this argument yields
\[
\overline{u} \geq u_n \quad \text{in } \Omega,
\quad \text{for all } n=0,1,2,\ldots,n_0.
\]
This completes the proof of Claim \ref{claim1}. Now, with the sequence $(u_n)$, we define the piecewise functions $u_{\Delta_t} $ and $\tilde{u}_{\Delta_t} $ on the interval $[0, T]$ for $n \in \{1, 2, \dots, n_0\}$ as follows:
\begin{gather*}
u_{\Delta_t}(t) = u_n, \quad \text{for } t \in [t_{n-1}, t_n),\\
\tilde{u}_{\Delta_t}(t) = \frac{t - t_{n-1}}{\Delta_t} \left( u_n^{m+1} - u_{n-1}^{m+1} \right) + u_{n-1}^{m+1}, \quad \text{for } t \in [t_{n-1}, t_n).
\end{gather*}
It can then be readily verified that the following relation holds in the weak sense for all $t \in (0, T)$:
\begin{equation*}
u_{\Delta_t}^{m}\partial_{t}\tilde{u}_{\Delta_t}
- \Delta_{p}u_{\Delta_t}
+ (-\Delta)^{s}_{p}u_{\Delta_t}
= g^n u_{\Delta_t}^{m}
+ m\,d(x)^{-\gamma}u_{\Delta_t}^{\delta}.
\end{equation*}
Equivalently, for any test function $\varphi \in C_{c}^{\infty}(Q_T)$, we have
\begin{equation}\label{4equ7}
\begin{aligned}
&\int_{0}^{T}\!\!\int_{\Omega}
u_{\Delta_t}^{m}\partial_{t}\tilde{u}_{\Delta_t}\,\varphi(t,x)\,dx\,dt
+ \int_{0}^{T}\!\!\int_{\Omega}
|\nabla u_{\Delta}|^{p-2}\nabla u_{\Delta}\cdot\nabla\varphi(t,x)\,dx\,dt \\
&\quad
+ \int_{0}^{T}\!\!\iint_{\mathbb{R}^{2N}}
\frac{\big[u_{\Delta_t}(t,x)-u_{\Delta_t}(t,y)\big]^{p-1}
	\big(\varphi(t,x)-\varphi(t,y)\big)}
{|x-y|^{N+sp}}\,dx\,dy\,dt \\
&=
\int_{0}^{T}\!\!\int_{\Omega}
g^n u_{\Delta_t}^{m}\varphi(t,x)\,dx\,dt
+ m\int_{0}^{T}\!\!\int_{\Omega}
d(x)^{-\gamma}u_{\Delta_t}^{\delta}\varphi(t,x)\,dx\,dt.
\end{aligned}
\end{equation}
In order to pass to the limit as $\Delta_t \to 0$, we consider the following claim.

\begin{claim}\label{claim2}
	Let $m\geq 0$. The following uniform bounds and convergence properties hold:
	
	\noindent\textnormal{(i)}
	$\big(\partial_t \tilde{u}_{\Delta_t}\big)_{\Delta_t}$
	is uniformly bounded in $L^2(Q_T)$, independently of $\Delta_t$.
	
	\noindent\textnormal{(ii)}
	$\big(\tilde{u}_{\Delta_t}^{1/(m+1)}\big)_{\Delta_t}$ and
	$\big(u_{\Delta_t}\big)_{\Delta_t}$
	are uniformly bounded in
	$L^\infty(0,T;W^{1,p}_0(\Omega))$.
	
	\noindent\textnormal{(iii)}
	$\tilde{u}_{\Delta_t}^{1/m+1}
	\stackrel{\star}{\rightharpoonup} u$
	in $L^\infty(0,T;W^{1,p}_0(\Omega))$ and $ u_{\Delta_t}\stackrel{\star}{\rightharpoonup} u$
	in $L^\infty(0,T;W^{1,p}_0(\Omega))$.
	
	\noindent\textnormal{(iv)}
	$\tilde{u}_{\Delta_t}\to u^{m+1}$
	in $C([0,T];L^r(\Omega))$ and
	$u_{\Delta_t}\to u$
	in $L^\infty(0,T;L^r(\Omega))$,
	for all $r\ge 1$.
	
	\noindent\textnormal{(v)}
	$\partial_t \tilde{u}_{\Delta_t}
	\rightharpoonup \partial_t u^{m+1}$
	in $L^2(Q_T)$.
	
	\noindent\textnormal{(vi)}
	There exists $C>0$ such that, for every $t\in[0,T]$,
	\[
	C^{-1} d(x)\le u(t,x)\le C\, d(x)^{\alpha'}
	\quad \text{a.e. in } \Omega,
	\quad \text{ for some } \alpha' \in (0, s).
	\]
\end{claim}

\medskip
\noindent
Assuming Claim~\ref{claim2}, we show that the limit $u$ is a solution to problem~\eqref{PP1} in the sense of Definition~\ref{definition1}.
Indeed, let $\varphi\in C_c^\infty(Q_T)$. By Claim~\ref{claim2}-(ii), the quantity
\[
\frac{\left[u_{\Delta_t}(t,x)-u_{\Delta_t}(t,y)\right]^{p-1}}
{|x-y|^{\frac{N+sp}{p'}}},
\qquad p'=\frac{p}{p-1},
\]
is uniformly bounded in $L^\infty(0,T;L^{p'}(\mathbb{R}^{2N}))$.
Moreover, the convergence of $u_{\Delta_t}$ to $u$ yields, for a.e. $t\in[0,T]$,
\[
\frac{\left[u_{\Delta_t}(t,x)-u_{\Delta_t}(t,y)\right]^{p-1}}
{|x-y|^{\frac{N+sp}{p'}}}
\longrightarrow
\frac{\left[u(t,x)-u(t,y)\right]^{p-1}}
{|x-y|^{\frac{N+sp}{p'}}}
\quad\text{a.e. in } \mathbb{R}^{2N}.
\]
Consequently,
\[
\frac{\left[u_{\Delta_t}(t,x)-u_{\Delta_t}(t,y)\right]^{p-1}}
{|x-y|^{\frac{N+sp}{p'}}}
\rightharpoonup
\frac{\left[u(t,x)-u(t,y)\right]^{p-1}}
{|x-y|^{\frac{N+sp}{p'}}}
\quad
\text{in } L^{p'}((0,T)\times\mathbb{R}^{2N}),
\]
which implies
\begin{equation}\label{4equ8}
\begin{aligned}
\int_0^T &\!\!\iint_{\mathbb{R}^{2N}}
\frac{\left[u_{\Delta_t}(t, x)-u_{\Delta_t}(t, y)\right]^{p-1}
	(\varphi(t, x)-\varphi(t, y))}
{|x-y|^{N+sp}}\,dx\,dy\,dt\\
&\longrightarrow
\int_0^T \!\!\iint_{\mathbb{R}^{2N}}
\frac{\left[u(t, x)-u(t, y)\right]^{p-1}
	(\varphi(t, x)-\varphi(t, y))}
{|x-y|^{N+sp}}\,dx\,dy\,dt .
\end{aligned}
\end{equation}
We now distinguish two cases. \textbf{Case $m=0$.}
By virtue of Claim~\ref{claim2}-(v), we obtain
\begin{gather}\label{4equ12}
\int_{0}^{T}\!\!\int_{\Omega}
\partial_{t}\tilde{u}_{\Delta_t}\,\varphi(t,x)\,dx\,dt
\longrightarrow
\int_{0}^{T}\!\!\int_{\Omega}
\partial_{t}u\,\varphi(t,x)\,dx\,dt,
\quad \text{as } \Delta_t \to 0.
\end{gather}
Moreover, thanks to \eqref{4equ9}, it holds that
\begin{equation}\label{4equ13}
\int_{0}^{T}\!\!\int_{\Omega}
g^n \varphi(t, x) \,dx\,dt
\longrightarrow
\int_{0}^{T}\!\!\int_{\Omega}
g \varphi(t, x) \,dx\,dt,
\quad \text{as } \Delta_t \to 0.
\end{equation}

\medskip
\noindent
\textbf{Case $m>0$.}
From \eqref{equ23} and Claim~\ref{claim2}-(i), (iv), we obtain
\begin{equation}\label{4equ17}
\begin{aligned}
&\|u_{\Delta_t}^{m}-u^{m}\|_{L^{2}(Q_T)}
\le T^{\frac{1}{2}} \, \|u_{\Delta_t}^{m}-u^{m}\|_{L^{\infty}(0,T;L^{2}(\Omega))} \\
&\qquad\le T^{\frac{1}{2}} |\Omega|^{\frac{1}{2(m+1)}} 
\|u_{\Delta_t}^{m}-u^{m}\|_{L^{\infty}(0,T;L^{\frac{2(m+1)}{m}}(\Omega))} 
\quad (\to 0 \text{ if } m=1 \text{ by Claim} ~\ref{claim2}\text{-(iv)}) \\
&\qquad\le T^{\frac{1}{2}} |\Omega|^{\frac{1}{2(m+1)}} 
\|u_{\Delta_t}^{m+1}-u^{m+1}\|_{L^{\infty}(0,T;L^{2}(\Omega))}^{\frac{m}{m+1}} \\
&\qquad\le T^{\frac{1}{2}} |\Omega|^{\frac{1}{2(m+1)}} 
\Big[ \|u_{\Delta_t}^{m+1}- \tilde{u}_{\Delta_{t}}\|_{L^{\infty}(0,T;L^{2}(\Omega))} 
+ \|\tilde{u}_{\Delta_{t}} - u^{m+1}\|_{L^{\infty}(0,T;L^{2}(\Omega))} \Big]^{\frac{m}{m+1}} \\
&\qquad\le T^{\frac{1}{2}} |\Omega|^{\frac{1}{2(m+1)}} 
\Big[ C\,\Delta_t^{\frac{1}{2}} + \|\tilde{u}_{\Delta_{t}} - u^{m+1}\|_{L^{\infty}(0,T;L^{2}(\Omega))} \Big]^{\frac{m}{m+1}} 
\longrightarrow 0, \quad \text{as } \Delta_t \to 0.
\end{aligned}
\end{equation}
Using H\"older’s inequality together with Claim~\ref{claim2}-(i), (iv), (v),
\eqref{4equ17}, and \eqref{4equ9}, we infer that
\begin{align*}
&\Bigg|\int_{0}^{T}\!\!\int_{\Omega}
\Big(u_{\Delta_t}^{m}\partial_{t}\tilde{u}_{\Delta_t}
- u^{m}\partial_{t}u^{m+1}\Big)\varphi(t,x)\,dx\,dt\Bigg| \\
&\le
\|u_{\Delta_t}^{m}-u^{m}\|_{L^{2}(Q_T)}
\|\partial_{t}\tilde{u}_{\Delta_t}\varphi\|_{L^{2}(Q_T)}
+ \left|\int_{0}^{T}\!\!\int_{\Omega}
u^{m}\big(\partial_{t}\tilde{u}_{\Delta_t}-\partial_{t}u^{m+1}\big)
\varphi\,dx\,dt\right|  \longrightarrow 0,
\quad \text{as } \Delta_t\to 0.
\end{align*}
Similarly,
\begin{align*}
\int_{0}^{T}\!\!\int_{\Omega}(g^n u_{\Delta_t}^{m}-g u^{m})\varphi\,dx\,dt
&=
\int_{0}^{T}\!\!\int_{\Omega} g^n(u_{\Delta_t}^{m}-u^{m})\varphi\,dx\,dt
+ \int_{0}^{T}\!\!\int_{\Omega}(g^n-g)u^{m}\varphi\,dx\,dt \\[4pt]
&\le
\|g^n\varphi\|_{L^{2}(Q_T)}\|u_{\Delta_t}^{m}-u^{m}\|_{L^{2}(Q_T)}
+ \|u^{m}\varphi\|_{L^{2}(Q_T)}\|g^n-g\|_{L^{2}(Q_T)} \\[4pt]
&\longrightarrow 0,
\quad \text{as } \Delta_t\to 0.
\end{align*}
Consequently,
\begin{equation}\label{4equ14}
\int_{0}^{T}\!\!\int_{\Omega}
u_{\Delta_t}^{m}\partial_{t}\tilde{u}_{\Delta_t}\varphi(t, x)\,dx\,dt
\longrightarrow
\int_{0}^{T}\!\!\int_{\Omega}
u^{m}\partial_{t}u^{m+1}\varphi(t, x)\,dx\,dt,
\quad \text{as } \Delta_t\to 0,
\end{equation}
and
\begin{equation}\label{4equ15}
\int_{0}^{T}\!\!\int_{\Omega} g^n u_{\Delta_t}^{m}\varphi(t, x)\,dx\,dt
\longrightarrow
\int_{0}^{T}\!\!\int_{\Omega} g u^{m}\varphi(t, x)\,dx\,dt,
\quad \text{as } \Delta_t\to 0.
\end{equation}
Furthermore, in view of Claim~\ref{claim2}-(iv), we have
$ d(x)^{-\gamma} u_{\Delta_t}^{\delta} \varphi \to d(x)^{-\gamma} u^{\delta} \varphi $  a.e. in $ Q_T, $ up to a subsequence.   Since $2(\delta-\gamma) > 2m-1$, it follows from \eqref{4equ3} that   $ \left(d(x)^{-\gamma} u_{\Delta_t}^{\delta}\right) $ is uniformly bounded in $L^{2}(Q_T)$.   Therefore, by the Dominated Convergence Theorem, we deduce that
\begin{equation}\label{4equ10}
\int_{0}^{T}\!\!\int_{\Omega}
d(x)^{-\gamma}u_{\Delta_t}^{\delta}\varphi(t, x)\,dx\,dt
\longrightarrow
\int_{0}^{T}\!\!\int_{\Omega}
d(x)^{-\gamma}u^{\delta}\varphi(t, x)\,dx\,dt,
\quad \text{as } \Delta_t\to 0.
\end{equation}
To pass to the limit in the local term, we test the weak formulation
\eqref{4equ7} with $u_{\Delta_{t}}-u$.
Using again \eqref{4equ3} and Claim~\ref{claim2}-(iv), (v), we obtain
\begin{equation*}
\begin{aligned}
&\left|\int_{0}^{T}\!\!\int_{\Omega}
|\nabla u_{\Delta_{t}}|^{p-2}\nabla u_{\Delta_{t}}\cdot\nabla(u_{\Delta_{t}}-u)\,dx\,dt\right| \\
&\le
\left|\int_{0}^{T}\!\!\iint_{\mathbb{R}^{2N}}
\frac{[u_{\Delta_{t}}(t, x)-u_{\Delta_{t}}(t, y)]^{p-1}
	\big((u_{\Delta_{t}}-u)(t, x)-(u_{\Delta_{t}}-u)(t, y)\big)}
{|x-y|^{N+sp}}\,dx\,dy\,dt\right| \\
&\quad
+ \left|\int_{0}^{T}\!\!\int_{\Omega} g^n u_{\Delta_t}^{m}(u_{\Delta_{t}}-u)\,dx\,dt\right|
+ m\left|\int_{0}^{T}\!\!\int_{\Omega}
d(x)^{-\gamma}u_{\Delta_t}^{\delta}(u_{\Delta_{t}}-u)\,dx\,dt\right| \\
&\quad
+ \left|\int_{0}^{T}\!\!\int_{\Omega}
u_{\Delta_t}^{m}\partial_{t}\tilde{u}_{\Delta_t}(u_{\Delta_{t}}-u)\,dx\,dt\right|
\longrightarrow 0,
\quad \text{as } \Delta_t\to 0.
\end{aligned}
\end{equation*}
By \cite[Proposition~2.6]{constantin2025doubly}, we conclude that
\[
|\nabla u_{\Delta_{t}}|^{p-2}\nabla u_{\Delta_{t}}
\longrightarrow
|\nabla u|^{p-2}\nabla u
\quad \text{in } (L^{\frac{p}{p-1}}(Q_T))^{N},
\]
and
\begin{equation}\label{4equ11}
\int_{0}^{T}\!\!\int_{\Omega}
|\nabla u_{\Delta_{t}}|^{p-2}\nabla u_{\Delta_{t}}\cdot\nabla\varphi(t,x)\,dx\,dt 
\longrightarrow
\int_{0}^{T}\!\!\int_{\Omega}
|\nabla u|^{p-2}\nabla u\cdot\nabla\varphi(t,x)\,dx\,dt,
\quad \text{as } \Delta_t\to 0.
\end{equation}
Finally, collecting
\eqref{4equ8}, \eqref{4equ12}, \eqref{4equ13}, \eqref{4equ14},
\eqref{4equ15}, \eqref{4equ10}, and \eqref{4equ11},
and passing to the limit in \eqref{4equ7} as $\Delta_t\to 0^{+}$,
we deduce that $u$ satisfies \eqref{PP1}, namely,
\begin{equation*}
\begin{aligned}
&\int_{0}^{T}\!\!\int_{\Omega}
u^{m}\partial_{t}(u^{m+1})\varphi(t, x)\,dx\,dt
+ \int_{0}^{T}\!\!\int_{\Omega}
|\nabla u|^{p-2}\nabla u\cdot\nabla\varphi(t, x)\,dx\,dt \\
&\quad
+ \int_{0}^{T}\!\!\iint_{\mathbb{R}^{2N}}
\frac{[u(t,x)-u(t,y)]^{p-1}
	(\varphi(t,x)-\varphi(t,y))}
{|x-y|^{N+sp}}\,dx\,dy\,dt \\
&=
\int_{0}^{T}\!\!\int_{\Omega}
\big(g(t,x)u^{m}
+ m\,d(x)^{-\gamma}u^{\delta}\big)\varphi(t, x)\,dx\,dt,
\end{aligned}
\end{equation*}
for all $\varphi\in C_{c}^{\infty}(Q_T)$.
By density, the above identity remains valid for any
$\varphi\in L^{2}(Q_T)\cap L^{1}(0,T;W_{0}^{1,p}(\Omega))$.
We now prove the previous Claim~\ref{claim2}. By virtue of \eqref{4equ3}, there exists a constant $c>0$, independent of $\Delta_t$, such that for every $(t,x)\in Q_T$, for all $m\geq 0$ and some $\alpha'\in (0, s)$, the following estimates hold:
\begin{equation}\label{equ68}
\begin{aligned}
c^{-1} \, d(x) &\leq u_{\Delta_t}(t,x) \leq c\, d(x)^{\alpha'},\\
c^{-1} \, d(x) &\leq \tilde{u}_{\Delta_t}^{1/m+1}(t,x) \leq c\, d(x)^{\alpha'}.
\end{aligned}
\end{equation}
We distinguish two cases, beginning with the situation where $m=0$. In this case, we employ the test function $u_n - u_{n-1}$ in the weak formulation~\eqref{4equ1} satisfied by $u_n$. This choice yields the discrete energy identity:
\begin{align*}
&\Delta_t \int_\Omega \left(\frac{u_n-u_{n-1}}{\Delta_t}\right)^2 \,dx
+ \int_\Omega |\nabla u_n|^{p-2}\nabla u_n \cdot \nabla (u_n-u_{n-1}) \,dx \\
&\quad
+ \iint_{\mathbb{R}^{2N}}
\frac{\left[ u_n(x)-u_n(y)\right] ^{p-1}
	\big((u_n-u_{n-1})(x)-(u_n-u_{n-1})(y)\big)}
{|x-y|^{N+sp}} \,dx\,dy \\
&\qquad = \int_\Omega g_n \,(u_n-u_{n-1}) \,dx .
\end{align*}
By applying Young’s inequality and exploiting the convexity of the functional $ u \longmapsto \frac{1}{p}\Big(
\|u\|_{W^{1,p}_0(\Omega)}^{p}
+ \|u\|_{W^{s,p}_0(\Omega)}^{p}
\Big), $ we obtain the following discrete energy inequality:
\begin{align*}
&\frac{\Delta_t}{2}\int_\Omega \left(\frac{u_n-u_{n-1}}{\Delta_t}\right)^2 \,dx
+ \frac{1}{p}\Big(
\|u_n\|_{W^{1,p}_0(\Omega)}^{p}
- \|u_{n-1}\|_{W^{1,p}_0(\Omega)}^{p}
\Big) \\
&\quad
+ \frac{1}{p}\Big(
\|u_n\|_{W^{s,p}_0(\Omega)}^{p}
- \|u_{n-1}\|_{W^{s,p}_0(\Omega)}^{p}
\Big)
\leq \frac{\Delta_t}{2}\|g_n\|_{L^2(\Omega)}^{2}.
\end{align*}
Summing this inequality from $n=1$ up to $n=n_0'<n_0$, we infer that
\begin{align*}
&\frac{\Delta_t}{2} \sum_{n=1}^{n_0'}
\int_\Omega \left(\frac{u_n-u_{n-1}}{\Delta_t}\right)^2 \,dx \\
&\quad
+ \frac{1}{p}\Big(
\|u_{n_0'}\|_{W^{1,p}_0(\Omega)}^{p}
- \|u_0\|_{W^{1,p}_0(\Omega)}^{p}
\Big)
+ \frac{1}{p}\Big(
\|u_{n_0'}\|_{W^{s,p}_0(\Omega)}^{p}
- \|u_0\|_{W^{s,p}_0(\Omega)}^{p}
\Big) \\
&\qquad
\leq \frac{\Delta_t}{2} \sum_{n=1}^{n_0'} \|g_n\|_{L^2(\Omega)}^{2}
\leq \frac{1}{2}\|g\|_{L^2(Q_T)}^{2}.
\end{align*}
Since $n_0'$ is arbitrary, we deduce
\begin{align}\label{4equ18}
\left\| \partial_t \tilde{u}_{\Delta_t} \right\|_{L^{2}(Q_T)}^{2}
+ \frac{1}{p}\|u_{\Delta_t}\|_{L^\infty(0,T;W^{1,p}_0(\Omega))}^{p}
\leq \frac{1}{2}\|g\|_{L^2(Q_T)}^{2}
+ \frac{2}{p}\|u_0\|_{W^{1,p}_0(\Omega)}^{p}.
\end{align}
As a consequence, the family $\big(\partial_t \tilde{u}_{\Delta_t}\big)_{\Delta_t}$
is uniformly bounded in $L^2(Q_T)$ independently of $\Delta_t$, while
$\big(u_{\Delta_t}\big)_{\Delta_t}$ is uniformly bounded in
$L^\infty\! (0,T;W^{1,p}_0(\Omega))$.
Moreover, by convexity, the family
$(\tilde{u}_{\Delta_t})_{\Delta_t}$ is also uniformly bounded in
$L^\infty\! (0,T;W^{1,p}_0(\Omega))$.
This proves Claim~\ref{claim2}-(i) and~(ii). Therefore, up to the extraction of a sub- sequence, there exist
$u,v \in L^\infty\! (0,T;W^{1,p}_0(\Omega))$ such that
\[
\tilde{u}_{\Delta_t} \stackrel{\ast}{\rightharpoonup} u
\quad \text{in } L^\infty\!\left(0,T;W^{1,p}_0(\Omega)\right),
\]
and
\[
u_{\Delta_t} \stackrel{\ast}{\rightharpoonup} v
\quad \text{in } L^\infty\!\left(0,T;W^{1,p}_0(\Omega)\right).
\]
Furthermore, the uniform boundedness of
$\big(\partial_t \tilde{u}_{\Delta_t}\big)_{\Delta_t}$ in $L^2(Q_T)$ implies
\[
\partial_t \tilde{u}_{\Delta_t} \rightharpoonup \partial_t u
\quad \text{in } L^2(Q_T).
\]
As a direct consequence, we obtain
\begin{equation}\label{4equ19}
\|u_{\Delta_t} - \tilde{u}_{\Delta_t}\|_{L^\infty(0,T;L^2(\Omega))}
\leq
\max_{1 \leq n \leq n_0}
\|u_n - u_{n-1}\|_{L^2(\Omega)}
\leq
\Delta_t^{\frac{1}{2}}
\|\partial_t \tilde{u}_{\Delta_t}\|_{L^2(Q_T)}
\leq
C\,\Delta_t^{\frac{1}{2}}.
\end{equation}
It follows that $u = v$. Together with \eqref{equ68}, this completes the proof of Claim~\ref{claim2} parts~(ii), (v), and~(vi). Finally, in view of~\eqref{4equ18} and the compact embedding
$W^{1,p}_0(\Omega)\hookrightarrow L^2(\Omega)$, the Aubin--Simon Lemma
(see e.g. \cite[Theorem~II.5.16]{boyer2012mathematical}) yields
\[
\tilde{u}_{\Delta_t} \to u
\quad \text{in } C\!\left([0,T];L^2(\Omega)\right).
\]
By interpolation, we further deduce, up to a subsequence, that
\[
\tilde{u}_{\Delta_t} \to u
\quad \text{in } C\!\left([0,T];L^r(\Omega)\right),
\qquad \text{for all } r \geq 1.
\]
Combining this convergence with~\eqref{4equ19}, we conclude that
$u_{\Delta_t} \to u$ in $L^\infty\!\left(0,T;L^2(\Omega)\right)$, and,
by interpolation,
\[
u_{\Delta_t} \to u
\quad \text{in } L^\infty\!\left(0,T;L^r(\Omega)\right),
\qquad \text{for all } r \geq 1.
\]
This concludes the proof of Claim~\ref{claim2} in the case $m = 0$.  
Next, we turn our attention to the case $m > 0$.   Since $u_n,\, u_{n-1} \in L^{\infty}(\Omega)$, for any fixed $\varepsilon > 0 $, we may select the test function
 \[
 \frac{(u_n+\varepsilon)^{m+1}-(u_{n-1}+\varepsilon)^{m+1}}{(u_n+\varepsilon)^m}
 \in W_0^{1,p}(\Omega)\cap L^{\infty}(\Omega),
 \]
 in the weak formulation~\eqref{4equ1} satisfied by $u_n$.
 As a consequence, we obtain
 \begin{equation}\label{4equ16}
 \begin{aligned}
 &\Delta_t \int_\Omega
 \left(\frac{u_n^{m+1}-u_{n-1}^{m+1}}{\Delta_t}\right)
 u_n^{m}
 \left(
 \frac{(u_n+\varepsilon)^{m+1}-(u_{n-1}+\varepsilon)^{m+1}}{(u_n+\varepsilon)^m}
 \right)\,dx \\
 &\quad
 + \int_\Omega |\nabla u_n|^{p-2}\nabla u_n \cdot
 \nabla\!\left(
 \frac{(u_n+\varepsilon)^{m+1}-(u_{n-1}+\varepsilon)^{m+1}}{(u_n+\varepsilon)^m}
 \right)\,dx \\
 &\quad
 + \iint_{\mathbb{R}^{2N}}
 \frac{\left[ u_n(x)-u_n(y)\right] ^{p-1}}{|x-y|^{N+sp}} \\
 &\quad\times
 \Biggl[\left( 
 \frac{(u_n+\varepsilon)^{m+1}-(u_{n-1}+\varepsilon)^{m+1}}{(u_n+\varepsilon)^m}\right) (x)
 -
 \left( \frac{(u_n+\varepsilon)^{m+1}-(u_{n-1}+\varepsilon)^{m+1}}{(u_n+\varepsilon)^m}\right) (y)
 \Biggr]
 \,dx\,dy \\
 &=
 \int_\Omega g_n u_n^{m}
 \left(
 \frac{(u_n+\varepsilon)^{m+1}-(u_{n-1}+\varepsilon)^{m+1}}{(u_n+\varepsilon)^m}
 \right)\,dx \\
 &\quad
 + m \int_\Omega d(x)^{-\gamma} u_n^{\delta}
 \left(
 \frac{(u_n+\varepsilon)^{m+1}-(u_{n-1}+\varepsilon)^{m+1}}{(u_n+\varepsilon)^m}
 \right)\,dx .
 \end{aligned}
 \end{equation}
Since $m < p - 1$, it follows from the \textbf{Picone-type inequalities} in both the local and nonlocal frameworks \cite[Propositions~2.6 and~4.2]{brasco2014convexity}, together with Young's inequality, that
 \begin{equation*}
 \begin{aligned}
 &|u_n(x)-u_n(y)|^{p-2}\big(u_n(x)-u_n(y)\big)
 \Biggl[
 \frac{(u_{n-1}+\varepsilon)^{m+1}}{(u_n+\varepsilon)^m}(x)
 -
 \frac{(u_{n-1}+\varepsilon)^{m+1}}{(u_n+\varepsilon)^m}(y)
 \Biggr] \\
 &\leq |u_{n-1}(x)-u_{n-1}(y)|^{m+1}
 |u_n(x)-u_n(y)|^{p-m-1} \\
 &\leq \frac{m+1}{p}|u_{n-1}(x)-u_{n-1}(y)|^{p}
 + \frac{p-m-1}{p}|u_n(x)-u_n(y)|^{p},
 \end{aligned}
 \end{equation*}
 and
 \begin{equation*}
 \begin{aligned}
 |\nabla u_n|^{p-2}\nabla u_n \cdot
 \nabla\!\left(\frac{(u_{n-1}+\varepsilon)^{m+1}}{(u_n+\varepsilon)^m}\right)
 \leq
 \frac{m+1}{p}|\nabla u_{n-1}|^{p}
 + \frac{p-m-1}{p}|\nabla u_n|^{p}.
 \end{aligned}
 \end{equation*}
 Combining these estimates with~\eqref{4equ16}, we infer
 \begin{align*}
 &\Delta_t \int_\Omega
 \left(\frac{u_n^{m+1}-u_{n-1}^{m+1}}{\Delta_t}\right)
 u_n^{m}
 \left(
 \frac{(u_n+\varepsilon)^{m+1}-(u_{n-1}+\varepsilon)^{m+1}}{(u_n+\varepsilon)^m}
 \right)\,dx \\
 &\quad
 + \frac{m+1}{p}\Big(
 \|u_n\|_{W^{1,p}_0(\Omega)}^{p}
 - \|u_{n-1}\|_{W^{1,p}_0(\Omega)}^{p}
 \Big)
 + \frac{m+1}{p}\Big(
 \|u_n\|_{W^{s,p}_0(\Omega)}^{p}
 - \|u_{n-1}\|_{W^{s,p}_0(\Omega)}^{p}
 \Big) \\
 &\qquad
 \leq \int_\Omega g_n u_n^{m}
 \left(
 \frac{(u_n+\varepsilon)^{m+1}-(u_{n-1}+\varepsilon)^{m+1}}{(u_n+\varepsilon)^m}
 \right)\,dx \\
 &\quad
 + m \int_\Omega d(x)^{-\gamma} u_n^{\delta}
 \left(
 \frac{(u_n+\varepsilon)^{m+1}-(u_{n-1}+\varepsilon)^{m+1}}{(u_n+\varepsilon)^m}
 \right)\,dx .
 \end{align*}
 Letting $\varepsilon\to0$ and applying Fatou’s lemma together with the dominated convergence theorem, we obtain
 \begin{equation}\label{4equ21}
  \begin{aligned}
 &\frac{\Delta_t}{2}\int_\Omega
 \left(\frac{u_n^{m+1}-u_{n-1}^{m+1}}{\Delta_t}\right)^2\,dx \\
 &\quad
 + \frac{m+1}{p}\Big(
 \|u_n\|_{W^{1,p}_0(\Omega)}^{p}
 - \|u_{n-1}\|_{W^{1,p}_0(\Omega)}^{p}
 \Big)
 + \frac{m+1}{p}\Big(
 \|u_n\|_{W^{s,p}_0(\Omega)}^{p}
 - \|u_{n-1}\|_{W^{s,p}_0(\Omega)}^{p}
 \Big) \\
 &\qquad
 \leq
 \Delta_t\|g_n\|_{L^2(\Omega)}^{2}
 + m^2\Delta_t\int_\Omega d(x)^{-2\gamma}u_n^{2(\delta-m)}\,dx .
  \end{aligned}
 \end{equation}
 Using~\eqref{4equ3} and the assumption $2(\delta-\gamma)>2m-1$, we have
 \[
 \int_\Omega d(x)^{-2\gamma}u_n^{2(\delta-m)}\,dx
 \leq \int_\Omega d(x)^{-2\gamma+2(\delta-m)}\,dx < \infty .
 \]
 Summing the inequality \eqref{4equ21} from $n=1$ up to $n=n_0'<n_0$, we deduce that
 \begin{align*}
 &\frac{\Delta_t}{2}\sum_{n=1}^{n_0'}\int_\Omega
 \left(\frac{u_n^{m+1}-u_{n-1}^{m+1}}{\Delta_t}\right)^2\,dx \\
 &\quad
 + \frac{m+1}{p}\Big(
 \|u_{n_0'}\|_{W^{1,p}_0(\Omega)}^{p}
 - \|u_0\|_{W^{1,p}_0(\Omega)}^{p}
 \Big)
 + \frac{m+1}{p}\Big(
 \|u_{n_0'}\|_{W^{s,p}_0(\Omega)}^{p}
 - \|u_0\|_{W^{s,p}_0(\Omega)}^{p}
 \Big) \\
 &\qquad
 \leq \|g\|_{L^2(Q_T)}^{2} + C,
 \end{align*}
 where $C$ is independent of $n_0'$.
 Since $n_0'$ is arbitrary, we infer that
 \begin{align}\label{4equ20}
 \left\|\partial_t\tilde{u}_{\Delta_t}\right\|_{L^2(Q_T)}^{2}
 + \frac{m+1}{p}\|u_{\Delta_t}\|_{L^\infty(0,T;W^{1,p}_0(\Omega))}^{p}
 \leq
 \|g\|_{L^2(Q_T)}^{2}
 + \frac{2(m+1)}{p}\|u_0\|_{W^{1,p}_0(\Omega)}^{p}.
 \end{align}
 Once again, the sequence $\big(\partial_t \tilde{u}_{\Delta_t}\big)_{\Delta_t}$ is uniformly bounded in $L^2(Q_T)$, independently of $\Delta_t$, whereas $\big(u_{\Delta_t}\big)_{\Delta_t}$ is uniformly bounded in $L^\infty(0,T; W^{1,p}_0(\Omega))$.  
 Moreover, since $ \tilde{u}_{\Delta_t} = \zeta\, u_n^{m+1} + (1-\zeta)\, u_{n-1}^{m+1}, $  with $ \zeta := \frac{t-t_{n-1}}{\Delta_t}, $
 an application of \cite[Proposition~2.6]{brasco2014convexity} yields
 \[
 \int_{\Omega}
 \left|\nabla\!\left(\tilde{u}_{\Delta_t}^{1/m+1}\right)\right|^p
 \,dx
 \leq
 \zeta \int_{\Omega} |\nabla u_n|^p \,dx
 +
 (1-\zeta)\int_{\Omega} |\nabla u_{n-1}|^p \,dx .
 \]
 As a consequence, 
 $\big(\tilde{u}_{\Delta_t}^{1/m+1}\big)_{\Delta_t}$
 is uniformly bounded in
 $L^\infty\! (0,T;W^{1,p}_0(\Omega))$.
 This proves Claim~\ref{claim2}-(i) and~(ii). Therefore, up to the extraction of a subsequence, there exist
 $u,v \in L^\infty\! (0,T;W^{1,p}_0(\Omega))$ such that
 \[
 u_{\Delta_t} \stackrel{\ast}{\rightharpoonup} u
 \quad \text{in }
 L^\infty\!\left(0,T;W^{1,p}_0(\Omega)\right),
 \]
 and
 \[
 \tilde{u}_{\Delta_t}^{1/m+1}
 \stackrel{\ast}{\rightharpoonup} v
 \quad \text{in }
 L^\infty\!\left(0,T;W^{1,p}_0(\Omega)\right).
 \]
Moreover, by invoking \eqref{equ23} together with the uniform boundedness of
 $\big(\partial_t \tilde{u}_{\Delta_t}\big)_{\Delta_t}$ in $L^2(Q_T)$, we infer that
 \begin{equation}\label{4equ23}
 \begin{aligned}
 \sup_{t\in[0,T]}
 \left\|
 \tilde{u}_{\Delta_t}^{1/m+1} - u_{\Delta_t}
 \right\|_{L^{2(m+1)}(\Omega)}^{2(m+1)}
 &\leq
 \sup_{t\in[0,T]}
 \left\|
 \tilde{u}_{\Delta_t} - u_{\Delta_t}^{m+1}
 \right\|_{L^2(\Omega)}^{2} \\
 &\leq
 \max_{1\leq n \leq n_0}
 \|u_n - u_{n-1}\|_{L^2(\Omega)}^{2} \\
 &\leq
 \Delta_t
 \|\partial_t \tilde{u}_{\Delta_t}\|_{L^2(Q_T)}^{2}
 \leq
 C\,\Delta_t^{\frac{1}{2}} .
 \end{aligned}
 \end{equation}
It follows that $u = v$. Together with \eqref{equ68}, this completes the proof of Claim~\ref{claim2} parts~(ii), (v), and~(vi). By the mean value theorem and the uniform boundedness of the family $\big(\partial_t \tilde{u}_{\Delta_t}\big)_{\Delta_t}$ in $L^2(Q_T)$,
it follows that the sequence $\big(\tilde{u}_{\Delta_t}\big)_{\Delta_t}$
is equicontinuous in $C\!\left([0,T];L^r(\Omega)\right)$ for $1<r\leq 2$.
Hence, in view of~\eqref{equ68} and by interpolation, we deduce that
$\big(\tilde{u}_{\Delta_t}\big)_{\Delta_t}$ is equicontinuous in
$C\!\left([0,T];L^r(\Omega)\right)$ for all $r>1$.
Moreover, by~\eqref{equ23}, the sequence
$\big(\tilde{u}_{\Delta_t}^{1/m+1}\big)_{\Delta_t}$
is uniformly bounded and equicontinuous in
$C\!\left([0,T];L^r(\Omega)\right)$ for $r>1$.
Therefore, by the Arzel\`a--Ascoli theorem, there exists a subsequence
(still denoted by $(\tilde{u}_{\Delta_t})$) such that
\begin{equation}\label{4equ22}
\tilde{u}_{\Delta_t} \to u^{m+1}
\quad \text{in } C\!\left([0,T];L^r(\Omega)\right),
\qquad \text{for all } r \geq 1.
\end{equation}
Combining this convergence with~\eqref{4equ23}, we conclude that
$u_{\Delta_t}\to u$ in $L^\infty\!\left(0,T;L^2(\Omega)\right)$ and,
by interpolation,
\[
u_{\Delta_t}\to u
\quad \text{in } L^\infty\!\left(0,T;L^r(\Omega)\right),
\qquad \text{for all } r \geq 1.
\]
Furthermore, since the family
$\big(\partial_t \tilde{u}_{\Delta_t}\big)_{\Delta_t}$ is uniformly
bounded in $L^2(Q_T)$, independently of $\Delta_t$, it follows from
\eqref{4equ22} that $ \partial_t \tilde{u}_{\Delta_t}
\rightharpoonup \partial_t u^{m+1} $
in $ L^2(Q_T). $ This completes the proof of Claim~\ref{claim2} in the case $m>0$.\\[4pt]
\textbf{Step~2: Uniqueness of weak solutions.}  
Let $u$ and $\tilde u$ be two weak solutions of~\eqref{PP1} corresponding to the data $g_{1}$ and $g_{2}$, respectively.   We begin with the case $m = 0$.   By choosing the test functions $\Phi = u - \tilde u$ and $\Psi = \tilde u - u$ in~\eqref{weakform}, which are satisfied by $u$ and $\tilde u$, respectively, we deduce that for every $t \in (0,T]$,
\begin{equation*}
\begin{aligned}
& \int_{0}^{t}\!\!\int_{\Omega} \partial_{t} u \, (u-\tilde u)\, dx\, ds
+ \int_{0}^{t}\!\!\int_{\Omega} |\nabla u|^{p-2} \nabla u \cdot \nabla (u-\tilde u)\, dx\, ds \\
&\quad + \int_{0}^{t}\!\!\iint_{\mathbb{R}^{2N}}
\frac{\left[ u(s,x)-u(s,y)\right] ^{p-1}
	\big((u-\tilde u)(s,x)-(u-\tilde u)(s,y)\big)}
{|x-y|^{N+sp}} \, dx\, dy\, ds \\
&= \int_{0}^{t}\!\!\int_{\Omega} g_{1}(s,x)\,(u-\tilde u)\, dx\, ds,
\end{aligned}
\end{equation*}
and

\begin{equation*}
\begin{aligned}
& \int_{0}^{t}\!\!\int_{\Omega} \partial_{t} \tilde u \, (\tilde u-u)\, dx\, ds
+ \int_{0}^{t}\!\!\int_{\Omega} |\nabla \tilde u|^{p-2} \nabla \tilde u \cdot \nabla (\tilde u-u)\, dx\, ds \\
&\quad + \int_{0}^{t}\!\!\iint_{\mathbb{R}^{2N}}
\frac{\left[ \tilde u(s,x)-\tilde u(s,y)\right] ^{p-1}
	\big((\tilde u-u)(s,x)-(\tilde u-u)(s,y)\big)}
{|x-y|^{N+sp}} \, dx\, dy\, ds \\
&= \int_{0}^{t}\!\!\int_{\Omega} g_{2}(s,x)\,(\tilde u-u)\, dx\, ds .
\end{aligned}
\end{equation*}
Summing the two identities yields
\begin{equation*}
\begin{aligned}
& \int_{0}^{t}\!\!\int_{\Omega} \partial_{t}(u-\tilde u)\,(u-\tilde u)\, dx\, ds \\
&\quad + \int_{0}^{t}\!\!\int_{\Omega}
\Big(|\nabla u|^{p-2}\nabla u-|\nabla \tilde u|^{p-2}\nabla \tilde u\Big)
\cdot \nabla (u-\tilde u)\, dx\, ds \\
&\quad + \int_{0}^{t}\!\!\iint_{\mathbb{R}^{2N}}
\frac{
	\Big( [u(x)-u(y)]^{p-1}- [\tilde u(x)-\tilde u(y)]^{p-1} \Big)
	\Big( (u-\tilde u)(x)-(u-\tilde u)(y) \Big)
}
{|x-y|^{N+sp}}\, dx\, dy\, ds \\
&= \int_{0}^{t}\!\!\int_{\Omega} (g_{1}-g_{2})(u-\tilde u)\, dx\, ds .
\end{aligned}
\end{equation*}
Invoking the elementary monotonicity inequalities from~\cite{farina2013monotonicity} (see also~\cite[Section~10]{lindqvist2017notes}), we deduce
\begin{equation*}
\frac{1}{2}\int_{0}^{t}\!\!\int_{\Omega} \partial_{t}(u-\tilde u)^{2}\, dx\, ds
\le \int_{0}^{t}\!\!\int_{\Omega} (g_{1}-g_{2})(u-\tilde u)\, dx\, ds .
\end{equation*}
Using H\"older's inequality together with Gr\"onwall's lemma~\cite[Lemma~A.5]{brezis1973ope}, we conclude that for all $t\in[0,T]$,
\[
\|u(t)-\tilde u(t)\|_{L^{2}(\Omega)}
\le \|u(0)-\tilde u(0)\|_{L^{2}(\Omega)}
+ \int_{0}^{t} \|g_{1}(s)-g_{2}(s)\|_{L^{2}(\Omega)}\, ds .
\]
This estimate proves the uniqueness of weak solutions in the case $m=0$. We now turn to the case $m>0$. Since $u,\tilde u\in L^{\infty}(Q_{T})$, for any $\varepsilon\in(0,1)$ we define
\begin{equation}\label{equ35}
\Phi := \frac{(u+\epsilon)^{m+1} - (\tilde{u}+\epsilon)^{m+1}}{(u+\epsilon)^{m}}, \quad
\Psi := \frac{(\tilde{u}+\epsilon)^{m+1} - (u+\epsilon)^{m+1}}{(\tilde{u}+\epsilon)^{m}}.
\end{equation}
Clearly, $\Phi, \Psi \in L^2(Q_T) \cap L^1(0,T; W^{1,p}_0(\Omega))$, and for any $t \in (0,T]$ we have
\begin{align*}
& \int_0^t \int_\Omega u^{m} \, \partial_t(u^{m+1}) \, \Phi \, dx \, ds\\
&\quad   + \int_0^t \int_\Omega |\nabla u|^{p-2} \nabla u \cdot \nabla \Phi \, dx \, ds \\
&\quad + \int_0^t \iint_{\mathbb{R}^{2N}} 
\frac{[u(s,x)-u(s,y)]^{p-1} (\Phi(s,x)-\Phi(s,y))}{|x-y|^{N+sp}} \, dx \, dy \, ds \\
&= \int_0^t \int_\Omega \Big(g_{1}(s,x) u^{m} + m\, d(x)^{-\gamma} u^{\delta}\Big) \Phi \, dx \, ds,
\end{align*}
and
\begin{align*}
& \int_0^t \int_\Omega \tilde{u}^{m} \, \partial_t(\tilde{u}^{m+1}) \, \Psi \, dx \, ds \\
&\quad  + \int_0^t \int_\Omega |\nabla \tilde{u}|^{p-2} \nabla \tilde{u} \cdot \nabla \Psi \, dx \, ds \\
&\quad + \int_0^t \iint_{\mathbb{R}^{2N}} 
\frac{[\tilde{u}(s,x)-\tilde{u}(s,y)]^{p-1} (\Psi(s,x)-\Psi(s,y))}{|x-y|^{N+sp}} \, dx \, dy \, ds \\
&= \int_0^t \int_\Omega \Big(g_{2}(s,x) \tilde{u}^{m} + m\, d(x)^{-\gamma} \tilde{u}^{\delta}\Big) \Psi \, dx \, ds.
\end{align*}
Summing these identities, we obtain
\begin{align*}
& \int_0^t \int_\Omega 
\Bigg( \frac{u^{m}\, \partial_t(u^{m+1})}{(u+\epsilon)^{m}} 
- \frac{\tilde{u}^{m} \, \partial_t(\tilde{u}^{m+1})}{(\tilde{u}+\epsilon)^{m}} \Bigg) 
\big((u+\epsilon)^{m+1} - (\tilde{u}+\epsilon)^{m+1}\big) \, dx \, ds \\
&\quad + \int_0^t \int_\Omega |\nabla u|^{p-2} \nabla u \cdot \nabla\left(  \frac{(u+\epsilon)^{m+1} - (\tilde{u}+\epsilon)^{m+1}}{(u+\epsilon)^{m}}\right)  \, dx \, ds \\
&\quad + \int_0^t \int_\Omega |\nabla \tilde{u}|^{p-2} \nabla \tilde{u} \cdot \nabla \left( \frac{(\tilde{u}+\epsilon)^{m+1} - (u+\epsilon)^{m+1}}{(\tilde{u}+\epsilon)^{m}}\right)  \, dx \, ds \\
&\quad + \int_0^t \iint_{\mathbb{R}^{2N}} \frac{[u(s,x)-u(s,y)]^{p-1}}{|x-y|^{N+sp}} \Bigg[ \left( \frac{(u+\epsilon)^{m+1} - (\tilde{u}+\epsilon)^{m+1}}{(u+\epsilon)^{m}}\right) (s, x) \\
&\qquad - \left( \frac{(u+\epsilon)^{m+1} - (\tilde{u}+\epsilon)^{m+1}}{(u+\epsilon)^{m}}\right) (s, y) \Bigg] \, dx \, dy \, ds \\
&\quad + \int_0^t \iint_{\mathbb{R}^{2N}} \frac{[\tilde{u}(s,x)-\tilde{u}(s,y)]^{p-1}}{|x-y|^{N+sp}} \Bigg[ \left( \frac{(\tilde{u}+\epsilon)^{m+1} - (u+\epsilon)^{m+1}}{(\tilde{u}+\epsilon)^{m}}\right) (s, x) \\
&\qquad - \left( \frac{(\tilde{u}+\epsilon)^{m+1} - (u+\epsilon)^{m+1}}{(\tilde{u}+\epsilon)^{m}}\right) (s, y) \Bigg] \, dx \, dy \, ds \\
&= \int_0^t \int_\Omega \Bigg( \frac{g_{1} u^{m}}{(u+\epsilon)^{m}} - \frac{g_{2} \tilde{u}^{m}}{(\tilde{u}+\epsilon)^{m}} \Bigg)
\big((u+\epsilon)^{m+1} - (\tilde{u}+\epsilon)^{m+1}\big) \, dx \, ds \\
&\quad + m \int_0^t \int_\Omega d(x)^{-\gamma} \Bigg( \frac{u^{\delta}}{(u+\epsilon)^{m}} - \frac{\tilde{u}^{\delta}}{(\tilde{u}+\epsilon)^{m}} \Bigg) 
\big((u+\epsilon)^{m+1} - (\tilde{u}+\epsilon)^{m+1}\big) \, dx \, ds.
\end{align*}
By Lemma~\ref{Lem2}, it follows that
\begin{equation}\label{4equ29}
 \begin{aligned}
& \int_0^t \int_\Omega 
\Bigg( \frac{u^{m} \, \partial_t(u^{m+1})}{(u+\epsilon)^{m}} 
- \frac{ \tilde{u}^{m} \, \partial_t(\tilde{u}^{m+1})}{(\tilde{u}+\epsilon)^{m}} \Bigg) 
\big((u+\epsilon)^{m+1} - (\tilde{u}+\epsilon)^{m+1}\big) \, dx \, ds  \\
&\le \int_0^t \int_\Omega \Bigg( \frac{g_{1} u^{m}}{(u+\epsilon)^{m}} - \frac{g_{2} \tilde{u}^{m}}{(\tilde{u}+\epsilon)^{m}} \Bigg)
\big((u+\epsilon)^{m+1} - (\tilde{u}+\epsilon)^{m+1}\big) \, dx \, ds \\
&\quad + m \int_0^t \int_\Omega d(x)^{-\gamma} \Bigg( \frac{u^{\delta}}{(u+\epsilon)^{m}} - \frac{\tilde{u}^{\delta}}{(\tilde{u}+\epsilon)^{m}} \Bigg)
\big((u+\epsilon)^{m+1} - (\tilde{u}+\epsilon)^{m+1}\big) \, dx \, ds.
\end{aligned}
\end{equation}
Since $\frac{u}{u+\epsilon}, \frac{\tilde{u}}{\tilde{u}+\epsilon} < 1$ and $u, \tilde{u} \in L^{\infty}(Q_T)$, we have
\begin{align*}
& \Bigg| 
\Bigg( \frac{u^{m} \, \partial_t(u^{m+1})}{(u+\epsilon)^{m}} 
- \frac{\tilde{u}^{m} \, \partial_t(\tilde{u}^{m+1})}{(\tilde{u}+\epsilon)^{m}} \Bigg) 
((u+\epsilon)^{m+1} - (\tilde{u}+\epsilon)^{m+1}) 
\Bigg| \\[2mm]
& \quad \le C \big(|\partial_t(u^{m+1})| + |\partial_t(\tilde{u}^{m+1})|\big) \in L^1(Q_T),
\end{align*}
where $C>0$ is independent of $\epsilon$. Moreover, as $\epsilon \to 0^+$,
\[
\Bigg( \frac{u^{m}\, \partial_t(u^{m+1}) }{(u+\epsilon)^{m}} - \frac{\tilde{u}^{m} \, \partial_t(\tilde{u}^{m+1})}{(\tilde{u}+\epsilon)^{m}} \Bigg)
((u+\epsilon)^{m+1} - (\tilde{u}+\epsilon)^{m+1}) \to \frac{1}{2} \partial_t (u^{m+1} - \tilde{u}^{m+1})^2.
\]
Therefore, by the Dominated Convergence Theorem, we obtain
\begin{equation*}
\begin{aligned}
& \int_0^t \int_\Omega 
\Bigg( \frac{u^{m} \, \partial_t(u^{m+1})}{(u+\epsilon)^{m}} 
- \frac{\tilde{u}^{m} \, \partial_t(\tilde{u}^{m+1})}{(\tilde{u}+\epsilon)^{m}} \Bigg) \notag  ((u+\epsilon)^{m+1} - (\tilde{u}+\epsilon)^{m+1}) \, dx \, ds \notag \\
& \quad \longrightarrow \frac{1}{2} \int_0^t \int_\Omega \partial_t(u^{m+1}-\tilde{u}^{m+1})^2 \, dx \, ds.
\end{aligned}
\end{equation*}
Using Fatou's Lemma, we also have
\begin{equation}\label{4equ24}
\begin{aligned}
\liminf_{\epsilon \to 0} \int_0^t \int_\Omega d(x)^{-\gamma}\frac{u^{\delta} (\tilde{u}+\epsilon)^{m+1}}{(u+\epsilon)^{m}} \, dx \, ds &\ge \int_0^t \int_\Omega d(x)^{-\gamma}u^{\delta - m } \tilde{u}^{m+1}\, dx \, ds, \\
\liminf_{\epsilon \to 0} \int_0^t \int_\Omega d(x)^{-\gamma}\frac{\tilde{u}^{\delta} (u+\epsilon)^{m+1}}{(\tilde{u}+\epsilon)^{m}} \, dx \, ds &\ge \int_0^t \int_\Omega d(x)^{-\gamma} \tilde{u}^{\delta - m} u^{m+1}\, dx \, ds.
\end{aligned}
\end{equation}
Combining \eqref{4equ24} and the above estimates, we deduce from \eqref{4equ29} that
\[
\frac{1}{2} \int_0^t \int_\Omega \partial_t(u^{m+1} - \tilde{u}^{m+1})^2 \, dx \, ds 
\le \int_0^t \int_\Omega (g_{1}-g_{2})(u^{m+1}-\tilde{u}^{m+1}) \, dx \, ds.
\]
Using again H\"{o}lder's inequality and Gr\"{o}nwall's Lemma \cite[Lemma A.5]{brezis1973ope}, we conclude that for any $t \in [0,T]$,
\[
\|u^{m+1}(t)-\tilde{u}^{m+1}(t)\|_{L^2(\Omega)} 
\le \|u^{m+1}(0)-\tilde{u}^{m+1}(0)\|_{L^2(\Omega)} + \int_0^t \|g_{1}(s)-g_{2}(s)\|_{L^2(\Omega)} \, ds,
\]
which establishes the uniqueness of weak solutions in the case $ m > 0 $.\\[4pt]
\textbf{Step~3: Regularity of weak solutions.}  
We show that $u \in C([0,T]; W_{0}^{1,p}(\Omega))$ and then verify~\eqref{4equ31}, which completes the proof of this step.  
The argument closely follows that of \cite[Theorem~3.3, Step~4]{giacomoni2018quasilinear} (see also \cite{arora2020doubly}). We first consider the case $m = 0$.   Since $ u \in L^{\infty}(0,T; W_{0}^{1,p}(\Omega)) \cap L^{\infty}(Q_T) $ and $ \partial_t u \in L^{2}(Q_T), $ the Aubin--Simon compactness lemma (see e.g. \cite[Theorem~II.5.16]{boyer2012mathematical})  yields $ u \in C([0,T]; L^{r}(\Omega)) $  for all  $ r \geq 1. $ We now turn to the case $m > 0$.   In this situation, we have $ u \in L^{\infty}(0,T; W_{0}^{1,p}(\Omega)) \cap C([0,T]; L^{p}(\Omega)). $ Furthermore, in both of the above cases, by the Sobolev embedding theorems (Theorems~\ref{thm0} and~\ref{thm3}), the spaces
$W_{0}^{1,p}(\Omega)$ and $W_{0}^{s,p}(\Omega)$ are compactly embedded into $L^{p}(\Omega)$. As a consequence, the mapping $ u : [0,T] \longrightarrow W_{0}^{1,p}(\Omega) $ is weakly continuous. Therefore, for any $t_{0} \in [0,T]$, we have $u(\cdot,t_{0}) \in W_{0}^{1,p}(\Omega)$ and
\begin{equation}\label{4equ28}
\|u(t_{0})\|_{W_{0}^{1,p}(\Omega)}
\leq \liminf_{t \to t_{0}} \|u(t)\|_{W_{0}^{1,p}(\Omega)}.
\end{equation}
Moreover,
\begin{equation}\label{4equ32}
\|u(t_{0})\|_{W_{0}^{s,p}(\Omega)}
\leq \liminf_{t \to t_{0}} \|u(t)\|_{W_{0}^{s,p}(\Omega)}.
\end{equation}
On the other hand, in the case $m = 0$, proceeding as in \textbf{Step~1}, we select the test function $u_n - u_{n-1}$ in the weak formulation~\eqref{4equ1} satisfied by $u_n$. Summing the resulting identity over $n$ from $n_{0}'$ to $n_{0}''< n_{0}$, we obtain
\begin{equation}\label{4equ25}
\begin{aligned}
&\Delta_t \sum_{n=n_{0}'}^{n_{0}''}
\int_\Omega \left(\frac{u_n-u_{n-1}}{\Delta_t}\right)^2 \,dx \\
&\quad
+ \frac{1}{p}\Big(
\|u_{n_{0}''}\|_{W^{1,p}_0(\Omega)}^{p}
- \|u_{n_{0}'}\|_{W^{1,p}_0(\Omega)}^{p}
\Big)
+ \frac{1}{p}\Big(
\|u_{n_{0}''}\|_{W^{s,p}_0(\Omega)}^{p}
- \|u_{n_{0}'}\|_{W^{s,p}_0(\Omega)}^{p}
\Big) \\
&\qquad
\leq \Delta_t \sum_{n=n_{0}'}^{n_{0}''}
\int_{\Omega} g_{\Delta_t}
\left(\frac{u_n-u_{n-1}}{\Delta_t}\right)\,dx.
\end{aligned}
\end{equation}
Let $t \in [t_{0},T]$ be arbitrary, and choose integers $n_{0}'$ and $n_{0}''$ such that $ t_{n_{0}'} = n_{0}'\Delta_t \to t_{0} $ and $ t_{n_{0}''} = n_{0}''\Delta_t \to t, $ as $ \Delta_t \to 0. $ Passing to the limit in~\eqref{4equ25}, we infer that
\begin{align}\label{4equ30}
&\int_{t_{0}}^{t}\int_{\Omega} \left(\partial_{t}u\right)^{2}\,dx\,ds
+ \frac{1}{p}\Big(
\|u(t)\|_{W^{1,p}_0(\Omega)}^{p}
+ \|u(t)\|_{W^{s,p}_0(\Omega)}^{p}
\Big) \\
&\qquad
\leq \int_{t_{0}}^{t}\int_{\Omega} g\,\partial_{t}u \,dx\,ds
+ \frac{1}{p}\Big(
\|u(t_{0})\|_{W^{1,p}_0(\Omega)}^{p}
+ \|u(t_{0})\|_{W^{s,p}_0(\Omega)}^{p}
\Big).
\end{align}
From the above inequality, by letting $t \to t_{0}^{+}$ and using \eqref{4equ32}, we obtain
\begin{equation}\label{4equ26}
\begin{aligned}
\lim_{t \to t_{0}^{+}} \|u(t)\|_{W^{1,p}_0(\Omega)}^{p}
+ \|u(t_{0})\|_{W^{s,p}_0(\Omega)}^{p}
&\leq \lim_{t \to t_{0}^{+}}
\left(
\|u(t)\|_{W^{1,p}_0(\Omega)}^{p}
+ \|u(t)\|_{W^{s,p}_0(\Omega)}^{p}
\right) \\
&\leq
\|u(t_{0})\|_{W^{1,p}_0(\Omega)}^{p}
+ \|u(t_{0})\|_{W^{s,p}_0(\Omega)}^{p}.
\end{aligned}
\end{equation}
Consequently, in view of \eqref{4equ28}, we infer that
\begin{equation}\label{4equ33}
\|u(t_{0})\|_{W^{1,p}_0(\Omega)}
= \lim_{t \to t_{0}^{+}} \|u(t)\|_{W^{1,p}_0(\Omega)}.
\end{equation}
Similarly,
\begin{equation}\label{4equ34}
\|u(t_{0})\|_{W^{s,p}_0(\Omega)}
= \lim_{t \to t_{0}^{+}} \|u(t)\|_{W^{s,p}_0(\Omega)}.
\end{equation}
Furthermore, $u(t)\to u(t_{0})$ in $W^{1,p}_0(\Omega)$ as $t\to t_{0}^{+}$, which implies that $u$ is right--continuous on $[0,T]$. Now, we turn to the proof of left continuity. Let $0<\eta\leq t-t_{0}$. Multiplying \eqref{PP1} (recall that $m=0$) by $ \tau_{\eta}(u)
:= \frac{1}{\eta} (u(x,\cdot+\eta) - u(x,\cdot))$ and integrating over $(t_{0},t) \times \mathbb{R}^{N}$, we obtain, by exploiting the convexity of the functional $ u \longmapsto \frac{1}{p}\Big(
\|u\|_{W^{1,p}_0(\Omega)}^{p}
+ \|u\|_{W^{s,p}_0(\Omega)}^{p}
\Big), $ the inequality
\begin{equation*}
\begin{aligned}
&\int_{t_{0}}^{t} \int_{\Omega} \partial_{t}u\, \tau_{\eta}(u)\,dx\,ds
+ \frac{1}{p\,\eta} \int_{t_{0}}^{t}
\left(
\|u(s+\eta)\|_{W^{1,p}_0(\Omega)}^{p}
- \|u(s)\|_{W^{1,p}_0(\Omega)}^{p}
\right)\,ds \\
&\quad
+ \frac{1}{p\,\eta} \int_{t_{0}}^{t}
\left(
\|u(s+\eta)\|_{W^{s,p}_0(\Omega)}^{p}
- \|u(s)\|_{W^{s,p}_0(\Omega)}^{p}
\right)\,ds
\geq
\int_{t_{0}}^{t} \int_{\Omega} g(x, s)\, \tau_{\eta}(u)\,dx\,ds.
\end{aligned}
\end{equation*}
This implies
\begin{equation}\label{4equ27}
\begin{aligned}
&\int_{t_{0}}^{t} \int_{\Omega} \partial_{t}u\, \tau_{\eta}(u)\,dx\,ds
+ \frac{1}{p\,\eta} \int_{t}^{t+\eta}
\left(
\|u(s)\|_{W^{1,p}_0(\Omega)}^{p}
+ \|u(s)\|_{W^{s,p}_0(\Omega)}^{p}
\right)\,ds \\
&\quad
- \frac{1}{p\,\eta} \int_{t_{0}}^{t_{0}+\eta}
\left(
\|u(s)\|_{W^{1,p}_0(\Omega)}^{p}
+ \|u(s)\|_{W^{s,p}_0(\Omega)}^{p}
\right)\,ds
\geq
\int_{t_{0}}^{t} \int_{\Omega} g(x, s)\, \tau_{\eta}(u)\,dx\,ds.
\end{aligned}
\end{equation}
By virtue of \eqref{4equ33} and \eqref{4equ34}, we have
\[
\lim_{\eta \to 0}
\frac{1}{\eta}
\int_{t}^{t+\eta}
\left(
\|u(s)\|_{W^{1,p}_{0}(\Omega)}^{p}
+ \|u(s)\|_{W^{s,p}_{0}(\Omega)}^{p}
\right)\,ds
=
\|u(t)\|_{W^{1,p}_{0}(\Omega)}^{p}
+ \|u(t)\|_{W^{s,p}_{0}(\Omega)}^{p},
\]
and
\[
\lim_{\eta \to 0}
\frac{1}{\eta}
\int_{t_{0}}^{t_{0}+\eta}
\left(
\|u(s)\|_{W^{1,p}_{0}(\Omega)}^{p}
+ \|u(s)\|_{W^{s,p}_{0}(\Omega)}^{p}
\right)\,ds
=
\|u(t_{0})\|_{W^{1,p}_{0}(\Omega)}^{p}
+ \|u(t_{0})\|_{W^{s,p}_{0}(\Omega)}^{p}.
\]
Passing to the limit as $\eta \to 0$ in \eqref{4equ27}, we obtain
\begin{align*}
&\int_{t_{0}}^{t}\int_{\Omega} \left(\partial_{t}u\right)^{2}\,dx\,ds
+ \frac{1}{p}\Big(
\|u(t)\|_{W^{1,p}_0(\Omega)}^{p}
+ \|u(t)\|_{W^{s,p}_0(\Omega)}^{p}
\Big) \\
&\qquad
\geq
\int_{t_{0}}^{t}\int_{\Omega} g(x, s)\,\partial_{t}u \,dx\,ds
+ \frac{1}{p}\Big(
\|u(t_{0})\|_{W^{1,p}_0(\Omega)}^{p}
+ \|u(t_{0})\|_{W^{s,p}_0(\Omega)}^{p}
\Big).
\end{align*}
In view of \eqref{4equ26}, the above inequality must in fact be an equality. Hence, we infer that $ u \in C([0,T]; W^{1,p}_{0}(\Omega)). $ Moreover, identity~\eqref{4equ31} follows immediately by taking $t_{0}=0$ in the above equality. We now consider the case $m>0$. In this setting, we take as a test function
\[
\frac{(u_n+\varepsilon)^{m+1}-(u_{n-1}+\varepsilon)^{m+1}}{(u_n+\varepsilon)^m}
\in W^{1,p}_0(\Omega)\cap L^{\infty}(\Omega)
\]
in the weak formulation~\eqref{4equ1} satisfied by $u_n$. Arguing as in \textbf{Step~1} and summing the resulting identity with respect to $n$ from $n_{0}'$ up to $n_{0}''<n_0$, we obtain
\begin{equation}\label{4equ35}
\begin{aligned}
&\Delta_t \sum_{n=n_{0}'}^{n_{0}''}
\int_\Omega \left(\frac{u^{m+1}_n-u^{m+1}_{n-1}}{\Delta_t}\right)^2 \,dx \\
&\quad
+ \frac{m + 1}{p}\Big(
\|u_{n_{0}''}\|_{W^{1,p}_0(\Omega)}^{p}
- \|u_{n_{0}'}\|_{W^{1,p}_0(\Omega)}^{p}
\Big)
+ \frac{m+1}{p}\Big(
\|u_{n_{0}''}\|_{W^{s,p}_0(\Omega)}^{p}
- \|u_{n_{0}'}\|_{W^{s,p}_0(\Omega)}^{p}
\Big) \\
&\qquad
\leq
\Delta_t \sum_{n=n_{0}'}^{n_{0}''}
\int_{\Omega}
g_{\Delta_t}
\left(\frac{u^{m+1}_n-u^{m+1}_{n-1}}{\Delta_t}\right)\,dx \\
&\qquad\quad
+ m\,\Delta_t \sum_{n=n_{0}'}^{n_{0}''}
\int_{\Omega}
d(x)^{-\gamma} u_n^{\delta-m}
\left(\frac{u^{m+1}_n-u^{m+1}_{n-1}}{\Delta_t}\right)\,dx .
\end{aligned}
\end{equation}
Let $t\in[t_{0},T]$ be arbitrary, and choose integers $n_{0}'$ and $n_{0}''$ such that $ t_{n_{0}'} = n_{0}'\Delta_t \to t_{0} $ and
$ t_{n_{0}''} = n_{0}''\Delta_t \to t, $ as  $ \Delta_t \to 0. $
Passing to the limit in~\eqref{4equ35}, we infer
\begin{equation}\label{4equ36}
\begin{aligned}
&\int_{t_{0}}^{t}\int_{\Omega}
\left(\partial_{t}u^{m+1}\right)^{2}\,dx\,ds
+ \frac{m+1}{p}\Big(
\|u(t)\|_{W^{1,p}_0(\Omega)}^{p}
+ \|u(t)\|_{W^{s,p}_0(\Omega)}^{p}
\Big) \\
&\qquad
\leq
\int_{t_{0}}^{t}\int_{\Omega}
g(x, s)\,\partial_{t}\big(u^{m+1}\big)\,dx\,ds
+ m \int_{t_{0}}^{t}\int_{\Omega}
d(x)^{-\gamma} u^{\delta-m}\,
\partial_{t}\big(u^{m+1}\big)\,dx\,ds \\
&\qquad\quad
+ \frac{m+1}{p}\Big(
\|u(t_{0})\|_{W^{1,p}_0(\Omega)}^{p}
+ \|u(t_{0})\|_{W^{s,p}_0(\Omega)}^{p}
\Big).
\end{aligned}
\end{equation}
Letting $t\to t_{0}^{+}$ in the above inequality and using~\eqref{4equ28} and~\eqref{4equ32}, we deduce that
\begin{equation}\label{4equ37}
\|u(t_{0})\|_{W^{1,p}_0(\Omega)}
=
\lim_{t \to t_{0}^{+}} \|u(t)\|_{W^{1,p}_0(\Omega)},
\end{equation}
and, analogously,
\begin{equation}\label{4equ38}
\|u(t_{0})\|_{W^{s,p}_0(\Omega)}
=
\lim_{t \to t_{0}^{+}} \|u(t)\|_{W^{s,p}_0(\Omega)}.
\end{equation}
Consequently, the mapping $ u : [0,T] \longrightarrow W^{1,p}_0(\Omega) $ is \textbf{right--continuous}. We now prove the left continuity. Let $0<\eta \leq t-t_{0}$ and multiply \eqref{PP1} by
\[
\tau_{\eta,\varepsilon}(u)
:= \frac{\big(u(x,\cdot+\eta)+\varepsilon\big)^{m+1}
	- \big(u(x,\cdot)+\varepsilon\big)^{m+1}}
{\eta\,(u+\varepsilon)^{m}}
\in L^{2}(Q_{T}) \cap L^{1}(0,T; W^{1,p}_{0}(\Omega)).
\]
Integrating over $(t_{0},t)\times \mathbb{R}^{N}$, and using the assumption $m<p - 1$, together with the Picone--type inequalities in both the local and nonlocal settings
\cite[Propositions~2.6 and~4.2]{brasco2014convexity} and the Dominated Convergence Theorem, we obtain
\begin{equation*}
\begin{aligned}
&\int_{t_{0}}^{t}\int_{\Omega}
u^{m}\partial_{t}u\,\tau_{\eta,0}(u)\,dx\,ds
+ \frac{m+1}{p\,\eta}\int_{t_{0}}^{t}
\left(
\|u(s+\eta)\|_{W^{1,p}_{0}(\Omega)}^{p}
- \|u(s)\|_{W^{1,p}_{0}(\Omega)}^{p}
\right)\,ds \\
&\quad
+ \frac{m+1}{p\,\eta}\int_{t_{0}}^{t}
\left(
\|u(s+\eta)\|_{W^{s,p}_{0}(\Omega)}^{p}
- \|u(s)\|_{W^{s,p}_{0}(\Omega)}^{p}
\right)\,ds \\
&\geq
\int_{t_{0}}^{t}\int_{\Omega}
g(x, s)\,u^{m}\tau_{\eta,0}(u)\,dx\,ds
+ m\int_{t_{0}}^{t}\int_{\Omega}
d(x)^{-\gamma}u^{\delta}\tau_{\eta,0}(u)\,dx\,ds .
\end{aligned}
\end{equation*}
Consequently, we infer
\begin{equation}\label{4equ45}
\begin{aligned}
&\int_{t_{0}}^{t}\int_{\Omega}
\partial_{t}u\,\tau_{\eta,0}(u)\,dx\,ds
+ \frac{m+1}{p\,\eta}\int_{t}^{t+\eta}
\left(
\|u(s)\|_{W^{1,p}_{0}(\Omega)}^{p}
+ \|u(s)\|_{W^{s,p}_{0}(\Omega)}^{p}
\right)\,ds \\
&\quad
- \frac{m+1}{p\,\eta}\int_{t_{0}}^{t_{0}+\eta}
\left(
\|u(s)\|_{W^{1,p}_{0}(\Omega)}^{p}
+ \|u(s)\|_{W^{s,p}_{0}(\Omega)}^{p}
\right)\,ds \\
&\geq
\int_{t_{0}}^{t}\int_{\Omega}
g(x, s)\,u^{m}\tau_{\eta,0}(u)\,dx\,ds
+ m\int_{t_{0}}^{t}\int_{\Omega}
d(x)^{-\gamma}u^{\delta}\tau_{\eta,0}(u)\,dx\,ds .
\end{aligned}
\end{equation}
By \eqref{4equ37} and \eqref{4equ38}, we have
\[
\lim_{\eta\to0^{+}}
\frac{1}{\eta}\int_{t}^{t+\eta}
\left(
\|u(s)\|_{W^{1,p}_{0}(\Omega)}^{p}
+ \|u(s)\|_{W^{s,p}_{0}(\Omega)}^{p}
\right)\,ds
=
\|u(t)\|_{W^{1,p}_{0}(\Omega)}^{p}
+ \|u(t)\|_{W^{s,p}_{0}(\Omega)}^{p},
\]
\[
\lim_{\eta\to0^{+}}
\frac{1}{\eta}\int_{t_{0}}^{t_{0}+\eta}
\left(
\|u(s)\|_{W^{1,p}_{0}(\Omega)}^{p}
+ \|u(s)\|_{W^{s,p}_{0}(\Omega)}^{p}
\right)\,ds
=
\|u(t_{0})\|_{W^{1,p}_{0}(\Omega)}^{p}
+ \|u(t_{0})\|_{W^{s,p}_{0}(\Omega)}^{p}.
\]
Letting $\eta\to0^{+}$ in \eqref{4equ45}, we deduce
\begin{equation*}
\begin{aligned}
&\int_{t_{0}}^{t}\int_{\Omega}
\left(\partial_{t}u^{m+1}\right)^{2}\,dx\,ds
+ \frac{m+1}{p}\Big(
\|u(t)\|_{W^{1,p}_{0}(\Omega)}^{p}
+ \|u(t)\|_{W^{s,p}_{0}(\Omega)}^{p}
\Big) \\
&\qquad
\geq
\int_{t_{0}}^{t}\int_{\Omega}
g(x, t)\,\partial_{t}\big(u^{m+1}\big)\,dx\,ds
+ m\int_{t_{0}}^{t}\int_{\Omega}
d(x)^{-\gamma}u^{\delta-m}\,
\partial_{t}\big(u^{m+1}\big)\,dx\,ds \\
&\qquad\quad
+ \frac{m+1}{p}\Big(
\|u(t_{0})\|_{W^{1,p}_{0}(\Omega)}^{p}
+ \|u(t_{0})\|_{W^{s,p}_{0}(\Omega)}^{p}
\Big).
\end{aligned}
\end{equation*}
In view of \eqref{4equ36}, the above inequality is in fact an equality. Consequently, we get that $ u \in C([0,T]; W^{1,p}_{0}(\Omega)). $ Moreover, identity~\eqref{4equ31} follows by taking $t_{0}=0$ in the above equality.
\end{proof}

\begin{proof}[\textbf{Proof of Theorem}~\ref{thm6}]
Let $m>0$ and let $u$ be the weak solution of problem~\eqref{PP1}, in the sense of Definition~\ref{definition1}, obtained in Theorem~\ref{2thm1}.  
We define $v:=u^{m+1}$.  
By testing the weak formulation~\eqref{weakform} with $\varphi=\Psi v^{-m}$, where
\[
\Psi \in L^{\infty}\bigl(0,T;L^{\infty}_{d^{m+1}}(\Omega)\bigr),
\qquad
|\nabla \Psi|d(\cdot)^{-m}\in L^{1}\bigl(0,T;L^{p}(\Omega)\bigr),
\]
it is readily seen that $v$ is a weak solution of problem~\eqref{0equ4} in the following sense:
\begin{equation}\label{4equ40}
\begin{aligned}
&\int_0^{t}\!\!\int_{\Omega} \partial_t v\,\Psi \,dx\,ds
+ \int_0^{t}\!\!\int_{\Omega}
\left|\nabla\!\left(v^{\frac{1}{m+1}}\right)\right|^{p-1}
\nabla\!\left(v^{\frac{1}{m+1}}\right)
\cdot
\nabla\!\left(\frac{\Psi}{v^{\frac{m}{m+1}}}\right)
\,dx\,ds \\[1ex]
&+ \int_0^{t}\!\!\iint_{\mathbb{R}^{2N}}
\frac{
	\left[ v^{\frac{1}{m+1}}(x)-v^{\frac{1}{m+1}}(y)\right] ^{p-1}
}
{|x-y|^{N+sp}} 
\left[
\left(\frac{\Psi}{v^{\frac{m}{m+1}}}\right)(x)
-
\left(\frac{\Psi}{v^{\frac{m}{m+1}}}\right)(y)
\right]
\,dx\,dy\,ds \\[1ex]
&= \int_0^{t}\!\!\int_{\Omega} g(x,s)\,\Psi \,dx\,ds
+ m\int_0^{t}\!\!\int_{\Omega}
d(x)^{-\gamma}
v^{\frac{\delta-m}{m+1}}\,\Psi
\,dx\,ds .
\end{aligned}
\end{equation}
	Moreover, by the regularity properties of the solution $u$, it follows that for any $1\le r<\infty$, $ v\in C\bigl([0,T];L^{r}(\Omega)\bigr), $ and there exists a constant $C>0$ such that, for every $t\in[0,T]$,
	\[
	C^{-1}d(x)\le v^{1/m+1}(t,x)\le C\,d(x)^{\alpha}
	\quad \text{a.e. in }\Omega,\qquad \forall\,\alpha\in[s,1).
	\]
Now, let $v_{1},v_{2}$ be two solutions of problem~\eqref{0equ4}.  
Define  $ u_{i}:=v_{i}^{1/m+1}, \,i=1,2, $ and choose  $ \Psi_{i}:=\varphi_{i}v_{i}^{m}, \,
\varphi_{i}\in L^{\infty}\bigl(0,T;L^{\infty}_{d}(\Omega)\bigr) \cap L^{1}\bigl(0,T;W_{0}^{1,p}(\Omega)\bigr), \, i=1,2. $ Substituting $\Psi_{1}$ and $\Psi_{2}$ into~\eqref{4equ40}, we deduce that $u_{1}$ and $u_{2}$ satisfy the weak formulation~\eqref{weakform} with test functions  $ \varphi_{1}, \varphi_{2} \in L^{\infty}\bigl(0,T;L^{\infty}_{d}(\Omega)\bigr) \cap L^{1}\bigl(0,T;W_{0}^{1,p}(\Omega)\bigr). $ Since $u_{1}$ and $u_{2}$ also verify~\eqref{4equ41}, for any $\varepsilon\in(0,1)$ we define the test functions
\[
 \varphi_{1}:=\frac{(u_{1}+\varepsilon)^{m+1}-(u_{2}+\varepsilon)^{m+1}}{(u_{1}+\varepsilon)^{m}}, 
\qquad
 \varphi_{2}:=\frac{(u_{2}+\varepsilon)^{m+1}-(u_{1}+\varepsilon)^{m+1}}{(u_{2}+\varepsilon)^{m}},
\]
which belong to $ L^{\infty}\bigl(0,T;L^{\infty}_{d}(\Omega)\bigr) \cap L^{1}\bigl(0,T;W_{0}^{1,p}(\Omega)\bigr). $ Proceeding as in \textbf{Step~2} of the proof of Theorem~\ref{2thm1}, we derive~\eqref{equ34}, which establishes the desired uniqueness result $ v_{1}=v_{2}. $
\end{proof}
\begin{proof}[\textbf{Proof of Theorem}~\ref{thereom1}]
	Let $u_0 \in \overline{D^{\star}(\mathcal{T}_0^{\star})}^{\,L^{\infty}(\Omega)}$, let $\lambda>0$, and let $g_{1},g_{2}\in L^{\infty}(\Omega)$.  
	By Theorem~\ref{theorem3}, there exist unique solutions $ 	u_{1},u_{2}\in W^{1,p}_{0}(\Omega)\cap C^{1,\xi}(\overline{\Omega})\cap \mathcal{M}^{1}_{1,\alpha}(\Omega), $ for some $\xi\in(0,1)$ and for every $\alpha\in[s,1)$ with $\alpha\neq \frac{ps}{p-1}$, to the following problems, respectively:
	\begin{equation}\label{4equ123}
	\begin{aligned}
	u_{1}+\lambda \mathcal{T}_0^{\star}u_{1}&=g_{1},\\
	u_{2}+\lambda \mathcal{T}_0^{\star}u_{2}&=g_{2},
	\end{aligned}
	\end{equation}
	Clearly, $u_{1},u_{2}\in D^{\star}(\mathcal{T}_0^{\star})$. Set
	\[
	w:=\big(u_{1}-u_{2}-\|g_{1}-g_{2}\|_{L^{\infty}(\Omega)}\big)^{+}.
	\]
	Using $w$ as a test function in the weak formulation of \eqref{4equ123}, we obtain
	\begin{align*}
	&\int_{\Omega} w^{2}\,dx
	+ \lambda\int_{\Omega}
	\big(|\nabla u_{1}|^{p-2}\nabla u_{1}-|\nabla u_{2}|^{p-2}\nabla u_{2}\big)\cdot \nabla w\,dx \\
	&\quad + \lambda \iint_{\mathbb{R}^{2N}}
	\frac{\Big(\,[u_{1}(x)-u_{1}(y)]^{p-1}-[u_{2}(x)-u_{2}(y)]^{p-1}\Big)\big(w(x)-w(y)\big)}
	{|x-y|^{N+sp}}\,dx\,dy
	\le 0 .
	\end{align*}
	By applying the elementary inequalities established in~\cite{farina2013monotonicity} (see also~\cite[Section~10]{lindqvist2017notes}) for the local term, together with the computations in the proof of Lemma~3.5 in~\cite{giacomoni2018existence} for the nonlocal term, we deduce that
	\[
	u_{1}-u_{2}\le \|g_{1}-g_{2}\|_{L^{\infty}(\Omega)}.
	\]
	By reversing the roles of $u_{1}$ and $u_{2}$, we obtain
	\[
	\|u_{1}-u_{2}\|_{L^{\infty}(\Omega)}
	\le \|g_{1}-g_{2}\|_{L^{\infty}(\Omega)}.
	\]
	This shows that the operator $\mathcal{T}_0^{\star}$ is maximal accretive in $L^{\infty}(\Omega)$. Assertion~(ii) of Theorem~\ref{thereom1} then follows immediately from~\cite[Theorem~4.2, p.~130]{barbu1976nonlinear}.  
	Assertion~(iii) can be proved by suitably adapting the arguments used in the proof of~\cite[Theorem~4.4, p.~141]{barbu1976nonlinear}; see also~\cite[Proof of Proposition~0.1]{badra2012singular}.
\end{proof}

\subsection{Asymptotic Analysis of Problem~\eqref{P1}}

\begin{proof}[\textbf{Proof of Theorem}~\ref{the3}]
We follow the approach presented in the proof of \cite[Theorem 3.10]{giacomoni2018existence}, suitably adapted to accommodate both the local and nonlocal operators. In particular, we consider two distinct cases.\\
\textbf{Case 1:} $g = g_{\infty}$.  
Let $m \ge 0$ and consider the family $\{S(t) : t \ge 0\}$ defined as 
\[ 
\begin{cases} 
u(t) = S(t)u_0, & m = 0,\ u_0 \in W^{1,p}_0(\Omega) \cap \mathcal{M}^{1}_{\alpha',\alpha'}(\Omega),\\[1mm]
 v(t) = S(t)v_0, & m > 0,\ v_0 \in \dot{\mathbf{V}}_{+}^{\, m+1} \cap \mathcal{M}^{1/(m+1)}_{\alpha',\alpha'}(\Omega), 
 \end{cases}
  \]
   where $\alpha' \in (0,s)$ is defined by \eqref{1equ2}, and $w := u$ (resp.\ $v$) denotes the unique solution given by Theorem~\ref{2thm1} (resp.\ Theorem~\ref{thm6}).Then, the family $\{S(t)\}_{t \ge 0}$ generates a contraction semigroup on $L^2(\Omega)$. Indeed, by the uniqueness and the aforementioned properties of the solution $w$, it follows that for all $t,s \ge 0$,
\begin{equation}\label{equ49_combined}
S(t+s)w_0 = S(t)S(s)w_0, \qquad S(0)w_0 = w_0,
\end{equation}
for any initial datum $w_0 = u_0$ or $v_0$.  Moreover, owing to the regularity property $w \in C([0,T];L^r(\Omega))$ for every $1 \le r < \infty$, the mapping $t \mapsto S(t)w_0$ is strongly continuous in $L^2(\Omega)$. In addition, the operator $S(t)$ is $T$-accretive in $L^2(\Omega)$.  In fact, it suffices to establish the case $m > 0$, since the case $m = 0$ can be proved analogously.   For $m > 0$, let $u$ denote the solution of \eqref{PP1} corresponding to $g = g_{\infty}$ with initial data $u_0$. Hence, we have 
\[
v(t) = u(t)^{m+1} = S(t)v_0, \quad \text{with } v_0 = u_0^{m+1}.
\]  
Let $\underline{u} := u_k$ be the solution of problem~\eqref{4equ6}, and let $\overline{u} := \tilde{u}_M$ denote the solution of problem~\eqref{4equ100}. Then, both $\underline{u}$ and $\overline{u}$ belong to $\mathcal{M}^{1}_{\alpha',\alpha'}(\Omega)$. Moreover, for $k>0$ sufficiently small and $K>0$ sufficiently large, $\underline{u}$ and $\overline{u}$ are, respectively, a sub-solution and a super-solution to problem~\eqref{P7} with $b=g_{\infty}$, satisfying $ \underline{u} \leq u_0 \leq \overline{u}. $ We then define
\[
\underline{v}(t) := S(t)\,\underline{u}^{\,m+1}, 
\qquad 
\overline{v}(t) := S(t)\,\overline{u}^{\,m+1},
\]
which correspond to the solutions of problem~\eqref{0equ4}. Consequently, the functions $\underline{v}=\underline{u}^{\,m+1}$ and $\overline{v}=\overline{u}^{\,m+1}$ are generated through the iterative scheme~\eqref{4equ1} with initial data $u_0=\underline{u}$ and $u_0=\overline{u}$, respectively.  By virtue of the comparison principle, the mappings $t \mapsto \underline{v}(t)$ and $t \mapsto \overline{v}(t)$ are nondecreasing and nonincreasing, respectively.  On the other hand, \eqref{equ34} ensures that for any $t \geq 0$,
\begin{equation}\label{equ50}
\underline{u} \leq \underline{v}(t) \leq v(t) \leq \overline{v}(t) \leq \overline{u}.
\end{equation}  
We set 
\[
\underline{v}_{\infty} = \lim_{t \to \infty} \underline{v}(t), \qquad \overline{v}_{\infty} = \lim_{t \to \infty} \overline{v}(t).
\]  
It then follows from \eqref{equ49_combined} that
\begin{align*}
\underline{v}_{\infty} &= \lim_{s \to \infty} S(t+s) \underline{u}^{m+1} 
= S(t) \lim_{s \to \infty} \bigl(S(s) (\underline{u}^{m+1})\bigr) = S(t) \underline{v}_{\infty}, \\
\overline{v}_{\infty} &= \lim_{s \to \infty} S(t+s) \overline{u}^{m+1} 
= S(t) \lim_{s \to \infty} \bigl(S(s) (\overline{u}^{m+1})\bigr) = S(t) \overline{v}_{\infty}.
\end{align*}  
This shows that $\underline{v}_{\infty}$ and $\overline{v}_{\infty}$ are stationary solutions of \eqref{equu22} with $b = g_{\infty}$. By uniqueness, we deduce that
\[
v_{\mathrm{stat}} := \underline{v}_{\infty} = \overline{v}_{\infty},
\]  
where $v_{\mathrm{stat}}$ denotes the stationary solution of \eqref{0equ4}. Consequently, using \eqref{equ50} together with the Dominated Convergence Theorem, we obtain
\[
\| v(t) - v_{\mathrm{stat}} \|_{L^2(\Omega)} \to 0 \quad \text{as } t \to \infty.
\]  
Furthermore, applying \eqref{equ50} and the standard interpolation inequality, valid for $2 < r < \infty$,
\[
\| \cdot \|_{L^r(\Omega)} \leq \| \cdot \|_{L^{\infty}(\Omega)}^{\alpha} \| \cdot \|_{L^2(\Omega)}^{1-\alpha},
\]  
we conclude that the convergence extends to all $r \geq 1$.\\[4pt]
\noindent\textbf{Case 2:} $g \not\equiv g_{\infty}$.   From \eqref{equ51},  there exists $t_0 > 0$ sufficiently large such that, for any $t \geq t_0$,
\[
\exp(t - t_{0})\| g(t, \cdot) - g_{\infty} \|_{L^2(\Omega)} \leq M \quad \text{for some } M>0.
\]  
Let $T > 0$ and let $u$ be the solution of problem \eqref{PP1} obtained by Theorem Theorem~\ref{2thm1}  with $g$ and initial data $u_0 = v_0^{1/m+1}$, and set $v = u^{m+1}$. Since $u$ satisfies \eqref{4equ41}, we define
\[
\tilde{v}(t) = S(t+t_0) v_0 = S(t) v(t_0).
\]  
Then, by \eqref{equ34} and the uniqueness argument, for any $t > 0$,
\begin{align*}
\| v(t+t_0, \cdot) - \tilde{v}(t, \cdot) \|_{L^2(\Omega)}
&\leq \int_0^{t} \| g(s+t_0, \cdot) - g_{\infty} \|_{L^2(\Omega)} ds \\
&\leq M \int_{t_0}^{+\infty} \exp(-t+ t_{0})ds \leq M .
\end{align*}  
By \textbf{Case 1}, we have $\tilde{v}(t) \to v_{\text{stat}}$ in $L^2(\Omega)$ as $t \to \infty$. Therefore, we deduce
\[
\| v(t) - v_{\text{stat}} \|_{L^2(\Omega)} \to 0 \quad \text{as } t \to \infty.
\]  
Using again the interpolation inequality above, the proof of Theorem \ref{the3} is complete.
\end{proof}

\section*{Statements and Declarations}
\subsection*{Ethics approval and consent to participate}
Not applicable.
\subsection*{Funding}
Not applicable
\subsection*{Availability of data and materials}
Not applicable
\subsection*{Conflict of interest}
The author declares that there is no conflict of interest.

\end{document}